\RequirePackage{rotating} 
\documentclass[reqno, 12pt]{amsart}

\makeatletter
\let\origsection=\section \def\section{\@ifstar{\origsection*}{\mysection}} 
\def\mysection{\@startsection{section}{1}\z@{.7\linespacing\@plus\linespacing}{.5\linespacing}{\normalfont\scshape\centering\S}}
\makeatother       

\usepackage{geometry}
\numberwithin{equation}{section}
\numberwithin{figure}{section}
\usepackage{ifpdf,ifxetex}

\usepackage{graphicx}
\usepackage{wrapfig}
\usepackage{pdflscape}
\makeatletter
\newenvironment{sidewaysfigurepage}{
    \clearpage
    \@rotfloat{figure}
    \ifthenelse{\isodd{\rot@LR}}{
       }{
       }
}{
    \end@rotfloat
    \clearpage
}
\makeatother

\usepackage{xcolor}
\usepackage{hyperref}
\hypersetup{
    unicode,
    colorlinks,
    linkcolor={red!60!black},
    citecolor={green!60!black},
    urlcolor={blue!60!black}
}

\usepackage{amsfonts,amsthm,amssymb,amsmath}
\usepackage[T1]{fontenc}
\usepackage{microtype}
\usepackage{enumitem}
\usepackage{tikz}
\usepackage{epsfig}
\usepackage{enumitem}
\setlist[enumerate]{label=(\arabic*), ref=(\arabic*)}
\usepackage{verbatim}
\usepackage{thmtools, thm-restate}

\newtheorem{theorem}{Theorem}[section]
\newtheorem{lemma}[theorem]{Lemma}

\newtheorem{corollary}[theorem]{Corollary}
\newtheorem{conjecture}[theorem]{Conjecture}

\newtheorem*{ubqconjecture}{The Ubiquity Conjecture}
\newtheorem{lemmadef}[theorem]{Lemma and Definition}

\theoremstyle{definition}
\newtheorem{definition}[theorem]{Definition}

\newtheorem{remark}[theorem]{Remark}
\newtheorem{question}[theorem]{Question}

\newtheorem{remarkdef}[theorem]{Remark and Definition}

\newcommand{\V}{\mathcal{V}}

\newcommand{\Nbb}{\mathbb{N}}
\newcommand{\Pcal}{{\mathcal P}}

\newcommand{\Rcal}{{\mathcal R}}
\newcommand{\Scal}{{\mathcal S}}
\newcommand{\Tcal}{{\mathcal T}}
\newcommand{\Fcal}{{\mathcal F}}
\newcommand{\Ecal}{{\mathcal E}}

\newcommand{\Hcal}{\mathcal{H}}

\newcommand\subsetsim{\mathrel{%
  \ooalign{\raise0.2ex\hbox{$\subset$}\cr\hidewidth\raise-0.8ex\hbox{\scalebox{0.9}{$\sim$}}\hidewidth\cr}}}

\renewcommand{\triangleleft}{\vartriangleleft}
\renewcommand{\leq}{\leqslant}
\renewcommand{\geq}{\geqslant}
\renewcommand{\preceq}{\preccurlyeq}

\newcommand{\Gtribe}{{$G$-\text{tribe}}}
\newcommand{\Gsubtribe}{{$G$-\text{subtribe}}}

\DeclareMathOperator{\init}{init}
\DeclareMathOperator{\core}{core}
\DeclareMathOperator{\dist}{dist}
\DeclareMathOperator{\RG}{RG}

\title[Ubiquity of graphs with extensive tree-decompositions]{Ubiquity in graphs III: Ubiquity of locally finite graphs with extensive tree-decompositions}

\author[Bowler, Elbracht, Erde, Gollin, Heuer, Pitz, Teegen]{Nathan Bowler \and Christian Elbracht \and Joshua Erde \and J.~Pascal Gollin \and Karl Heuer \and Max Pitz \and Maximilian Teegen}

\address[Bowler, Elbracht, Pitz, Teegen]{Universit\"{a}t Hamburg, Department of Mathematics, Bundesstra{\ss}e 55 (Geomatikum), 20146 Hamburg, Germany}

\address[Erde]{Graz University of Technology, Institute of Discrete Mathematics, Steyrergasse 30, 8010 Graz, Austria}

\address[Gollin]{Institute for Basic Science (IBS), Discrete Mathematics Group, 55, Expo-ro, Yuseong-gu, Daejeon, Republic of Korea, 
34126}

\address[Heuer]{Technische Universit\"{a}t Berlin,
Institut f\"{u}r Softwaretechnik und Theoretische Informatik,
Ernst-Reuter-Platz 7, 10587 Berlin, Germany}

\email{nathan.bowler@uni-hamburg.de}
\email{christian.elbracht@uni-hamburg.de}
\email{erde@math.tugraz.at}
\email{pascalgollin@ibs.re.kr}
\email{karl.heuer@tu-berlin.de}
\email{max.pitz@uni-hamburg.de}
\email{maximilian.teegen@uni-hamburg.de}

\thanks{The fourth author was supported by the Institute for Basic Science (IBS-R029-C1).}
\thanks{The fifth author was supported by the European Research Council (ERC) under the European Union's Horizon 2020 research and innovation programme (ERC consolidator grant DISTRUCT, agreement No.\ 648527).}

\begin{document}

\begin{abstract}
    A graph~$G$ is said to be \emph{ubiquitous}, 
    if every graph~$\Gamma$ that contains arbitrarily many disjoint $G$-minors automatically contains infinitely many disjoint $G$-minors.
    The well-known \emph{Ubiquity conjecture} of Andreae says that every locally finite graph is ubiquitous.
    
    In this paper we show that locally finite graphs admitting a certain type of tree-decomposition, which we call an \emph{extensive tree-decomposition}, are ubiquitous. 
    In particular this includes all locally finite graphs of finite tree-width, and also all locally finite graphs with finitely many ends, all of which have finite degree. It remains an open  question whether every locally finite graph admits  an extensive tree-decomposition.
\end{abstract}

\maketitle
\date{}

\section{Introduction}
Given a graph~$G$ and some relation~$\triangleleft$ between graphs, we say that~$G$ is \emph{$\triangleleft$-ubiquitous} 
if whenever~$\Gamma$ is a graph such that ${nG \triangleleft \Gamma}$ for all~${n \in \Nbb}$, then ${\aleph_0 G \triangleleft \Gamma}$, 
where ${\alpha G}$ is the disjoint union of~$\alpha$ many copies of~$G$. 
A classic result of Halin~\cite[Satz~1]{H65} says that the ray, i.e.~a one-way infinite path, is $\subseteq$-ubiquitous, where~$\subseteq$ is the subgraph relation. 
That is, any graph which contains arbitrarily large collections of vertex-disjoint rays must contain an infinite collection of vertex-disjoint rays. 
Later, Halin showed that the double ray, i.e.~a two-way infinite path, is also $\subseteq$-ubiquitous~\cite{H70}. 
However, not all graphs are $\subseteq$-ubiquitous, and in fact even trees can fail to be $\subseteq$-ubiquitous (see for example~\cite{W76}). 

The question of ubiquity for classes of graphs has also been considered for other graph relations. 
In particular, whilst there are still reasonably simple examples of graphs which are not $\leq$-ubiquitous (see~\cite{L76,A77}), where~$\leq$ is the topological minor relation, it was shown by Andreae that all rayless countable graphs~\cite{A80} and all locally finite trees~\cite{A79} are $\leq$-ubiquitous. 
The latter result was recently extended to the class of all trees by the present authors~\cite{BEEGHPTI}.

In~\cite{A02} Andreae initiated the study of ubiquity of graphs with respect to the minor relation~$\preceq$. 
He constructed a graph which is not $\preceq$-ubiquitous, however the construction relies on the existence of a counterexample to the well-quasi-ordering of infinite graphs under the minor relation, for which only examples of uncountable size are known~\cite{komjath1995note,Pitz2020,T88}. 
In particular, the question of whether there exists a countable graph which is not $\preceq$-ubiquitous remains open.

Andreae conjectured that at least all \emph{locally finite} graphs, those with all degrees finite, should be $\preceq$-ubiquitous.

\begin{ubqconjecture}
    Every locally finite connected graph is $\preceq$-ubiquitous.
\end{ubqconjecture}

In~\cite{A13} Andreae established the following pair of results, demonstrating that his conjecture holds for wide classes of locally finite graphs. Recall that a \emph{block} of a graph is a maximal $2$-connected subgraph, and that a graph has \emph{finite tree-width} if there is an integer~$k$ such that the graph has a tree-decomposition of width~$k$.

\begin{theorem}[Andreae, {\cite[Corollary 1]{A13}}]
    \label{t:And1}
    Let~$G$ be a locally finite, connected graph with finitely many ends such that every block of~$G$ is finite. 
    Then~$G$ is $\preceq$-ubiquitous.
\end{theorem}

\begin{theorem}[Andreae, {\cite[Corollary 2]{A13}}]
    \label{t:And2}
    Let~$G$ be a locally finite, connected graph of finite tree-width such that every block of~$G$ is finite. 
    Then~$G$ is $\preceq$-ubiquitous.
\end{theorem}

Note, in particular, that if~$G$ is such a graph, then the degree of every end in~$G$ must be one.\footnote{A precise definitions of the ends of a graph and their degree can be found in Section~\ref{s:prelim}.} 
The main result of this paper is a far-reaching extension of Andreae's results, removing the assumption of finite blocks. 

\begin{theorem}
    \label{c:locfin}
    Let~$G$ be a locally finite, connected graph with finitely many ends such that every end of~$G$ has finite degree. 
    Then~$G$ is $\preceq$-ubiquitous.
\end{theorem}

\begin{theorem}
    \label{c:finend}
    Every locally finite, connected graph of finite tree-width is $\preceq$-ubiquitous.
\end{theorem}

The reader may have noticed that these results are of a similar flavour: they all make an assertion that locally finite graphs which are built by pasting  finite graphs in a tree like fashion are ubiquitous -- with differing requirements on the size of the finite graphs, how far they are allowed to overlap, and the structure of the underlying decomposition trees. And indeed, behind all the above results there are unifying but more technical theorems,  the strongest of which is the true main result of this paper:

\begin{theorem}[Extensive tree-decompositions and  ubiquity]
\label{t:mainintro}
    Every locally finite connected graph admitting an extensive tree-decomposition is $\preceq$-ubiquitous.
\end{theorem}

The precise definition of an extensive tree-decomposition is somewhat involved and  will be given in detail in Section~\ref{s:extensive} up to Theorem~\ref{t:nice}. Roughly, however, it implies that we can find many self-minors of the graph at spots whose precise positions are governed by the decomposition tree. We hope that the proof sketch in Section~\ref{s:sketch} is a good source for additional intuition before the reader delves into the technical details.

To summarise, we are facing two main tasks in this paper. One is to prove our main ubiquity result, Theorem~\ref{t:mainintro}. This will occupy the second part of this paper, Sections~\ref{s:nonpebbly} to~\ref{sec:countable-subtrees}.
And as our other task, we also need to prove that the graphs in Theorems~\ref{c:locfin} and~\ref{c:finend} do indeed possess such extensive tree-decompositions.

This analysis occupies Section~\ref{s:extensive} and~\ref{s:getnice}. The proof uses in an essential way certain results about the well-quasi-ordering of graphs under the minor relation, including Thomas's result~\cite{T89} that for all $k \in \mathbb{N}$, the classes of graphs of tree-width at most $k$ are well-quasi-ordered under the minor relation. 
In fact, the class of locally finite graphs having an extensive tree-decomposition is certainly larger than the results stated in Theorems~\ref{c:locfin} and~\ref{c:finend}; for example, it is easy to see that the infinite grid  $\mathbb{N} \times  \mathbb{N}$  has  such an extensive tree-decomposition. It remains an open question whether \emph{every} locally finite graph has an extensive tree-decomposition. A more precise discussion of how this problem relates to the theory of well-quasi- and better-quasi-orderings  of finite graphs will be given in Section~\ref{s:WQO}.

But first, in Section~\ref{s:sketch} we will give a sketch of the key ideas in the proof, at the end of which we will provide a more detailed overview of the structure and the different sections of this paper.

\section{Proof sketch}
\label{s:sketch}

To give a flavour of the main ideas in this paper, let us begin by considering the case of a locally finite connected graph~$G$ with a single end~$\omega$, where~$\omega$ has finite degree~${d \in \Nbb}$ (this means that there is a family ${(A_i \colon 1 \leq i \leq d)}$ of~$d$ disjoint rays in~$\omega$, but no family of more than~$d$ such rays). 
Our construction will exploit the fact that graphs of this kind have a very particular structure. 
More precisely, there is a tree-decomposition ${(S, (V_s)_{s \in V(S)})}$ of~$G$, where ${S = s_0s_1s_2\ldots}$ is a ray and such that, if we denote~$V_{s_n}$ by~$V_n$ and~${G[\bigcup_{l \geq n} V_l]}$ by~$G_n$ for each~$n$, the following holds:
\begin{enumerate}
    \item \label{item:sketch-1} each~$V_n$ is finite;
    \item \label{item:sketch-2} 
        $|V_i \cap V_{j}|=d$ if  $|i-j|=1$, and $|V_i \cap V_{j}|=0$ otherwise;  
    \item \label{item:sketch-3} all the~$A_i$ begin in~$V_0$; 
    \item \label{item:sketch-4} for each~${m \geq 1}$ there are infinitely many~${n > m}$ such that~$G_m$ is a minor of~$G_n$, in such a way that for any edge~$e$ of~$G_m$ and any~${i \leq d}$, the edge~$e$ is contained in~$A_i$ if and only if the edge representing it in this minor is. 
\end{enumerate}

Property~\ref{item:sketch-4} seems rather strong -- it is a first glimpse of the strength of extensive tree-decompositions alluded to in Theorem~\ref{t:mainintro}. The reason it can always be achieved has to do with the well-quasi-ordering of finite graphs. For details of how this works, see Section~\ref{s:getnice}. 
The sceptical reader who does not yet see how to achieve this may consider the argument in this section as showing ubiquity simply for graphs~$G$ with a decomposition of the above kind. 

Now we suppose that we are given some graph~$\Gamma$ such that~${nG \preccurlyeq \Gamma}$ for each~$n$, and we wish to show that ${\aleph_0G \preccurlyeq \Gamma}$. 
Consider a $G$-minor~$H$ in~$\Gamma$. 
Any ray~$R$ of~$G$ can be expanded to a ray~${H(R)}$ in the copy~$H$ of~$G$ in~$\Gamma$, 
and since~$G$ only has one end, all rays~${H(R)}$ go to the same end~$\epsilon_H$ of~$\Gamma$; 
we shall say that~$H$ \emph{goes to} the end~$\epsilon_H$. 

Techniques from an earlier paper~\cite{BEEGHPTI} show that we may assume that there is some end $\epsilon$ of $\Gamma$ such that all $G$-minors in $\Gamma$ go to $\epsilon$, otherwise it can be shown that $\aleph_0 G \preceq \Gamma$.

From any $G$-minor~$H$ we obtain rays~${H(A_i)}$ corresponding to our marked rays~$A_i$ in~$G$, which by the above all go to $\epsilon$. 
We will call this family of rays the \emph{bundle} of rays given by~$H$. 

Our aim now is to build up an ${\aleph_0G}$-minor of~$\Gamma$ recursively. 
At stage~$n$ we hope to construct~$n$ disjoint~${G[\bigcup_{m \leq n}V_m]}$-minors~${H^n_1, H^n_2, \ldots, H^n_n}$, 
such that for each such~$H^n_m$ there is a family ${(R^n_{m,i} \colon i \leq k)}$ of disjoint rays in~$\epsilon$, where the path in~$H^n_m$ corresponding to the initial segment of the ray~$A_i$ in~${\bigcup_{m \leq n}G_m}$ is an initial segment of~$R^n_{m,i}$, 
but these rays are otherwise disjoint from the various~$H^n_l$ and from each other, see Figure~\ref{fig:bundles}. 
We aim to do this in such a way that each~$H^n_m$ extends all previous~$H^l_m$ for~${l \leq n}$, so that at the end of our construction we can obtain infinitely many disjoint $G$-minors as~${(\bigcup_{n \geq m}H^n_m \colon m \in \Nbb)}$. 
The rays chosen at later stages need not bear any relation to those chosen at earlier stages; 
we just need them to exist so that there is some hope of continuing the construction. 

We will again refer to the families~${(R^n_{m,i} \colon i \leq k)}$ of rays starting at the various~$H^n_m$ as the \emph{bundles} of rays from those~$H^n_m$. 

\begin{figure}[h!]
    \centering
\begin{tikzpicture}[magenta]
\draw[black] (9,-.6)++(-.1, 0) -- ++(.1,0) -- (9,-2) node[anchor=west,pos=.5] {\scriptsize\textit{bundle}} -- ++(-.1, 0);
\draw[black] (9, -1.2) ;
\foreach \i in {1,...,4}{
\begin{scope}[xshift=-3cm,yshift=-\i*1.5cm]
    \draw (-2.5, 0) node {$ H_\i^n $};
    \draw (-1.4, 0) circle [radius=.5];
    \draw (-.7, .05) circle [radius=.5];
    \draw (0, 0) circle [radius=.5];
    \draw (60:.5) arc (-30:60:-.5);
    \foreach \y in {1,...,3} {
        \begin{scope}[yshift=-\y*.25cm+.5cm]
            \draw[->] (-1.4, 0) to[out=0,in=180] (-.7, .05) to[out=0,in=180] (0,0) to[out=0,in=180] (1, .2) -- +(10, 0);
        \end{scope}
        \draw (11.5, 1 - .4 * \y) node {\scriptsize $ R_{\i,\y}^n $};
        }
\end{scope}
}
\draw[black] (-3, -4.5*1.5) node {$ \vdots $};
\end{tikzpicture}
    \caption{Stage~$n$ of the construction with disjoint ${G[\bigcup_{m \leq n}V_m]}$-minors~$H^n_i$ with their bundles of disjoint rays.}
    \label{fig:bundles}    
\end{figure}
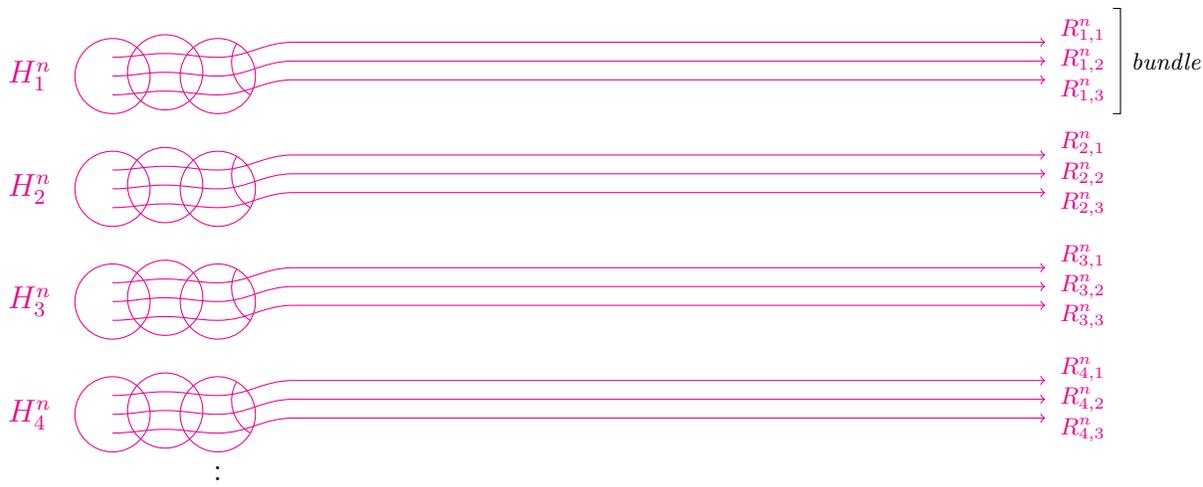

The rough idea for getting from the~$n$\textsuperscript{th} to the~${n+1}$\textsuperscript{st} stage of this construction is now as follows: 
we choose a very large family~$\mathcal{H}$ of disjoint $G$-minors in~$\Gamma$. 
We discard all those which meet any previous~$H^n_m$ and we consider the family of rays corresponding to the~$A_i$ in the remaining minors. 
Then it is possible to find a collection of paths transitioning from the~$R^n_{m,i}$ from stage~$n$ onto these new rays. 
Precisely what we need is captured in the following definition, which also introduces some helpful terminology for dealing with such transitions: 

\begin{definition}[Linkage of families of rays]
    \label{d:linkage}
    Let ${\mathcal{R} = (R_i \colon i\in I)}$ and ${\mathcal{S} = (S_j \colon j \in J)}$ be families of disjoint rays, where the initial vertex of each~$R_i$ is denoted~$x_i$. 
    A family of paths ${\mathcal{P} = (P_i \colon i \in I)}$, is a \emph{linkage} from~$\mathcal{R}$ to~$\mathcal{S}$ if there is an injective function~${\sigma \colon I \rightarrow J}$ such that
    \begin{itemize}
        \item each~$P_i$ goes from a vertex~${x'_i \in R_i}$ to a vertex~${y_{\sigma(i)} \in S_{\sigma(i)}}$; 
        \item the family ${\mathcal{T} = (x_iR_ix'_iP_iy_{\sigma(i)}S_{\sigma(i)} \colon i\in I)}$ is a collection of disjoint rays.\footnote{Where we use the notation as in~\cite{D16}, see also Definition~\ref{def_concat}.}
            We write ${\mathcal{R}\circ_\mathcal{P}\mathcal{S}}$ for the family~$\mathcal{T}$ as well~${R_i \circ_{\mathcal{P}} \mathcal{S}}$ for the ray in~$\mathcal{T}$ with initial vertex~$x_i$. 
    \end{itemize}
    We say that~$\Tcal$ is obtained by \emph{transitioning} from~$\Rcal$ to~$\Scal$ along the linkage.
    We say the linkage~$\mathcal{P}$ \emph{induces} the mapping~$\sigma$. 
    We further say that~$\mathcal{P}$ \emph{links}~$\mathcal{R}$ to~$\mathcal{S}$.
    Given a set~$X$ we say that the linkage is \emph{after}~$X$ if ${X \cap V(R_i) \subseteq V(x_iR_ix'_i)}$ for all~${i \in I}$ and no other vertex in~$X$ is used by the members of~$\mathcal{T}$. 
\end{definition}

Thus, our aim is to find a linkage from the~${R^n_{m,i}}$ to the new rays after all the~$H^n_m$. 
That this is possible is guaranteed by the following lemma from~\cite{BEEGHPTI}: 

\begin{lemma}[Weak linking lemma {\cite[Lemma~4.3]{BEEGHPTI}}]
    \label{l:weaklink}
    Let~$\Gamma$ be a graph and~${\omega \in \Omega(\Gamma)}$. 
    Then, for any families~${\Rcal = (R_1, \ldots, R_n)}$ and~${\Scal = (S_1,\ldots, S_n)}$ of vertex disjoint rays in~$\omega$ and any finite set~$X$ of vertices, there is a linkage from~$\mathcal{R}$ to~$\mathcal{S}$ after~$X$.
\end{lemma}

The aim is now to use property~\ref{item:sketch-4} of our tree-decomposition of~$G$ to find minor-copies of~${G[V_{n+1}]}$ sufficiently far along the new rays that we can stick them onto our~$H^n_m$ to obtain suitable~${H^{n+1}_m}$. 
There are two difficulties at this point in this argument. 
The first is that, as well as extending the existing~$H^n_m$ to~$H^{n+1}_m$, we also need to introduce an~$H^{n+1}_{n+1}$. 
To achieve this, we ensure that one of the $G$-minors in~$\mathcal{H}$ is disjoint from all the paths in the linkage, so that we may take an initial segment of it as~$H^{n+1}_{n+1}$. 
This is possible because of a slight strengthening of the linking lemma above; 
see~\cite[Lemma~4.4]{BEEGHPTI} or Lemma~\ref{l:link} for a precise statement.

A more serious difficulty is that in order to stick the new copy of~$V_{n+1}$ onto~$H^n_m$ we need the following property:

\begin{equation}\tag{$*$}
    \label{property}
    \parbox{12cm}{For each of the bundles corresponding to an~$H^n_m$, the rays in the bundle are linked to the rays in the bundle coming from some~${H \in \mathcal{H}}$. 
    This happens in such a way that each~$R^n_{m,i}$ is linked to~$H(A_i)$.}\vspace{0.3cm}
\end{equation} 

Thus we need a great deal of control over which rays get linked to which. 
We can keep track of which rays are linked to which as follows: 

\begin{definition}[Transition function]
    \label{d:trans-func}
    Let~${\mathcal{R} = (R_i \colon i\in I)}$ and~${\mathcal{S} = (S_j \colon j \in J)}$ be families of disjoint rays. We say that a function~${\sigma \colon I \rightarrow J}$ is a \emph{transition function} from~$\Rcal$ to~$\Scal$ if for any finite set~$X$ of vertices there is a linkage from~$\Rcal$ to~$\Scal$ after~$X$ that induces~$\sigma$.
\end{definition}

So our aim is to find a transition function assigning new rays to the~$R^n_m$ so as to achieve~\eqref{property}. 
One reason for expecting this to be possible is that the new rays all go to the same end, and so they are joined up by many paths. 
We might hope to be able to use these paths to move between the rays, allowing us some control over which rays are linked to which. 
The structure of possible jumps is captured by a graph whose vertex set is the set of rays:

\begin{definition}[Ray graph]
    \label{def_raygraph}
    Given a finite family of disjoint rays~${\mathcal{R} = (R_i \colon i \in I)}$ in a graph~$\Gamma$ the \emph{ray graph}, ${\RG_{\Gamma}(\mathcal{R}) = \RG_{\Gamma}(R_i \colon i \in I)}$ is the graph with vertex set~$I$ and with an edge between~$i$ and~$j$ if there is an infinite collection of vertex disjoint paths from~$R_i$ to~$R_j$ which meet no other~$R_k$. 
    When the host graph~$\Gamma$ is clear from the context we will simply write~${\RG(\mathcal{R})}$ for~${\RG_{\Gamma}(\mathcal{R})}$. 
\end{definition}

Unfortunately, the collection of possible transition functions can be rather limited. 
Consider, for example, the case of families of disjoint rays in the grid. 
Any such family has a natural cyclic order, and any transition function must preserve this cyclic order. 
This paucity of transition functions is reflected in the sparsity of the ray graphs, which are all just cycles. 

However, in a previous paper~\cite{BEEGHPTII} we analysed the possibilities for how the ray graphs and transition functions associated to a given thick\footnote{An end is \emph{thick} if it contains infinitely many disjoint rays.} end may look. We found that there are just three possibilities. 

The easiest case is that in which the rays to the end are very joined up, in the sense that any injective function between two families of rays is a transition function. 
This case was already dealt with in~\cite{BEEGHPTII}, where is was shown that in any graph with such an end we can find a~$K_{\aleph_0}$ minor. 
The second possibility is that which we saw above for the grid: 
all ray graphs are cycles, and all transition functions between them preserve the cyclic order. 
The third possibility is that all ray graphs consist of a path together with a bounded number of further `junk' vertices, where these junk vertices are hanging at the ends of the paths (formally: all interior vertices on this \emph{central path} in the ray graph have degree~$2$). 
In this case, the transition functions must preserve the linear order along the paths. 

The second and third cases can be dealt with using similar ideas, so we will focus on the third one here. 

Since we are assuming that all the $G$-minors in $\Gamma$ go to $\epsilon$, given a large enough collection of $G$-minors $\mathcal{H}$, almost all of the rays from the bundles of the $H \in \mathcal{H}$ lie on the central path of the ray graph of this family of rays, and so in particular by a Ramsey type argument there must be a large collection of $H \in \mathcal{H}$ such that for each $H$, the rays~${H(A_i)}$ appear in the same order along the central path.

Since there are only finitely many possible orders, there is some consistent way to order the~$A_i$ such that for every~$n$ we can find~$n$ disjoint $G$-minors~$H$ such that there is some ray graph in which, for each~$H$, the rays~${H(A_i)}$ appear in this order along the central path, which we can assume, without loss of generality, is from $H(A_1)$ to $H(A_d)$.

This will allow us to recursively maintain a similar property for the rays from the bundles of the $H_m^n$. More precisely, we can guarantee that there is a slightly larger family~$\mathcal{R}$ of disjoint rays, consisting of the~$R^n_{m,i}$ and some extra `junk' rays, such that all of the~$R^n_{m,i}$ lie on the central path of $\RG (\Rcal)$, and for each~$n$ and~$m$ the~$R^n_{m,i}$ appear on this path consecutively in order from~$R^n_{m,1}$ to~$R^n_{m,k}$. 

Then, our extra assumption on the structure of the end $\epsilon$ ensures that given a linkage from $\Rcal$ to the bundles from $H \in \mathcal{H}$ which induces a transition function, we can reroute our linkage, using the edges of $\RG(\Rcal)$, so that~\eqref{property} holds.

There is one last subtle difficulty which we have to address, once more relating to the fact that we want to introduce a new~$H^{n+1}_{n+1}$ together with its private bundle of rays corresponding to its copies of the~$A_i$, disjoint from all the other~$H^{n+1}_m$ and their bundles. Our strengthening of the weak linking lemma allows us to find a linkage which avoids one of the $G$-minors in~$\mathcal{H}$, but this linkage may not have property~\eqref{property}. 

We can, as before, modify it to one satisfying~\eqref{property} by rerouting the linkage, but this new linkage may then have to intersect some of the rays in the bundle of~$H^{n+1}_{n+1}$, if these rays from~$H^{n+1}_{n+1}$ lie between rays linked to a bundle of some~$H^n_m$, see Figure~\ref{fig:only_rounting}. 

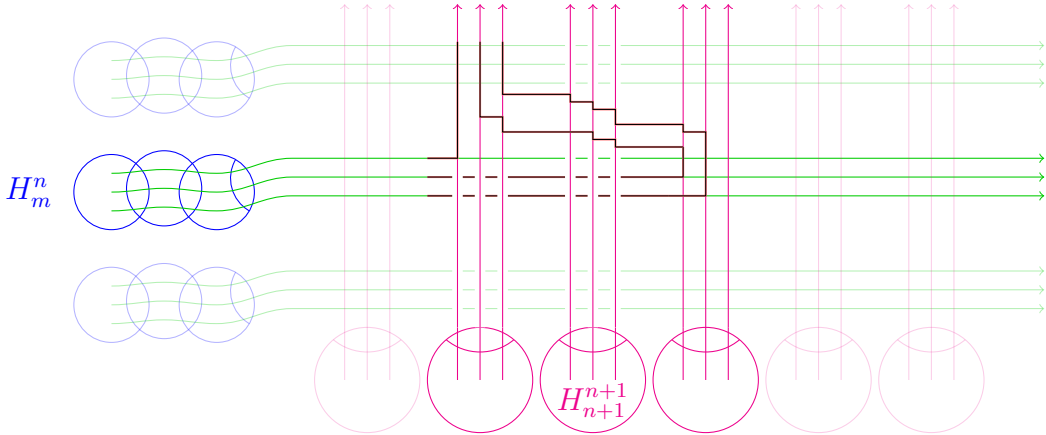
\begin{figure}[h]
    \centering
\begin{tikzpicture}
\colorlet{extendercolor}{green!80!black}
\newcommand\Hn[1]{
    \begin{scope}[blue]
    \draw (-2.5, 0) node {#1};
    \draw (-1.4, 0) circle [radius=.5];
    \draw (-.7, .05) circle [radius=.5];
    \draw (0, 0) circle [radius=.5];
    \draw (60:.5) arc (-30:60:-.5);
    \end{scope}
    \foreach \y in {1,...,3} {
        \begin{scope}[yshift=-\y*.25cm+.5cm]
            \draw[->,extendercolor] (-1.4, 0) to[out=0,in=180] (-.7, .05) to[out=0,in=180] (0,0) to[out=0,in=180] (1, .2) -- +(10, 0);
        \end{scope}
}}

\begin{scope}[yshift=1.5cm,opacity=.3]
    \Hn{}
\end{scope}
\Hn{$ H_m^n $}
\begin{scope}[yshift=-1.5cm,opacity=.3]
    \Hn{}
\end{scope}

\newcommand\newH[2]{
    \begin{scope}[magenta,yshift=-2.5cm]
    \draw (0, 0) circle [radius=.7];
    \draw (50:.7) arc (130:50:-.7);
    \foreach \x in {-1,...,1}{
        \begin{scope}[xshift=\x*.3cm]
            \draw (0, .0) -- (0,.7);
            \draw[->,preaction={#2}] (0,.7)  -- (0, 5);
        \end{scope}
    }
    \draw (0,0) node[anchor=north,yshift=3pt] {#1};
    \end{scope}
}

\begin{scope}[xshift=6.5 cm]
    \newH{}{}
\end{scope}
\tikzstyle{transition}=[draw=red!50!white,thick,postaction={draw=black,thin}]
\begin{scope}
\draw[transition] (2.8, .2) -| (6.5-.3, .6);
\draw[transition] (2.8, -.05) -- (6, -.05);
\end{scope}

\begin{scope}[xshift=3.5 cm]
        \newH{}{draw=white,line width = 4pt,shorten >= 2.2cm,-}
\end{scope}
\begin{scope}[xshift=5cm]
    \newH{$ H_{n+1}^{n+1} $}{draw=white,line width = 4pt}
\end{scope}

\begin{scope}
\draw[transition] (2.8, .45) -| (3.5-.3, 2);
\draw[transition] (6.5-.3, .6) -| (5+.3, .7) -| (5,.8) -| (5-.3,.8) -| (3.5+.3, 1) -| (3.5, 2);
\draw[transition] (6, -.05) -| (6.5, .8) -| (6.5-.3, .9) -| (5+.3, 1.1) -| (5,1.2) -| (5-.3,1.3) -| (3.5+.3, 2);
\end{scope}

\foreach \x in {8, 9.5,2} {
    \begin{scope}[xshift=\x cm, opacity=.2]
        \newH{}{}
    \end{scope}
}

\end{tikzpicture}
    \caption{Extending the $H^n_m$ by routing onto a set of disjoint $G$-minors might cause problems with introducing a new $H^{n+1}_{n+1}$ disjoint to the rest.}
    \label{fig:only_rounting} 
\end{figure}

However, we can get around this by instead rerouting the rays in~$\mathcal{R}$ \emph{before} the linkage, so as to rearrange which bundles make use of (the tails of) which rays. Of course, we cannot know before we choose our linkage how we will need to reroute the rays in~$\mathcal{R}$, but we do know that the structure of $\epsilon$ restricts the possible reroutings we might need to do. 

Hence, we can avoid this issue by first taking a large, but finite, set of paths between the rays in $\mathcal{R}$ which is rich enough to allow us to reroute the rays in $\mathcal{R}$ in every way which is possible in $\Gamma$. Since the rays in $\mathcal{R}$ also go to $\epsilon$, the structure of $\epsilon$ will guarantee that this includes all of the possible reroutings we might need to do. We call such a collection a \emph{transition box}. 

Only after building our transition box do we choose the linkage from~$\Rcal$ to the rays from~$\mathcal{H}$, and we make sure that this linkage is after the transition box. 
Then, when we later see how the rays in $\Rcal$ should be arranged in order that the rays from the bundle of~$H^{n+1}_{n+1}$ do not appear between rays linked to a bundle of some~$H^n_m$, we can go back and perform a suitable rerouting within the transition box, see Figure~\ref{l_fig_intro}.

\begin{figure}[h!]
    \centering
    \scalebox{.7}{
\begin{tikzpicture}[scale=1.1]
    \colorlet{junkcolor}{purple}
    \tikzstyle{JunkRay}=[junkcolor,->]
    \colorlet{extendercolor}{green!80!black}
    \tikzstyle{ExtenderRay}=[extendercolor,->]
    \begin{scope}[yshift=2cm]
        \foreach \y in {0,...,-2}
            \draw[JunkRay] (0, 0.2*\y) -- +(16,0);
    \end{scope}

    \newcommand\Qin[1]{
        \draw[blue] (0,0) circle [radius=1];
        \draw[blue] (-1.5,.35) circle [radius=1];
        \draw[blue] (-3,0) circle [radius=1];
        \draw[blue] (-40:1) arc (50:-40:-1);
        \draw[ExtenderRay] (-3,.6) to[out=0,in=180] ++(1.3, .2) to[out=0,in=180] (.65,.6) -- (16, .6);
        \draw[ExtenderRay] (-3,.3) to[out=0,in=180] ++(1.3, .2) to[out=0,in=180] (.8,.3) -- (16, .3);
        \draw[ExtenderRay] (-3,-.5) to[out=0,in=180] ++(1.3, .2) to[out=0,in=180]  (.75,-.5) -- (16,-.5);
    }
    \begin{scope}[yshift=0cm]
        \Qin{1}
    \end{scope}

    \draw[blue] (0, -2) node {$ \vdots $};
    \draw[blue] (0, -3) node {$ \vdots $};

    \begin{scope}[yshift=-5cm]
        \Qin{n}
    \end{scope}

    \begin{scope}[yshift=-7cm]
        \foreach \y in {0,...,3}
            \draw[JunkRay] (0, 0.2*\y) -- +(16,0);
    \end{scope}


    \colorlet{transitioncolor}{olive}
    \filldraw[draw=transitioncolor,fill=white,opacity=.9] (4, 3) rectangle (7, -8);
    \draw[transitioncolor] (5.5, 2.5) node {transition box};

    \draw[transitioncolor, thick] (4,.6) to[out=0,in=180] (7,2);
    \draw[transitioncolor, thick] (4,.3) to[out=0,in=180] (7,2-.2);
    \draw[transitioncolor, thick] (4,-.5) to[out=0,in=180] (7,2-.4);

    \draw[transitioncolor, thick] (4,.6-5) to[out=0,in=180] (7,.6-5);
    \draw[transitioncolor, thick] (4,.3-5) to[out=0,in=180] (7,-.5-5);
    \draw[transitioncolor, thick] (4,-.5-5) to[out=0,in=180] (7,-7+.6);


    \newcommand\newH{
        \begin{scope}[yshift=-9cm]
            \draw[magenta] (0, 0) circle [radius=.5];
            \draw[magenta] (155:.5) arc (45:155:-.5);
                \draw[magenta,->] (.25,.3) -- ++(0, 13.25);
                \draw[magenta,->] (-.15,.35) -- ++(0, 13.25);
                \draw[magenta,->] (-.35,.2) -- ++(0, 13.25);
        \end{scope}
    }

    \begin{scope}[xshift=9cm]
        \newH
    \end{scope}
    \draw[magenta] (10, -9) node {$ \cdots $};
    \draw[magenta] (10, 3.25) node {$ \cdots $};
    \begin{scope}[xshift=11cm]
        \newH
    \end{scope}
    \draw[magenta] (12, -9) node {$ \cdots $};
    \draw[magenta] (12, 3.25) node {$ \cdots $};

    \draw[magenta] (14,-9) node {$ \cdots $};
    \draw[magenta] (14,3.25) node {$ \cdots $};
    \begin{scope}[xshift=15cm]
        \newH
    \end{scope}

    \filldraw[draw=orange,fill=white,opacity=.9] (8, 3) rectangle (15.5, -8);

    \colorlet{pprimecolor}{red!50!white}
    \tikzstyle{Pprime}=[draw=pprimecolor, double=black]
    \tikzstyle{Pprimedashedend}=[Pprime,dashed,preaction={draw,solid,Pprime,shorten >=.5cm}]
    \tikzstyle{dashedend}=[dashed,postaction={draw,solid,shorten >=.5cm}]

    \draw[orange] (9, -1) node {linkage};

    \draw[Pprime] (8,2) -| (9-.35,3);
    \draw[Pprime] (8,2-.2) -| (9.25,3);
    \draw[Pprime] (8,2-.4) -| (9.5,3);
    \draw (8,.6) -| (9.75,3);
    \draw (8,.3) -| (10,3);
    \draw (8,-.5) -| (10.2,3);

    \draw[Pprimedashedend] (8,.6-5) -| (10.6,-3);
    \draw[dashedend] (8,.3-5) -| (10.7,-3.2);
    \draw[Pprimedashedend] (8,-.5-5) -| (10.8,-4);

    \draw[Pprimedashedend] (8,-7+.6) -| (11.2,-4.5);
    \draw[dashedend] (8,-7+.4) -| (11.8,-5);
    \draw[dashedend] (8,-7+.2) -| (14.7,-5.2);
    \draw[dashedend] (8,-7)    -| (14.9,-5.3);

    \draw (11, 2.5) node {$ \cdots $}
          (11, 1.5) node {$ \cdots $};

    \begin{scope}[xshift=13cm]
        \fill[fill=white, opacity=.95] (-.7, 4) rectangle (.7, -10.55);
        \draw[green, dashed] (-.7,4) |- (.7, -9.55) -- (.7, 4);
        \newH
    \end{scope}

\end{tikzpicture}
    }
    \caption{The transitioning strategy between the old and new bundles.}
    \label{l_fig_intro}
\end{figure}
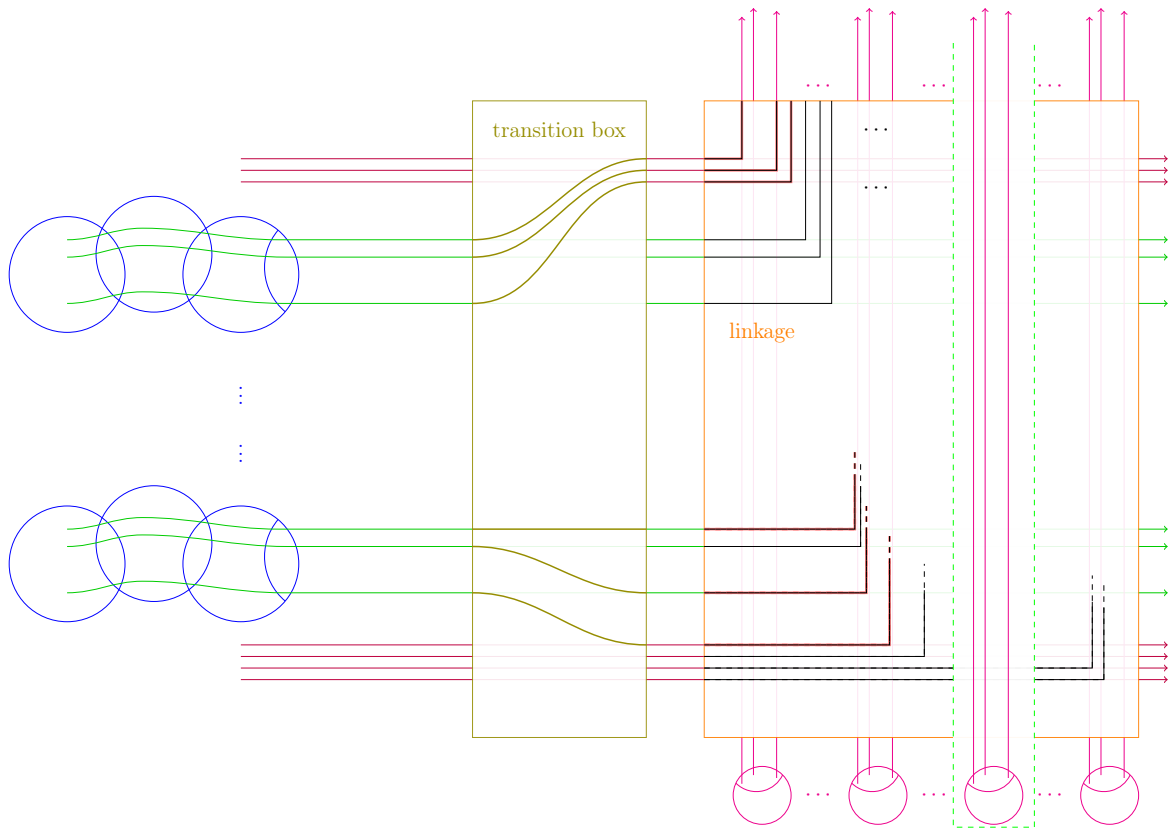

This completes the sketch of the proof that locally finite graphs with a single end of finite degree are ubiquitous. 
Our results in this paper are for a more general class of graphs, but one which is chosen to ensure that arguments of the kind outlined above will work for them. 
Hence we still need a tree-decomposition with properties similar to \ref{item:sketch-1}--\ref{item:sketch-4} from our ray-decomposition above. 
Tree-decompositions with these properties are called \emph{extensive}, and the details can be found in Section~\ref{s:extensive}. 

However, certain aspects of the sketch above must be modified to allow for the fact that we are now dealing with graphs~$G$ with multiple, indeed possibly infinitely many, ends. 
For any end~$\delta$ of~$G$ and any $G$-minor~$H$ of~$\Gamma$, all rays~${H(R)}$ with~$R$ in~$\delta$ belong to the same end~$H(\delta)$ of~$\Gamma$. 
If~$\delta$ and~$\delta'$ are different ends in~$G$, then~$H(\delta)$ and~$H(\delta')$ may well be different ends in~$\Gamma$ as well. 

So there is no hope of finding a single end~$\epsilon$ of~$\Gamma$ to which all rays in all $G$-minors converge. 
Nevertheless, we can still find an end~$\epsilon$ of~$\Gamma$ towards which the $G$-minors are \emph{concentrated}, in the sense that for any finite vertex set~$X$ there are arbitrarily large families of $G$-minors in the same component of~${G-X}$ as all rays of~$\epsilon$ have tails in. 
See Section~\ref{s:tribes} for details. 
In that section we introduce the term \emph{tribe} for a collection of arbitrarily large families of disjoint $G$-minors. 

The recursive construction will work pretty much as before, in that at each step~$n$ we will again have embedded $G^n$-minors for some large finite part~$G^n$ of~$G$, together with a number of rays in~$\epsilon$ corresponding to some designated rays going to certain ends~$\delta$ of~$G$. 

In order for this to work, we need some consistency about which ends~$H(\delta)$ of~$\Gamma$ are equal to~$\epsilon$ and which are not. 
It is clear that for any finite set~$\Delta$ of ends of~$G$ there is some subset~$\Delta'$ such that there is a tribe of $G$-minors~$H$ converging to~$\epsilon$ with the property that the set of~$\delta$ in~$\Delta$ with~${H(\delta) = \epsilon}$ is~$\Delta'$. 
This is because there are only finitely many options for this set. 
But if~$G$ has infinitely many ends, there is no reason why we should be able to do this for all ends of~$G$ at once.

Our solution is to keep track of only finitely many ends of~$G$ at any stage in the construction, and to maintain at each stage a tribe concentrated towards~$\epsilon$ which is consistent for all these finitely many ends. 
Thus in our construction consistency of questions such as which ends~$\delta$ of~$G$ converge to~$\epsilon$ or of the proper linear order in the ray graph of the families of canonical rays to those ends is achieved dynamically during the construction, rather than being fixed in advance. 
The ideas behind this dynamic process have already been used successfully in our earlier paper~\cite{BEEGHPTI}, where they appear in slightly simpler circumstances. 

The paper is then structured as follows. 
In Section~\ref{s:prelim} we give precise definitions of some of the basic concepts we will be using, and prove some of their fundamental properties. 
In Section~\ref{s:extensive} we introduce extensive tree-decompositions and in Section~\ref{s:getnice} we illustrate that many locally finite graphs admit such decompositions. In Section~\ref{s:nonpebbly} we analyse the possible collections of ray graphs and transition functions between them which can occur in a thick end. In Section~\ref{s:tribes} we introduce the notion of tribes and of their concentration towards an end and begin building some tools for the main recursive construction, which is given in Section~\ref{sec:countable-subtrees}. We conclude with a discussion of the future outlook in Section~\ref{s:WQO}.

\section{Preliminaries}
\label{s:prelim}

In this paper we will denote by~$\Nbb$ the set of positive integers and by~$\Nbb_0$ the set of non-negative integers. In our graph theoretic notation we generally follow the textbook of Diestel~\cite{D16}. 
For a graph~${G = (V,E)}$ and~${W \subseteq V}$ we write~${G[W]}$ for the induced subgraph of~$G$ on~$W$. For two vertices~${v,w}$ of a connected graph~$G$, we write~$\dist(v,w)$ for the edge-length of a shortest $v$--$w$ path. A path ${P = v_0v_1\ldots v_n}$ in a graph~$G$ is called a \emph{bare path} if~${d_G(v_i) = 2}$ for all inner vertices~$v_i$ for~${0 < i < n}$.

\subsection{Rays and ends}

\begin{definition}[Rays, double rays and initial vertices of rays]
    A one-way infinite path is called a \emph{ray} and a two-way infinite path is called a \emph{double ray}. 
    For a ray~$R$, let~$\init(R)$ denote the \emph{initial vertex} of~$R$, that is the unique vertex of degree~$1$ in~$R$. 
    For a family~$\mathcal{R}$ of rays, let~${\init(\mathcal{R})}$ denote the set of initial vertices of the rays in~$\mathcal{R}$. 
\end{definition}
    
\begin{definition}[Tail of a ray]
    Given a ray~$R$ in a graph~$G$ and a finite set~${X \subseteq V(G)}$, the \emph{tail of~$R$ after~$X$}, written~${T(R,X)}$, is the unique infinite component of~$R$ in~${G - X}$. 
\end{definition}
    
\begin{definition}[Concatenation of paths and rays]  
    \label{def_concat}
    For a path or ray~$P$ and vertices ${v,w \in V(P)}$, let~${vPw}$ denote the subpath of~$P$ with endvertices~$v$ and~$w$, and~${\mathring{v}P\mathring{w}}$ the subpath strictly between~$v$ and~$w$. 
    If~$P$ is a ray, let~$Pv$ denote the finite subpath of~$P$ between the initial vertex of~$P$ and~$v$, and let~$vP$ denote the subray (or \emph{tail}) of~$P$ with initial vertex~$v$. 
    Similarly, we write~${P\mathring{v}}$ and~${\mathring{v}P}$ for the corresponding path/ray without the vertex~$v$. 
    For a ray ${R = r_0 r_1 \ldots}$, let~$R^{-}$ denote the tail~${r_1 R}$ of~$R$ starting at~$r_1$.
    Given a family~$\mathcal{R}$ of rays, let~$\mathcal{R}^{-}$ denote the family~${(R^{-} \colon R \in \mathcal{R})}$.
    
    Given two paths or rays~$P$ and~$Q$, which intersect in a single vertex only, which is an endvertex in both~$P$ and~$Q$, we write~$PQ$ for the \emph{concatenation of~$P$ and~$Q$}, that is the path, ray or double ray~${P \cup Q}$. 
    Moreover, if we concatenate paths of the form~$vPw$ and~$wQx$, then we omit writing~$w$ twice and denote the concatenation by~${vPwQx}$. 
\end{definition}

\begin{definition}[{Ends of a graph, cf.~\cite[Chapter~8]{D16}}]
    An \emph{end} of an infinite graph~$\Gamma$ is an equivalence class of rays, where two rays~$R$ and~$S$ of~$\Gamma$ are \emph{equivalent} if and only if there are infinitely many vertex disjoint paths between~$R$ and~$S$ in~$\Gamma$. 
    We denote by~${\Omega(\Gamma)}$ the set of ends of~$\Gamma$. 
    
    We say that a ray~${R \subseteq \Gamma}$ \emph{converges} (or \emph{tends}) to an end~$\epsilon$ of~$\Gamma$ if~$R$ is contained in~$\epsilon$. 
    In this case, we call~$R$ an \emph{$\epsilon$-ray}. 
    Given an end~${\epsilon \in \Omega(\Gamma)}$ and a finite set~${X \subseteq V(\Gamma)}$ there is a unique component of~${\Gamma - X}$ which contains a tail of every ray in~$\epsilon$, which we denote by~${C(X,\epsilon)}$. Given two ends $\epsilon, \epsilon' \in \Omega(\Gamma)$, we say a finite set $X \subseteq V(\Gamma)$ \emph{separates} $\epsilon$ and $\epsilon'$ if $C(X,\epsilon) \neq C(X,\epsilon')$.
    
    For an end~${\epsilon \in \Omega(\Gamma)}$, we define the \emph{degree} of~$\epsilon$ in~$\Gamma$, denoted by~${\deg(\epsilon)}$, as the supremum in ${\Nbb \cup \{ \infty \}}$ of the set ${ \{ |\mathcal{R}| \; \colon \; \mathcal{R} \textnormal{ is a set of disjoint } \epsilon \textnormal{-rays} \} }$. 
    Note that this supremum is in fact an attained maximum, i.e.~for each end~$\epsilon$ of~$\Gamma$ there is a set~$\mathcal{R}$ of vertex-disjoint $\epsilon$-rays with~${|\mathcal{R}| = \deg(\omega)}$, as proved by Halin~\cite[Satz~1]{H65}. 
    An end with finite degree is called \emph{thin}, otherwise the end is called \emph{thick}.
\end{definition}

\subsection{Inflated copies of graphs}

\begin{definition}[Inflated graph, branch set]
    Given a graph~$G$, we say that a pair~${(H,\varphi)}$ is an \emph{inflated copy of~$G$}, or an~$IG$, if~$H$ is a graph and~${\varphi \colon V(H) \rightarrow V(G)}$ is a map such that: 
    \begin{itemize}
        \item For every~${v \in V(G)}$ the \emph{branch set}~${\varphi^{-1}(v)}$ induces a non-empty, connected subgraph of~$H$; 
        \item There is an edge in~$H$ between~${\varphi^{-1}(v)}$ and~${\varphi^{-1}(w)}$ if and only if~${vw \in E(G)}$ and this edge, if it exists, is unique. 
    \end{itemize}
\end{definition}

When there is no danger of confusion, we will simply say that~$H$ is an~$IG$ instead of saying that~${(H,\varphi)}$ is an~$IG$, and denote by~${H(v) = \varphi^{-1}(v)}$ the branch set of~$v$. 

\begin{definition}[Minor]
    A graph~$G$ is a \emph{minor} of another graph~$\Gamma$, written~${G \preceq \Gamma}$, if there is some subgraph~${H \subseteq \Gamma}$ such that~$H$ is an inflated copy of~$G$. In this case, we also say that $H$ is a \emph{$G$-minor} in $\Gamma$. 
\end{definition}

\begin{definition}[Extension of inflated copies]
    Suppose~${G \subseteq G'}$ as subgraphs, and that~$H$ is an~$IG$ and~$H'$ is an~$IG'$. 
    We say that~$H'$ \emph{extends}~$H$ (or that~$H'$ is an \emph{extension} of~$H$) if~${H \subseteq H'}$ as subgraphs and~${H(v) \subseteq H'(v)}$ for all~${v \in V(G)}$. 
    Note that, since~${H \subseteq H'}$, for every edge~${vw \in E(G)}$ the unique edge between the branch sets~$H'(v)$ and~$H'(w)$ is also the unique edge between~$H(v)$ and~$H(w)$.
    
    If~$H'$ is an extension of~$H$ and~${X \subseteq V(G)}$ is such that~${H'(x) = H(x)}$ for every~${x \in X}$, then we say~$H'$ is an extension of~$H$ \emph{fixing~$X$}. 
\end{definition}

\begin{definition}[Tidiness]
    Let~${(H,\varphi)}$ be an~$IG$. 
    We call~${(H,\varphi)}$ \emph{tidy} if 
    \begin{itemize}
        \item $H[\varphi^{-1}(v)]$ is a tree for all~${v \in V(G)}$;
        \item $H[\varphi^{-1}(v)]$ is finite if~$d_G(v)$ is finite.
    \end{itemize}
\end{definition}

Note that every~$H$ which is an~$IG$ contains a subgraph~$H'$ such that~${(H',\varphi \restriction V(H'))}$ is a tidy~$IG$, although this choice may not be unique. 
In this paper we will always assume, without loss of generality, that each~$IG$ is tidy. 

\begin{definition}[Restriction]
    Let~$G$ be a graph, ${M \subseteq G}$ a subgraph of~$G$, and let~${(H,\varphi)}$ be an~$IG$. 
    The \emph{restriction of~$H$ to~$M$}, denoted by~$H(M)$, is the~$IM$ given by~${(H(M),\varphi')}$, where~${\varphi'^{-1}(v) = \varphi^{-1}(v)}$ for all~${v \in V(M)}$, and~$H(M)$ consists of the union of the subgraphs of~$H$ induced on each branch set~${\varphi^{-1}(v)}$ for each~${v \in V(M)}$, together with the edge between~${\varphi^{-1}(u)}$ and~${\varphi^{-1}(v)}$ in~$H$ for each~${uv \in E(M)}$.
\end{definition}

Suppose~$R$ is a ray in some graph~$G$. 
If~$H$ is a tidy~$IG$ in a graph~$\Gamma$, then in the restriction~$H(R)$ all rays which do not have a tail contained in some branch set will share a tail. 
Later in the paper, we will want to make this correspondence between rays in~$G$ and~$\Gamma$ more explicit, with use of the following definition: 

\begin{definition}[Pullback]
    Let~$G$ be a graph, ${R \subseteq G}$ a ray, and let~${(H, \varphi)}$ be a tidy~$IG$. 
    The \emph{pullback of~$R$ to~$H$} is the subgraph~${H^{\downarrow}(R) \subseteq H(R)}$, where~${H^{\downarrow}(R)}$ is subgraph minimal such that~${(H^{\downarrow}(R), \varphi \restriction V(H^{\downarrow}(R)))}$ is an~$IR$. 
\end{definition}

Note that, since~$H$ is tidy, ${H^{\downarrow}(R)}$ is well defined. It can be shown that, in fact, $H^{\downarrow}(R)$ is also ray.

\begin{lemma}[{\cite[Lemma~2.11]{BEEGHPTII}}]
    \label{l:pullbackray}
    Let~$G$ be a graph and let~$H$ be a tidy~$IG$. 
    If~${R \subseteq G}$ is a ray, then the pullback~${H^{\downarrow}(R)}$ is also a ray.
\end{lemma}

\begin{definition}
    \label{d:rayfamily}
    Let~$G$ be a graph, $\Rcal$ be a family of disjoint rays in~$G$, and let~$H$ be a tidy~$IG$. 
    We will write~$H^{\downarrow}(\Rcal)$ for the family ${(H^{\downarrow}(R) \colon R \in \Rcal)}$.
\end{definition}

It is easy to check that if two rays~$R$ and~$S$ in~$G$ are equivalent, then also~$H^{\downarrow}(R)$ and~$H^{\downarrow}(S)$ are rays (Lemma~\ref{l:pullbackray}) which are equivalent in~$H$, and hence also equivalent in~$\Gamma$.

\begin{definition}
    \label{def:H(omega)}
    For an end~$\omega$ of~$G$ and~${H \subseteq \Gamma}$ a tidy~$IG$, we denote by~$H(\omega)$ the unique end of~$\Gamma$ containing all rays~${H^{\downarrow}(R)}$ for~${R \in \omega}$. 
\end{definition}

\subsection{Transitional linkages and the strong linking lemma}

The next definition is based on definitions already stated in Section~\ref{s:sketch} (cf.~Definition~\ref{d:linkage}, Definition~\ref{d:trans-func} and Definition~\ref{def_raygraph}).

\begin{definition}
    We say a linkage between two families of rays is \emph{transitional} if the function which it induces between the corresponding ray graphs is a transition function.
\end{definition}

\begin{lemma}
    \label{l:trans-link}
    Let~$\Gamma$ be a graph and let~${\epsilon \in \Omega(\Gamma)}$. 
    Then, for any finite families ${\Rcal = (R_i \colon i \in I)}$ and ${\Scal = (S_j \colon j \in J)}$ of disjoint $\epsilon$-rays in~$\Gamma$, there is a finite set~$X$ such that every linkage from~$\Rcal$ to~$\Scal$ after~$X$ is transitional.
\end{lemma}

\begin{proof}
    By definition, for every function~${\sigma \colon I \rightarrow J}$ which is not a transition function from~$\Rcal$ to~$\Scal$ there is a finite set~${X_\sigma \subseteq V(\Gamma)}$ such that there is no linkage from~$\Rcal$ to~$\Scal$ after~$X_\sigma$ which induces~$\sigma$. 
    If we let~$\Phi$ be the set of all such~$\sigma$ which are not transition functions, then the set~${X := \bigcup_{\sigma \in \Phi} X_\sigma}$ satisfies the conclusion of the lemma. 
\end{proof}

In addition to Lemma~\ref{l:weaklink}, we will also need the following stronger linking lemma, which is a slight modification of~\cite[Lemma~4.4]{BEEGHPTI}: 

\begin{lemma}[Strong linking lemma]
    \label{l:link}
    Let~$\Gamma$ be a graph and let~${\epsilon \in \Omega(\Gamma)}$. 
    Let~$X$ be a finite set of vertices and let~${\Rcal = (R_i \colon i \in I)}$ a finite family of vertex disjoint $\epsilon$-rays. 
    Let~${x_i = \init(R_i)}$ and let~${x'_i = \init(T(R_i,X))}$. 
    Then there is a finite number~${N = N(\Rcal,X)}$ with the following property: 
    For every collection ${(H_j \colon j\in[N])}$ of vertex disjoint subgraphs of~$\Gamma$, all disjoint from~$X$ and each including a specified ray~$S_j$ in~$\epsilon$, there is an~${\ell \in [N]}$ and a transitional linkage ${\Pcal = (P_i \colon i \in I)}$ from~$\Rcal$ to ${(S_j \colon j \in [N])}$, with transition function~$\sigma$, which is after~$X$ and such that the family 
    \[
        \mathcal{T} = \left(x_iR_ix'_iP_iy_{\sigma(i)}S_{\sigma(i)} \colon i\in I\right)
    \]
    avoids~$H_{\ell}$. 
\end{lemma}

\begin{proof}
    Let~${Y \subseteq V(\Gamma)}$ be a finite set as in Lemma~\ref{l:trans-link}. 
    We apply the strong linking lemma established in~\cite[Lemma~4.4]{BEEGHPTI} to the set~${X \cup Y}$ to obtain this version of the strong linking lemma.
\end{proof}

\begin{lemmadef}
    \label{l:transition-box}
    Let~$\Gamma$ be a graph, ${\epsilon \in \Omega(\Gamma)}$, ${X \subseteq V(\Gamma)}$ be finite, and let ${\mathcal{R} = ( R_i \colon i \in I)}$, ${\mathcal{S} = ( S_j \colon j \in J)}$ be two finite families of disjoint $\epsilon$-rays with~${|I| \leq |J|}$.
    Then there is a finite subgraph~${Y}$ such that, for any transition function~$\sigma$ from~$\mathcal{R}$ to~$\mathcal{S}$, there is a linkage~$\mathcal{P}_\sigma$ from~$\mathcal{R}$ to~$\mathcal{S}$ inducing~$\sigma$, with~${\bigcup \mathcal{P}_\sigma \subseteq Y}$, which is after $X$.
    
    We call such a graph~$Y$ a \emph{transition box between~$\mathcal{R}$ and~$\mathcal{S}$ (after~$X$)}. 
\end{lemmadef}

\begin{proof}
    Let~${\sigma \colon I \rightarrow J}$ be a transition function from~$\mathcal{R}$ to~$\mathcal{S}$. 
    By definition, there is a linkage~$\mathcal{P}_\sigma$ from~$\mathcal{R}$ to~$\mathcal{S}$ after~$X$ which induces~$\sigma$. Let~$\Phi$ be the set of all transition functions from~$\mathcal{R}$ to~$\mathcal{S}$ and let~${Y = \bigcup_{\sigma \in \Phi} \mathcal{P}_\sigma}$. 
    Then~$Y$ is a transition box between~$\mathcal{R}$ and~$\mathcal{S}$ (after~$X$).
\end{proof}

\begin{remarkdef}
    \label{rem:trans-link-combine}
    Let~$\Gamma$ be a graph and~${\epsilon \in \Omega(\Gamma)}$.
    Let $\Rcal_1$, $\Rcal_2$, $\Rcal_3$ be finite families of disjoint $\epsilon$-rays, $\mathcal{P}_1$ a transitional linkage from~$\Rcal_1$ to~$\Rcal_2$, 
    and let~$\mathcal{P}_2$ a transitional linkage from~$\Rcal_2$ to~$\Rcal_3$ after~${V(\bigcup \mathcal{P}_2)}$. Then
    \begin{enumerate}
        \item $\mathcal{P}_2$ is also a transitional linkage
        from ${(\Rcal_1 \circ_{\mathcal{P}_1} \Rcal_2)}$ to~$\Rcal_3$;\footnote{Formally, it is only the subset of $\mathcal{P}_2$ starting at the endpoints of $\mathcal{P}_1$ which is a linkage from ${(\Rcal_1 \circ_{\mathcal{P}_1} \Rcal_2)}$ to~$\Rcal_3$. Here and later in the paper, we will use such abuses of notation, when the appropriate subset of the path family is clear from context.}
        \item The linkage from~$\Rcal_1$ to~$\Rcal_3$ yielding the rays ${(\Rcal_1 \circ_{\mathcal{P}_1} \Rcal_2) \circ_{\mathcal{P}_2} \Rcal_3}$, which we call the \emph{concatenation~${\mathcal{P}_1 + \mathcal{P}_2}$ of~$\mathcal{P}_1$ and~$\mathcal{P}_2$}, is transitional.
    \end{enumerate}
\end{remarkdef}

The following lemmas are simple exercises.

\begin{lemma}
  \label{rem:tail-raygraph}
   Let~$\Gamma$ be a graph and ${( R_i \colon i \in I)}$ be a finite family of equivalent disjoint rays. Then the ray graph~${\RG( R_i \colon i \in I)}$ is connected.  Also, if~$R'_i$ is a tail of~$R_i$ for each~${i \in I}$, then we have that~${{\RG( R_i \colon i \in I)} = {\RG( R'_i \colon i \in I)}}$. 
   \qed
\end{lemma}

\begin{lemma}[{\cite[Lemma 3.4]{BEEGHPTII}}]
    \label{l:rayinducedsubgraph}
    Let~$\Gamma$ be a graph, ${\Gamma' \subseteq \Gamma}$, ${\mathcal{R} = (R_i \colon i \in I)}$ be a finite family of disjoint rays in~$\Gamma'$, and let~${\mathcal{S} = (S_j \colon j \in J)}$ be a finite family of disjoint rays in~${\Gamma - V(\Gamma')}$, where~$I$ and~$J$ are disjoint. 
    Then~${\RG_{\Gamma'}(\mathcal{R})}$ is a subgraph of~${\RG_{\Gamma}(\mathcal{R} \cup \mathcal{S})\big[I \big]}$. 
    \qed
\end{lemma}

\subsection{Separations and tree-decompositions of graphs}

\begin{definition}
    \label{def_separation}
    Let~${G = (V,E)}$ be a graph. 
    A \emph{separation} of~$G$ is a pair~${(A,B)}$ of subsets of vertices such that~${A \cup B = V}$ and such that there is no edge between~${B \setminus A}$ and~${A \setminus B}$. 
    Given a separation~${(A,B)}$, we write~$\overline{G[B]}$ for the graph obtained by deleting all edges in the \emph{separator}~${A \cap B}$ from~${G[B]}$.
    Two separations~${(A,B)}$ and~${(C,D)}$ are \emph{nested} if one of the following conditions hold: 
    \begin{align*}
        &{A \subseteq C} \textnormal{ and } {D \subseteq B},  \qquad \textnormal { or } \qquad 
        {B \subseteq C} \textnormal{ and } {D \subseteq A},  \qquad \textnormal { or } \qquad \\
        &{A \subseteq D} \textnormal{ and }  {C \subseteq B} ,  \qquad \textnormal { or } \qquad 
        {B \subseteq D} \textnormal{ and }  {C \subseteq A}.
    \end{align*}
\end{definition}
\begin{definition}
    Let~$T$ be a tree with a root~${v \in V(T)}$. 
    Given nodes~${x, y \in V(T)}$, let us denote by~$xTy$ the unique path in~$T$ between~$x$ and~$y$, 
    by~$T_x$ denote the component of~${T - E(vTx)}$ containing~$x$, 
    and by~$\overline{T_x}$ the tree~${T - T_x}$. 
    
    Given an edge~${e = tt' \in E(T)}$, we say that~$t$ is the \emph{lower vertex} of~$e$, denoted by~$e^-$, if~${t \in vTt'}$. 
    In this case, $t'$ is the \emph{higher vertex} of~$e$, denoted by~$e^+$.
    
    If~$S$ is a subtree of a tree~$T$, let us write~${\partial(S) = E(S,T \setminus S)}$ for the edge cut between~$S$ and its complement in~$T$. 
    
    We say that~$S$ is a \emph{initial subtree} of~$T$ if~$S$ contains~$v$. 
    In this case, we consider~$S$ to be rooted in~$v$ as well.
\end{definition}

A reader unfamiliar with tree-decompositions may also consult~\cite[Chapter 12.3]{D16}.

\begin{definition}[Tree-decomposition]
    Given a graph~${G = (V,E)}$, a \emph{tree-decomposition} of~$G$ is a pair ${(T,\mathcal{V})}$ consisting of a rooted tree~$T$, together with a family of subsets of vertices ${\mathcal{V} = (V_t \colon t \in V(T))}$, such that: 
    \begin{itemize}
        \item ${V(G) = \bigcup \mathcal{V}}$;
        \item For every edge~${e \in E(G)}$ there is a~${t \in V(T)}$ such that~$e$ lies in~${G[V_t]}$;
        \item ${V_{t_1} \cap V_{t_3} \subseteq V_{t_2}}$ whenever~${t_2 \in V(t_1 T t_3)}$. 
    \end{itemize}
    The vertex sets~$V_t$ for~${t \in V(T)}$ are called the \emph{parts} of the tree-decomposition~${(T,\mathcal{V})}$. 
\end{definition}

\begin{definition}[Tree-width]
    Suppose~${(T,\mathcal{V})}$ is a tree-decomposition of a graph~$G$. 
    The \emph{width} of~${(T,\mathcal{V})}$ is the number
    ${\sup\, \{ |V_t|-1 \colon t \in V(T) \} \in \Nbb_0 \cup \{\infty\}}$.
    The \emph{tree-width} of a graph~$G$ is the least width of any tree-decomposition of~$G$.
\end{definition}

\section{Extensive tree-decompositions and self minors}
\label{s:extensive}
The purpose of this section is to explain the extensive tree-decompositions mentioned in the proof sketch. 
Some ideas motivating this definition are already present in Andreae's proof that locally finite trees are ubiquitous under the topological minor relation~\cite[Lemma~2]{A79}.

\subsection{Extensive tree-decompositions}
\begin{definition}[Separations induced by tree-decompositions]
    Given a tree-decomposition ${(T,\mathcal{V})}$ of a graph~$G$, and an edge~${e \in E(T)}$, let 
    \begin{itemize}
        \item ${A(e) := \bigcup \{ V_{t'} \colon t' \notin V(T_{e^+}) \}}$;
        \item ${B(e) := \bigcup \{ V_{t'} \colon t' \in V(T_{e^+}) \}}$; 
        \item ${S(e) := A(e) \cap B(e)=V_{e^-}\cap V_{e^+}}$.
    \end{itemize}
    Then~${\left( A(e),B(e) \right)}$ is a separation of~$G$ (cf.~\cite[Chapter 12.3.1]{D16}). 
    We call~${B(e)}$ the \emph{bough of~${(T,\mathcal{V})}$ rooted in~$e$} and~${S(e)}$ the \emph{separator of~${B(e)}$}.  \
    When writing $\overline{G[B(e)]}$ it is implicitly understood that this refers to the separation~${\left( A(e),B(e) \right)}$ (cf.~Definition~\ref{def_separation}.)
\end{definition}

\begin{definition}
    \label{defn_treestuff}
    Let~${(T,\mathcal{V})}$ be a tree-decomposition of a graph~$G$. 
    For a subtree~${S \subseteq T}$, let us write
    \[
        G(S) = G\left[\bigcup_{t\in V(S)} V_t\right]
    \]
    and, if~$H$ is an~$IG$, we write~${H(S) = H(G(S))}$ for the restriction of~$H$ to~${G(S)}$.
\end{definition}

\begin{definition}[Self-similar bough]
    \label{d:selfsimilarbough} 
    Let~${(T, \mathcal{V})}$ be a tree-decomposition of a graph~$G$. 
    Given~${e \in E(T)}$, the bough~${B(e)}$ is called \emph{self-similar} (towards an end~$\omega$ of~$G$), 
    if there is a family ${\mathcal{R}_e = ( R_{e,s} \colon s \in S(e))}$ of disjoint $\omega$-rays in~$G$ such that for all~${n \in \Nbb}$ there is an edge~${e' \in E(T_{e^+})}$ with~${\dist(e^-,e'^-) \geq n}$ such that 
    \begin{itemize}
        \item for each~${s \in S(e)}$, the ray~$R_{e,s}$ starts in~$s$ and meets~${S(e')}$;
        \item there is a subgraph~${W \subseteq G[B(e')]}$ which is an inflated copy of~$\overline{G[B(e)]}$;
        \item for each~${s \in S(e)}$, we have~${V(R_{e,s}) \cap S(e') \subseteq W(s)}$. 
    \end{itemize}
    Such a~$W$ is called a \emph{witness for the self-similarity of~${B(e)}$ (towards an end~$\omega$ of~$G$) of distance at least~$n$}. 
\end{definition}

\begin{definition}[Extensive tree-decomposition]
    \label{d:extensive}
    A tree-decomposition~${( T, \mathcal{V} )}$ of~$G$ is \emph{extensive} if 
    \begin{itemize}
        \item $T$ is a locally finite, rooted tree;
        \item each part of~${(T, \mathcal{V})}$ is finite; 
        \item every vertex of~$G$ appears in only finitely many parts of~$\mathcal{V}$; 
        \item for each~${e \in E(T)}$, the bough~${B(e)}$ is self-similar towards some end~$\omega_e$ of~$G$.
    \end{itemize}
\end{definition}

\begin{remark}
    If ${( T, \mathcal{V} )}$ is extensive then, for each edge~${e \in E(T)}$ and every~${n \in \Nbb}$, there is an an edge~${e' \in E(T_{e^+})}$ with~${\dist(e^-,e'^-) \geq n}$, such that~${G[B(e')]}$ contains a witness for the self-similarity of~${B(e)}$. 
    Since~$T$ is locally finite, there is some ray~$R_e$ in~$T$ such that there are infinitely many such~$e'$ on~$R_e$.
\end{remark}

The following is the main result of this paper.

\begin{theorem}
    \label{t:nice}
    Every locally finite connected graph admitting an extensive tree-decomposi{-}tion is $\preceq$-ubiquitous.
\end{theorem}

\subsection{Self minors in extensive tree-decompositions}

The existence of an extensive tree-decomposition of a graph~$G$ will imply the existence of many self-minors of~$G$, which will be essential to our proof. 

Throughout this subsection, let~$G$ denote a locally finite, connected graph with an extensive tree-decomposition~${(T, \mathcal{V})}$.

\begin{definition}
    \label{d:amalgamation}
    Let~${(A,B)}$ be a separation of~$G$ with~${A \cap B = \{v_1,v_2,\ldots,v_n \}}$. 
    Suppose ${H_1, H_2}$ are subgraphs of a graph~$\Gamma$, where 
    $H_1$ is an inflated copy of~${G[A]}$, 
    $H_2$ is an inflated copy of~$\overline{G[B]}$,
    and for all vertices~${x \in A}$ and~${y \in B}$, we have~${H_1(x) \cap H_2(y) \neq \emptyset}$ only if~${x = y = v_i}$ for some~$i$. 
    Suppose further that~$\mathcal{P}$ is a family of disjoint paths ${(P_i \colon i \in [n])}$ in~$\Gamma$ such that each~$P_i$ is a path from~${H_1(v_i)}$ to~${H_2(v_i)}$, which is otherwise disjoint from~${H_1 \cup H_2}$. 
    Note that~$P_i$ may be a single vertex if~${H_1(v_i) \cap H_2(v_i) \neq \emptyset}$. 
    
    We write~${H_1 \oplus_{\Pcal} H_2}$ for the~$IG$ given by~${(H,\phi)}$, where~${H = H_1 \cup H_2 \cup \bigcup_{i\in [n]} P_i}$ and
    \[
        H(v) := 
        \begin{cases} 
            H_1(v_i) \cup V(P_i) \cup H_2(v_i) & \text{ if } v = v_i \in A \cap B,\\
            H_1(v) & \text{ if } v \in A \setminus B,\\
            H_2(v) & \text{ if } v \in B \setminus A.
        \end{cases}
    \]
    We note that this may produce a non-tidy~$IG$, in which case in practise (in order to maintain our assumption that each~$IG$ we consider is tidy) we will always delete some edges inside the branch sets to make it tidy.
    
    We will often use this construction when the family $\mathcal{P}$ consists of certain segments of a family of disjoint rays $\Rcal$. If $\Rcal$ is such that each~$R_i$ has its first vertex in~${H_1(v_i)}$ and is otherwise disjoint from~$H_1$, and such that every~$R_i$ meets~$H_2$, and does so first in some vertex~${x_i \in H(v_i)}$, then we write 
    \[ 
        H_1 \oplus_\Rcal H_2 = H_1 \oplus_{(R_ix_i \colon i \in [n])} H_2.
    \]
\end{definition}

\begin{definition}[Push-out]
    \label{d:pushout}
        A self minor ${G' \subseteq G}$ (meaning~$G'$ is an~$IG$) is called a \emph{push-out of~$G$ along~$e$ to depth~$n$} for some~${e \in E(T)}$ if there is an edge~${e' \in T_{e^+}}$ such that ${\dist(e^-,e'^-) \geq n}$ and a subgraph ${W \subseteq G[B(e')]}$, which is an inflated copy of~$\overline{G[B(e)]}$, such that
        ${G' = G[A(e)] \oplus_{\Rcal_e} W}$.

    Similarly, if~$H$ is an~$IG$, then a subgraph~$H'$ of~$H$ is a \emph{push-out of~$H$ along~$e$ to depth~$n$} for some ${e \in E(T)}$ if there is an edge~${e' \in T_{e^+}}$ such that~${\dist(e^-,e'^-) \geq n}$ and a subgraph~${W \subseteq H(G[B(e')])}$, which is an inflated copy of~$\overline{G[B(e)]}$, such that
    \[
         H' = H(G[A(e)]) \oplus_{H^{\downarrow}(\Rcal_e)} W.
    \]
    
    Note that if~$G'$ is a push-out of~$G$ along~$e$ to depth~$n$, then~${H(G')}$ has a subgraph which is a push-out of~$H$ along~$e$ to depth~$n$.
\end{definition}

\begin{lemma}
    \label{lem:pushout2}
    For each~${e \in E(T)}$, each~${n \in \Nbb}$, and each witness~$W$ of the self-similarity of~${B(e)}$ of distance at least~$n$ there is a corresponding push-out 
    ${G_W := G[A(e)] \oplus_{\mathcal{R}_e} W}$
    of~$G$ along~$e$ to depth~$n$. 
\end{lemma}

\begin{proof}
    Let~${e' \in E(T_{e^+})}$ be the edge in Definition~\ref{d:selfsimilarbough} such that~${W \subseteq G[B(e')]}$. 
    By Definition~\ref{d:selfsimilarbough},
    each ray~$R_{e,s}$
    meets~${S(e')}$ and ${R_{e,s} \cap S(e') \subseteq W(s)}$. 
    Hence, the initial segment of~$R_{e,s}$ up to the first point in~$W$ only meets~${G[A(e)] \cup W}$ in~${\{s\} \cup W(s)}$. 
    Now, if ${s' \in S(e) \cap W(s)}$ for some~$s'$, then~${s' \in S(e')}$, and so~${R_{e,s'} \cap S(e') \not\subseteq W(s')}$, contradicting Definition~\ref{d:selfsimilarbough}. 
    
    Since~${G[A(e)]}$ is an~${IG[A(e)]}$ and~$W$ is an inflated copy of~$\overline{G[B(e)]}$, by Definitions~\ref{d:amalgamation} and~\ref{d:pushout} ${G[A(e)] \oplus_{\Rcal_e} W}$ is well-defined and is indeed a push-out of~$G$ along~$e$ to depth~$n$.
\end{proof}

The existence of push-outs of~$G$ along~$e$ to arbitrary depths is in some sense the essence of extensive tree-decompositions, and lies at the heart of our inductive construction in Section~\ref{sec:countable-subtrees}.

\section{Existence of extensive tree-decompositions}
\label{s:getnice}

The purpose of this section is to examine two classes of locally finite connected graphs that have extensive tree-decompositions: 
Firstly, the class of graphs with finitely many ends, all of which are thin, and secondly the class of graphs of finite tree-width. 
In both cases we will show the existence of extensive tree-decompositions using some results about the \emph{well-quasi-ordering} of certain classes of graphs.

A \emph{quasi-order} is a reflexive and transitive binary relation, such as the minor relation between graphs. 
A quasi-order~$\preceq$ on a set~$X$ is a \emph{well-quasi-order} if for all infinite sequences~${(x_i)_{i \in \Nbb}}$ with~${x_i \in X}$ for every~${i \in \Nbb}$ there exist~${i, j \in \Nbb}$ with~${i < j}$ such that~${x_i \preceq x_j}$. 
The following two consequences will be useful.

\begin{remark}
    \label{r:increasingsubsequence}
    A simple Ramsey type argument shows that if~$\preceq$ is a well-quasi-order on~$X$, then every infinite sequence~${(x_i)_{i \in \Nbb}}$ with~${x_i \in X}$ for every~${i \in \Nbb}$ contains an increasing infinite subsequence~${x_{i_1},x_{i_2}, \ldots \in X}$. 
    That is, an increasing infinite sequence~${i_1<i_2<\ldots}$ such that~${x_{i_j} \preceq x_{i_k}}$ for all~${j < k}$.
    
    Also, it is simple to show that if~$\preceq$ is a well-quasi-order on~$X$, then for every infinite sequence~${(x_i)_{i \in \Nbb}}$ with~${x_i \in X}$ for every~${i \in \Nbb}$ there is an~${i_0 \in \Nbb}$ such that for every~${i \geq i_0}$ there are infinitely many~${j \in \Nbb}$ with~${x_i \preceq x_j}$. 
\end{remark}

A famous result of Robertson and Seymour~\cite{RS04}, proved over a series of 20 papers, shows that finite graphs are well-quasi-ordered under the minor relation. 
Thomas~\cite{T89} showed that for any~${k \in \Nbb}$ the class of graphs with tree-width at most~$k$ and arbitrary cardinality is well-quasi-ordered by the minor relation.

We will use slight strengthenings of both of these results, Lemma~\ref{lem:labelled-wqo} and Lemma~\ref{lem:labelled-wqo-btw}, to show that our two classes of graphs admit extensive tree-decompositions.

In Section~\ref{s:WQO} we will discuss in more detail the connection between our proof and well-quasi-orderings, and indicate how stronger well-quasi-ordering results could be used to prove the ubiquity of larger classes of graphs.

\subsection{Finitely many thin ends}

We will consider the following strengthening of the minor relation. 

\begin{definition}
    Given~${\ell \in \Nbb}$, an $\ell$-\emph{pointed graph} is a graph~$G$ together with a function~${\pi \colon [\ell] \to V(G)}$, called a \emph{point function}. 
    For $\ell$-pointed graphs~${(G_1, \pi_1)}$ and~${(G_2, \pi_2)}$, we say~${(G_1, \pi_1) \preceq_p (G_2, \pi_2)}$ if~${G_1 \preceq G_2}$ and this can be arranged in such a way that~${\pi_2(i)}$ is contained in the branch set of~${\pi_1(i)}$ for every~${i \in [\ell]}$. 
\end{definition}

\begin{lemma}
    \label{lem:labelled-wqo}
    For~${\ell \in \Nbb}$ the set of $\ell$-pointed finite graphs is well-quasi-ordered under the relation~$\preceq_p$.
\end{lemma}

\begin{proof}
    This follows from a stronger statement of Robertson and Seymour in~\cite[1.7]{RS10}. 
\end{proof}

We will also need the following structural characterisation of locally finite one-ended graphs with a thin end due to Halin.

\begin{lemma}[{\cite[Satz~$3'$]{H65}}]\label{lem:oneended-td}
	Every one-ended, locally finite connected graph~$G$ with a thin end of degree~${k \in \Nbb}$ has a tree-decomposition~${(R,\mathcal{V})}$ of~$G$ such that ${R = t_0t_1t_2\dots}$ is a ray, and for every~${i \in \Nbb_0}$:
    \begin{itemize}
        \item ${|V_{t_i}|}$ is finite;
        \item ${|S(t_{i}t_{i+1})| = k}$; 
        \item ${S(t_{i}t_{i+1}) \cap S(t_{i+1}t_{i+2}) = \emptyset}$. 
   \end{itemize}
\end{lemma}

\begin{remark}
    \label{rem:oneended-td}
    Note that in the above lemma, for a given finite set~${X \subseteq V(G)}$, by taking the union over parts corresponding to an initial segment of the ray of the decomposition, one may always assume that~${X \subseteq V_{t_0}}$. 
    Moreover, note that since~${S(t_{i}t_{i+1}) \cap S(t_{i+1}t_{i+2}) = \emptyset}$, it follows that every vertex of~$G$ is contained in at most two parts of the tree-decomposition.
\end{remark}

\begin{lemma}
    \label{lem:one-ended-nice}
    Every one-ended, locally finite connected graph~$G$ with a thin end has an extensive tree-decomposition~${( R, \mathcal{V} )}$ where~${R = t_0t_1t_2\ldots}$ is a ray 
    rooted in its initial vertex. 
\end{lemma}

\begin{proof}
    Let ${k \in \Nbb}$ be the degree of the thin end of~$G$ and let~${\Rcal = (R_j \colon j \in [k])}$ be a maximal family of disjoint rays in~$G$. 
    Let~${(R',\mathcal{W})}$ be the tree-decomposition of~$G$ given by Lemma~\ref{lem:oneended-td} where~${R' = t_0't_1'\ldots}$. 
    
    By Remark~\ref{rem:oneended-td} (and considering tails of rays if necessary), we may assume that each ray in~$\Rcal$ starts in~${S(t'_0t'_1)}$. 
    Note that each ray in~$\Rcal$ meets the separator~${S(t'_{i-1}t'_i)}$ for each~${i \in \Nbb}$. 
    Since~$\Rcal$ is a family of~$k$ disjoint rays and~${|S(t'_{i-1}t'_i)| = k}$ for each~${i \in \Nbb}$, each vertex in~${S(t'_{i-1}t'_i)}$ is contained in a unique ray in~$\Rcal$. 
    
    Let~${\ell = 2k}$ and consider a sequence~${(G_i,\pi_i)_{i\in \Nbb}}$ of $\ell$-pointed finite graphs defined by ${G_i := G[W_{t'_i}]}$ and 
    \[
        \pi_i \colon [\ell] \to V(G_i) , \; j \mapsto 
        \begin{cases}
            \text{ the unique vertex in } S(t'_{i-1}t'_i) \cap V(R_j)  & \text{ for } 1 \leq j \leq k , \\
            \text{ the unique vertex in } S(t'_it'_{i+1}) \cap V(R_{j-k}) & \text{ for } k < j \leq 2k = \ell.
        \end{cases}
    \]
    
    By Lemma~\ref{lem:labelled-wqo} and Remark~\ref{r:increasingsubsequence} there is an~${n_0 \in \Nbb}$ such that for every~${n \geq n_0}$ there are infinitely many~${m > n}$ with~${(G_n,\pi_n) \preceq_p (G_m,\pi_m)}$. 

    Let~${V_{t_0} := \bigcup_{i=0}^{n_0} W_{t'_i}}$ and~${V_{t_i} := W_{t'_{n_0 + i}}}$ for all~${i \in \Nbb}$. 
    We claim that ${(R, (V_{t_i} \colon i\in \Nbb_0))}$ is the desired extensive tree-decomposition of~$G$ where~${R = t_0t_1t_2\ldots}$ is a ray with root~$t_0$.
    The ray~$R$ is a locally finite tree and all the parts are finite. 
    Moreover, every vertex of~$G$ is contained in at most two parts by Remark~\ref{rem:oneended-td}. 
    It remains to show that for every~${i \in \Nbb}$, the bough~${B(t_{i-1}t_i)}$ is self-similar towards the end of~$G$. 

    Let~${e = t_{i-1}t_i}$ for some~${i \in \Nbb}$. 
    For each~${s \in S(e)}$, we let~${p(s) \in [k]}$ be such that~${s \in R_{p(s)}}$ and set~${R_{e,s} = sR_{p(s)}}$. 
    We wish to show there is a witness~$W$ for the self-similarity of~${B(e)}$ of distance at least~$n$ for each~${n \in \Nbb}$. 
    Note that~${B(e) = \bigcup_{j \geq 0} V(G_{n_0+i+j})}$. 
    By the choice of~$n_0$ in Remark~\ref{r:increasingsubsequence}, there exists an~${m > i + n}$ such that~${(G_{n_0+i},\pi_{n_0+i}) \preceq_p (G_{n_0+m},\pi_{n_0+m})}$. 
    Let~${e' = t_{m-1}t_{m}}$. 
    We will show that there exists a~${W \subseteq G[B(e')]}$ witnessing the self-similarity of~${B(e)}$ towards the end of~$G$.
    
    Recursively, for each~${j \geq 0}$ we can find~${m = m_0 < m_1 < m_2 < \cdots}$ with 
    \[
        (G_{n_0+i+j},\pi_{n_0+i+j}) \preceq_p (G_{n_0+m_j},\pi_{n_0+m_j}).
    \]
    In particular, there are subgraphs~${H_{m_j} \subseteq G_{n_0+m_j}}$ which are inflated copies of~$G_{n_0+i+j}$, all compatible with the point functions, and so
    \[
        {S(t'_{n_0+m_j -1}t'_{n_0+m_j}) \cup S(t'_{n_0+m_j}t'_{n_0+m_j+1}) \subseteq H_{m_j}}
    \]
    for each~${j \geq 0}$. 
    
    Hence, for every~${j \in \Nbb}$ and~${p \in [k]}$ there is a unique $H_{m_{j-1}}$--$H_{m_{j}}$ subpath~$P_{p,j}$ of~$R_p$. 
    We claim that  
    \[
        W' := \bigcup_{j \geq 0 } H_{m_j} \cup \bigcup_{j \in \Nbb }\bigcup_{p \in [k]} P_{p,j}
    \]
    is a subgraph of~${G[B(e')]}$ that is an~${IG[B(e)]}$. 
    Hence, the desired~$W$ can be obtained as a subgraph of~$W'$. 
    
    To prove this claim it is sufficient to check that for each~${j \in \Nbb}$ and each~${s \in S(t_{j-1}t_{j})}$, the branch sets of~$s$ in~$H_{j-1}$ and in~$H_{j}$ are connected by~$P_{p(s),j}$. 
    Indeed, by construction, every~$P_{p,j}$ is a path from~${\pi_{n_0+m_{j-1}}(k+p)}$ to~${\pi_{n_0+m_{j}}(p)}$. 
    And, since the~$H_{m_j}$ are pointed minors of~$G_{n_0+m_j}$, it follows that~${\pi_{n_0+m_{j-1}}(k+p(s)) \in H_{m_{j-1}}(s)}$ and~${\pi_{n_0+m_{j}}(p(s)) \in H_{m_{j}}(s)}$ are as desired.
    
    Finally, since~${(G_{n_0+i},\pi_{n_0+i})~\preceq_p~(G_{n_0 + m},\pi_{n_0 + m})}$ as witnessed by~$H_{m_0}$, the branch set of each~${s \in S(t_{i-1}t_i)}$ must indeed include~${V(R_{e,s}) \cap S(e')}$. 
\end{proof}

\begin{lemma}
    \label{lem:finendsisext}
    If~$G$ is a locally finite connected graph with finitely many ends, each of which is thin, then~$G$ has an extensive tree-decomposition.
\end{lemma}

\begin{proof}
    Let~${\Omega(G) = \{\omega_1,\ldots ,\omega_n\}}$ be the set of the ends of~$G$. 
    Let~${X \subseteq V(G)}$ be a finite set of vertices which separates the ends of~$G$, i.e.~so that all~${C_i = C(X,\omega_i)}$ are pairwise disjoint. 
    Without loss of generality, we may assume that~${V(G) = X \cup \bigcup_{i \in [n]} C_i}$. 
    
    Let ${G_i := G[C_i \cup X]}$. 
    Then each~$G_i$ is a locally finite connected one-ended graph, with a thin end~$\omega_i$, and hence by Lemma~\ref{lem:one-ended-nice} each of the~$G_i$ admits an extensive tree-decomposition~${(R^i,\V^i)}$, 
    where~$R^i$ is rooted in its initial vertex~$r^i$.
    Without loss of generality, ${X \subseteq V^i_{r^i}}$ for each~${i \in [n]}$.
    
    Let~$T$ be the tree formed by identifying the family of rays ${(R^i \colon i \in [n])}$ at their roots, let~$r$ be this identified vertex which we consider to be the root of~$T$, and let~${(T,\mathcal{V})}$ be the tree-decompositions whose root part is~${\bigcup_{i \in [n]} V^i_{r^i}}$, and which otherwise agrees with the~${(R^i,\V^i)}$. 
    It is a simple check that~${(T,\mathcal{V})}$ is an extensive tree-decomposition of~$G$. 
\end{proof}

\subsection{Finite tree-width}

\begin{definition}
    A rooted tree-decomposition~${(T,\V)}$ of~$G$ is \emph{lean} if for any~${k \in \Nbb}$, any nodes~${t_1,t_2 \in V(T)}$, and any~${X_1 \subseteq V_{t_1}, X_2 \subseteq V_{t_2}}$ such that~${|X_1|,|X_2| \geq k}$ there are either~$k$ disjoint paths in~$G$ between~$X_1$ and~$X_2$, or there is a vertex~$t$ on the path in~$T$ between~$t_1$ and~$t_2$ such that~${|V_t| < k}$. 
\end{definition}

\begin{remark}
    \label{rem:lean-tw}
    K{\v{r}}{\'\i}{\v{z}} and Thomas~\cite{KT91} showed that if~$G$ has tree-width at most~$m$ for some~${m \in \Nbb}$, then~$G$ has a lean tree-decomposition of width at most~$m$. 
\end{remark}

\begin{lemma}
    \label{lem:bough-connected}
    Let~$G$ be a locally finite connected graph and
    let~${(T,\V)}$ be a lean tree-decomposition of~$G$ 
    of width at most~$m$. 
    Then there exists a lean tree-decomposition of~$G$ of width at most~$m$ such that every bough is connected and the decomposition tree is locally finite. 
    Moreover, we may assume that every vertex appears in only finitely many parts.
\end{lemma}

\begin{proof}
    We begin by defining the underlying tree~$T'$ of this decomposition. 
    The root of~$T'$ will be the root~$r$ of~$T$, and the other vertices will be pairs~${(e, C)}$ where~$e$ is an edge of~$T$ and~$C$ is a component of~${G - S(e)}$ meeting (or equivalently, included in)~${B(e)}$. 
    There is an edge from~$r$ to~${(e, C)}$ whenever~${e^- = r}$, and from~${(e, C)}$ to~${(f, D)}$ whenever~${f^- = e^+}$ and~${D \subseteq C}$. 
    For future reference, we define a graph homomorphism~$\pi$ from~$T'$ to~$T$ by setting~${\pi(r) = r}$ and~${\pi(e, C) = e^+}$. 
    Next, we set~${V'_r := V_r}$ and 
    \[
        V'_{(e,C)} :=  V_{e^+} \cap (V(C) \cup N(V(C))),
    \]
    where~$N(V(C))$ is the neighbourhood of~$V(C)$. 
    Moreover, we let~$\mathcal{V}'$ denote the family of all~$V'_p$ for all nodes~$p$ of~$T'$. 
    
    To see that~$T'$ is locally finite, note that for any child ${(e,C)}$ of~$p$ the set~$C$ is also a component of~${G \setminus V_{\pi(p)}}$ and that no two distinct children yield the same component; 
    if~${(e, C)}$ and~${(f,C)}$ were distinct children of~$p$, then we would have~${V(C) \subseteq B(f) \subseteq A(e)}$ and so~${V(C) \subseteq A(e) \cap B(e) = S(e)}$, which is impossible.
    
    We now analyse, for a given vertex~$v$ of~$G$, which of the sets~$V'_p$ contain~$v$. 
    Since~${(T, \mathcal{V})}$ is a tree-decomposition, $T$ induces a subtree on the set of nodes~$t$ of~$T$ with~${v \in V_t}$, and so this set has a minimal element~$t_v$ in the tree order. 
    We set~${p_v := r}$ if~${t_v = r}$ and otherwise set~${p_v := (e, C)}$, where~$e$ is the unique edge of~$T$ with~${e^+ = t_v}$ and~$C$ is the unique component of~${G - S(e)}$ containing~$v$. 
    This guarantees that~${v \in V'_{p_v}}$. 
    For any other node~$p$ of~$T'$ with~${v \in V'_p}$, we have~${p \neq r}$ and so~$p$ has the form~${(e,C)}$. 
    Since~${v \in V_{e^+}}$ and~${p \neq p_v}$, it follows that~$e^-$ lies on the path from~$t_v$ to~$e^+$ and so~${v \in V_{e^-}}$, from which~${v \in N(V(C))}$ follows. 
    Thus, some neighbour~$w$ of~$v$ lies in~$C$. 
    Then~${w \in B(e) \setminus S(e) = B(e) \setminus A(e)}$ and so~$t_w$ lies in~$T_{e^+}$. 
    That is,~$p$ lies on the path from~$p_v$ to~$p_w$. 
    Conversely, for any~${p = (e,C)}$ on this path we have~${w \in V(C)}$ and so~${v \in N(V(C)) \subseteq S(e) \subseteq V_{e^+}}$, so that~${v \in V'_p}$. 
    
    What we have shown is that~$v$ is in~$V'_p$ precisely when~${p = p_v}$ or there is some neighbour~$w$ of~$v$ in~$G$ such that~$p$ lies on the path in~$T'$ from~$p_v$ to~${p_w \in V(T'_{p_v})}$. 
    Using this information, it is easy to deduce that~${(T', \mathcal{V}')}$ is a tree-decomposition: 
    A vertex~$v$ is in~$V'_{p_v}$ and an edge~$vw$ with~$p_v$ no higher (in the tree order) than~$p_w$ in~$T$ is also in~$V'_{p_v}$. 
    The third condition in the definition of tree-decompositions follows from the fact that the $T'$ induces a subtree on the set of all nodes~$p$ with~${v \in V'_p}$. 
    These sets are also all finite, since~$G$ is locally finite.
    
    Next we examine the boughs of this decomposition. 
    Let~${f \in E(T')}$ with~${f^+ = (e, C)}$. 
    Our aim is to show that~${B(f) = V(C) \cup N(V(C))}$. 
    For any~${(e', C') \in V(T'_{f^+})}$, we have ${V'_{(e', C')} \subseteq V(C') \cup N(V(C')) \subseteq V(C) \cup N(V(C))}$, so that~${B(f) \subseteq V(C) \cup N(V(C))}$. 
    For ${v \in V(C)}$, we have~${p_v \in V(T_{f^+})}$ and so~${v \in B(f)}$ and for~${v \in N(V(C))}$, there is a neighbour~$w$ of~$v$ such that~$f^+$ lies on the path from~$p_v$ to~$p_w$, yielding once more that~${v \in B(f)}$. 
    This completes the proof that~${B(f) = V(C) \cup N(V(C))}$, and in particular~${B(f)}$ is connected. 
    
    Since~$G$ is locally finite, for each~$e$, there are only finitely many components of~${G - V_{e^-}}$, so that~$T'$ is also locally finite. 
    The final thing to show is that this decomposition is lean. 
    So, suppose we have~${X_1 \subseteq V'_{p_1}}$ and~${X_2 \subseteq V'_{p_2}}$ with ${|X_1|, |X_2| \geq k}$. 
    Then also~${X_1 \subseteq V_{\pi(p_1)}}$ and~${X_2 \subseteq V_{\pi(p_2)}}$, so that if there are no~$k$ disjoint paths from~$X_1$ to~$X_2$ in~$G$, then there is some~$t$ on the path from~$\pi(p_1)$ to~$\pi(p_2)$ in~$T$ with~${|V_t| \leq k}$. 
    But then there is some~$p$ on the path from~$p_1$ to~$p_2$ in~$T'$ with~${\pi(p) = t}$ and, since~${V'_p \subseteq V_t}$, we have~${|V'_p| \leq k}$. 
\end{proof}

\begin{lemma}
    \label{lem:labelled-wqo-btw}
    For all~${k,\ell \in \Nbb}$, the class of $\ell$-pointed graphs with tree-width at most~$k$ is well-quasi-ordered under the relation~$\preceq_p$.
\end{lemma}

\begin{proof}
    This is a consequence of a result of Thomas~\cite{T89}.
\end{proof}

\begin{lemma}
    \label{lem:boundwidthisext}
    Every locally finite connected graph of finite tree-width has an extensive tree-decomposition.
\end{lemma}

\begin{proof}
    Let~$G$ be a locally finite connected graph of tree-width~${m \in \Nbb}$. By Remark~\ref{rem:lean-tw}, $G$ has a lean tree-decomposition of width at most $m$ and so, by Lemma~\ref{lem:bough-connected}, there is a lean tree-decomposition~${(T,\V)}$ of~$G$ with width~$m$ in which every bough is connected, every vertex is contained in only finitely many parts, and such that~$T$ is a locally finite tree with root~$r$.
    
    Let~$\epsilon$ be an end of~$T$ and let~$R$ be the unique $\epsilon$-ray starting at the root of~$T$. 
    Let ${d_\epsilon = \liminf_{e \in R} |S(e)|}$ and fix a tail ${t^\epsilon_0t^\epsilon_1\ldots}$ of~$R$ such that ${|S(t^\epsilon_{i-1}t^\epsilon_i)| \geq d_\epsilon}$ for all~${i \in \Nbb}$. 
    Note that, ${|S(t^\epsilon_{i_k-1}t^\epsilon_{i_k})| = d_\epsilon}$ for an infinite sequence~${i_1<i_2<\cdots}$ of indices.
    
    Since~${(T,\V)}$ is lean, there are~$d_\epsilon$ disjoint paths between~${S(t^\epsilon_{i_k-1}t^\epsilon_{i_k})}$ and~${S(t^\epsilon_{i_{k+1}-1}t^\epsilon_{i_{k+1}})}$ for every~${k \in \Nbb}$. 
    Moreover, since each ${S(t^\epsilon_{i_k-1}t^\epsilon_{i_k})}$ is a separator of size~$d_\epsilon$, these paths are all internally disjoint. 
    Hence, since every vertex appears in only finitely many parts, by concatenating these paths we get a family of~$d_\epsilon$ many disjoint rays in~$G$. 
    
    Fix one such family of rays~${( R^\epsilon_j \colon j \in [d_\epsilon] )}$. 
    We claim that there is an end~$\omega$ of~$G$ such that~${R^\epsilon_j \in \omega}$ for all~${j \in [d_\epsilon]}$. 
    Indeed, if not, then there is a finite vertex set~$X$ separating some pair of rays~$R$ and~$R'$ from the family. 
    However, since each vertex appears in only finitely many parts, there is some~${k \in \Nbb}$ such that~${X \cap V_t = \emptyset}$ for all~${t \in V(T_{t^\epsilon_{i_k-1}})}$. 
    By construction~$R$ and~$R'$, have tails in~${B(t^\epsilon_{i_{k}-1}t^\epsilon_{i_{k}})}$, which is connected and disjoint from~$X$, contradicting the fact that~$X$ separates~$R$ and~$R'$.
    
    For every~${k \in \Nbb}$, we define a point function ${\pi^\epsilon_{i_k} \colon [d_\epsilon] \to S(t^\epsilon_{i_k-1}t^\epsilon_{i_k})}$ by letting~$\pi^\epsilon_{i_k}(j)$ be the unique vertex in~${V(R^\epsilon_j) \cap S(t^\epsilon_{i_k-1}t^\epsilon_{i_k})}$. 
    
    By Lemma~\ref{lem:labelled-wqo-btw} and Remark~\ref{r:increasingsubsequence}, the sequence ${(G[B(t^\epsilon_{i_k-1}t^\epsilon_{i_k})], \pi^\epsilon_{i_k})_{k\in\Nbb}}$ has an increasing subsequence ${(G[B(t^\epsilon_{i-1}t^\epsilon_i)], \pi^\epsilon_i)_{i\in I_\epsilon}}$, i.e.~there exists an~${I_{\epsilon} \subseteq \{ i_k \colon k \in \Nbb \}}$ such that for any~${k,j \in I_\epsilon}$ with~${k < j}$, we have 
    \[
        (G[B(t^\epsilon_{k-1}t^\epsilon_k)], \pi^\epsilon_k) \preceq_p (G[B(t^\epsilon_{j-1}t^\epsilon_j)], \pi^\epsilon_j).
    \]
    Let us define ${F_\epsilon = \{ t^\epsilon_{k-1}t^\epsilon_{k} \colon k \in I_\epsilon \} \subseteq E(T)}$. 
    
    Consider~${T^- = T - \bigcup_{\epsilon\in\Omega(T)} F_\epsilon}$, and let us write~${\mathcal{C}(T^-)}$ for the components of~$T^-$. 
    We claim that every component~${C \in \mathcal{C}(T^-)}$ is a locally finite rayless tree, and hence finite. 
    Indeed, if~$C$ contains a ray~${R \subseteq T}$, then~$R$ is in an end~$\epsilon$ of~$T$ and hence~${F_\epsilon \cap R \neq \emptyset}$, a contradiction. 
    Consequently, each set~${\bigcup_{t\in C} V_t}$ is finite. 
    
    Let us define a tree-decomposition~${(T', \V')}$ of~$G$ with ${T' = T / \mathcal{C}(T^-)}$, that is where we contract each component $C \in \mathcal{C}(T^-)$ to a single vertex and where~${V'_{t'} = \bigcup_{t\in t'} V_t }$. 
    We claim this is an extensive tree-decomposition. 
    
    Clearly $T'$ is a locally finite tree, each part of~${(T', \V')}$ is finite, and every vertex of~$G$ in contained in only finitely many parts of the tree-decomposition. 
    Given~${e \in E(T')}$, there is some~${\epsilon \in \Omega(T)}$ such that~${e \in F_\epsilon}$. 
    Consider the family of rays ${(R_{e,j} \colon j \in [d_\epsilon] )}$ given by~${R_{e,j} = R^\epsilon_j \cap B(e)}$. 
    Let~$\omega_e$ be the end of~$G$ in which the rays~$R_{e,j}$ lie.
    
    There is some~${k \in \Nbb}$ such that~${e = t^\epsilon_{k-1}t^\epsilon_{k}}$. 
    Given~${n \in \Nbb}$, let~${k' \in I_\epsilon}$ be such that there are at least~$n$ indices~${\ell \in I_\epsilon}$ with~${k < \ell < k'}$, and let~${e' = t^\epsilon_{k'-1}t^\epsilon_{k'}}$. 
    Note that,~${e' \in F_\epsilon}$ and hence~${e' \in E(T')}$. 
    Furthermore, by construction~$e'^-$ has distance at least~$n$ from~$e^-$ in~$T'$. 
    Then, since~${G[ B(e)] = G[B(t^\epsilon_{k-1}t^\epsilon_{k})]}$ and~${G[B(e') ]= G[B(t^\epsilon_{k'-1}t^\epsilon_{k'})]}$, it follows that ${(G[B(e)], \pi^\epsilon_{k}) \preceq_p (G[B(e')], \pi^\epsilon_{k'})}$, and so suitable subgraphs witness the self-similarity of~${B(e)}$ towards~$\omega_e$ with the rays ${(R_{e,j} \colon j \in [d_\epsilon])}$, as in Lemma~\ref{lem:one-ended-nice}.
\end{proof}

\begin{remark}
    If for every~${\ell \in \Nbb}$ the class of $\ell$-pointed locally finite graphs without thick ends is well-quasi-ordered under~$\preceq_p$, then every locally finite graph without thick ends has an extensive tree-decomposition. 
    This follows by a simple adaptation of the proof above.
\end{remark}

\subsection{Sporadic examples}

We note that, whilst Lemmas~\ref{lem:finendsisext} and~\ref{lem:boundwidthisext} show that a large class of locally finite graphs have extensive tree-decompositions, for many other graphs it is possible to construct an extensive tree-decomposition `by hand'. 
In particular, the fact that no graph in these classes has a thick end is an artefact of the method of proof, rather than a necessary condition for the existence of such a tree-decomposition, as is demonstrated by the following examples:

\begin{remark}
    The grid~${\mathbb{Z} \times \mathbb{Z}}$ has an extensive tree-decomposition, which can be seen in Figure~\ref{f:grid}. 
    More explicitly, we can take a ray decomposition of the grid given by a sequence of increasing diamond shaped regions around the origin. 
    It is easy to check that every bough is self-similar towards the end of the grid. 
    
    A similar argument shows that the half-grid has an extensive tree-decomposition. 
    However, we note that both of these graphs were already shown to be ubiquitous in~\cite{BEEGHPTII}. 
\end{remark}

\begin{figure}[ht]
    \center
    \includegraphics[width=0.7\linewidth, trim=15cm 7cm 15cm 7cm, clip]{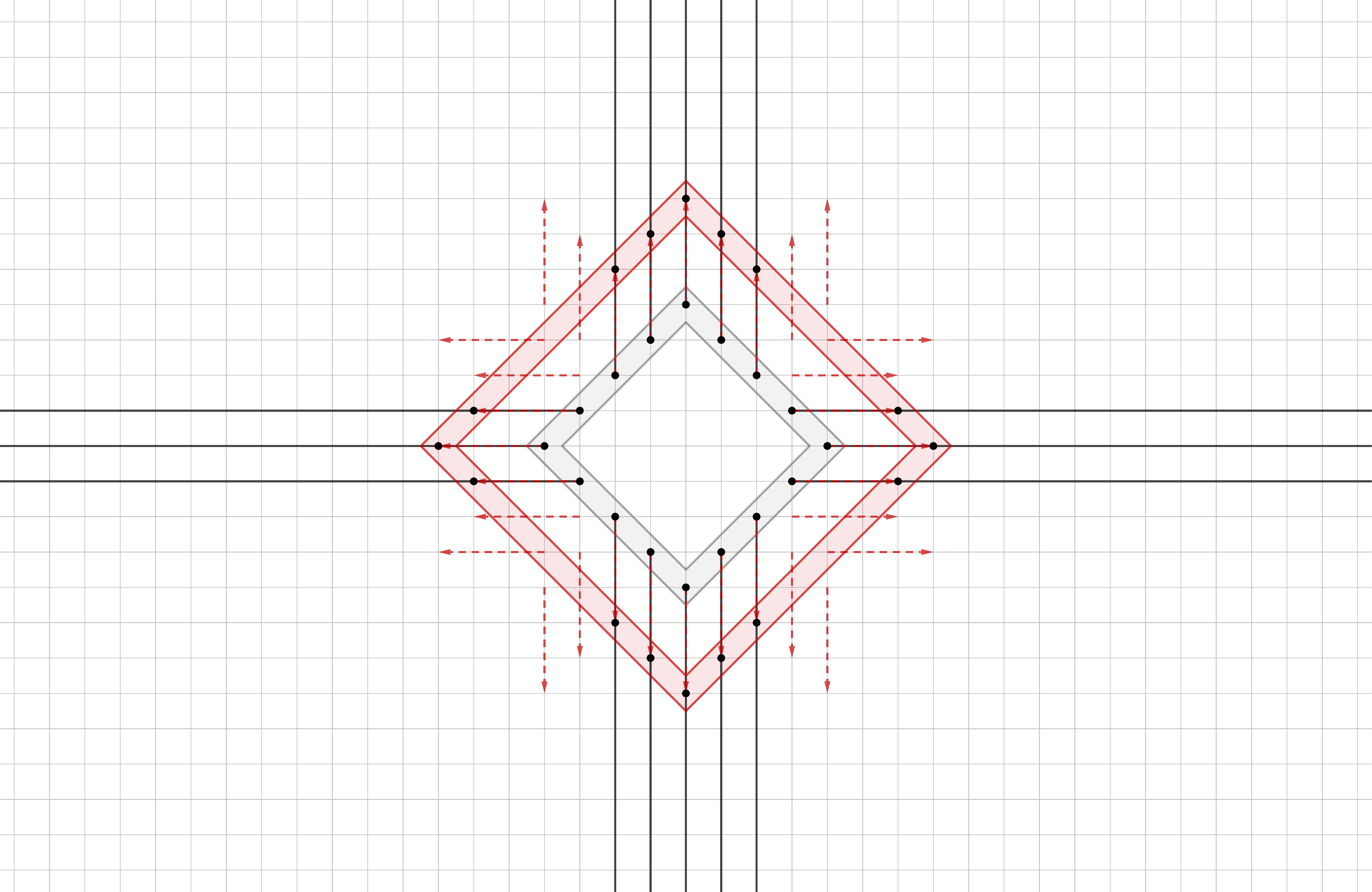}
    \caption{In the grid the boughs are self-similar.}
    \label{f:grid}
\end{figure}

In fact, we do not know of any construction of a locally finite connected graph which does not admit an extensive tree-decomposition.
\begin{question}
    \label{qst:loc_fin_admit ex_td}
    Do all locally finite connected graphs admit an extensive tree-decomposi{-}tion?
\end{question}

\section{The structure of non-pebbly ends}
\label{s:nonpebbly}

We will need a structural understanding of how the arbitrarily large families of~$IG$s (for some fixed graph~$G$) can be arranged inside some host graph~$\Gamma$.
In particular, we are interested in how the rays of these minors occupy a given end~$\epsilon$ of~$\Gamma$.
In~\cite{BEEGHPTII}, by considering a pebble pushing game played on ray graphs, we established a distinction between \emph{pebbly} and \emph{non-pebbly} ends. Furthermore, we showed that each non-pebbly end is either \emph{grid-like} or \emph{half-grid-like}.

\begin{theorem}[{\cite[Theorem 1.2]{BEEGHPTII}}]
    \label{t:trichotomy}
    Let~$\Gamma$ be a graph and let~$\epsilon$ be a thick end of~$\Gamma$. 
    Then~$\epsilon$ is either pebbly, half-grid-like or grid-like.
\end{theorem}

The precise technical definition of such ends is not relevant, in what follows we will simply need to use the following results from~\cite{BEEGHPTII}. 

\begin{corollary}[{\cite[Corollary 5.3]{BEEGHPTII}}]
    \label{c:pebblyubiq}
    Let~$\Gamma$ be a graph with a pebbly end~$\epsilon$ and let~$G$ be a countable graph. 
    Then~${\aleph_0 G \preceq \Gamma}$.
\end{corollary}

\begin{lemma}[{\cite[Lemma 7.1 and Corollary 7.3]{BEEGHPTII}}]
    \label{l:gridstructure}
    Let~$\Gamma$ be a graph with a grid-like end~$\epsilon$. 
    Then there exists an~${N \in \Nbb}$ such that the ray graph for any family~${(R_i \colon i \in I)}$ of disjoint $\epsilon$-rays in~$\Gamma$ with~${|I| \geq N+2}$ is a cycle. 
    
    Furthermore, there is a choice of a cyclic orientation, which we call the \emph{correct orientation}, of each such ray graph such that any transition function between two families of at least~${N + 3}$ disjoint $\epsilon$-rays preserves the correct orientation.
\end{lemma}

\begin{lemma}[{\cite[Lemma 7.6, Corollary 7.7 and Corollary 7.9]{BEEGHPTII}}]
    \label{l:halfgridstructure}
    Let~$\Gamma$ be a graph with a half-grid-like end~$\epsilon$. 
    Then there exists an~${N \in \Nbb}$ such that the ray graph~$K$ for any family~${(R_i \colon i \in I)}$ of disjoint $\epsilon$-rays in~$\Gamma$ with~${|I| \geq N+2}$ contains a bare path with at least~${|I|-N}$ vertices, which we call the \emph{central path} of~$K$, such that the following statements are true:
    \begin{enumerate}
        \item For any ${i \in I}$, if ${K - i}$ has precisely two components, each of size at least~${N+1}$, then~$i$ is an inner vertex of the central path of~$K$. 
        \item There is a choice of an orientation, which we call the \emph{correct orientation}, of the central path of each such ray graph such that any transition function between two families of at least~${N + 3}$ disjoint $\epsilon$-rays sends vertices of the central path to vertices of the central path and preserves the correct orientation.
    \end{enumerate}
\end{lemma}

By Corollary~\ref{c:pebblyubiq}, if we wish to show that a countable graph~$G$ is $\preceq$-ubiquitous we can restrict our attention to host graphs~$\Gamma$ where each end is non-pebbly. 
In which case, by Lemmas~\ref{l:gridstructure} and~\ref{l:halfgridstructure} for any end~$\epsilon$ of~$\Gamma$, the possible ray graphs, and the possible transition functions between two families of rays, are severely restricted. 

Later on in our proof we will be able to restrict our attention to a single end~$\epsilon$ of~$\Gamma$ and the proof will split into two cases according to whether~$\epsilon$ is half grid-like or grid-like. 
However, the two cases are very similar, with the grid-like case being significantly simpler. 
Therefore, in what follows we will prove only the results necessary for the case where~$\epsilon$ is half-grid-like, and then later, in Section~\ref{s:gridlike}, we will shortly 
sketch the differences for the grid-like case.

\subsection{Core rays in the half-grid-like case}
\label{s:core}

By Lemma~\ref{l:halfgridstructure}, in a half-grid-like end~$\epsilon$ every ray graph consists, apart for possibly some bounded number of rays on either end, of a bare-path, each of which comes with a correct orientation, which must be preserved by transition functions.

However, in the half-grid itself even more can be seen to true. 
There is a natural partial order defined on the set of all rays in the half-grid, where two rays are comparable if they have disjoint tails, and a ray~$R$ is less than a ray~$S$ if the tail of~$R$ lies `to the left' of the tail of~$S$ in the half-grid. 
Then it can be seen that the correct orientations of the central path of any disjoint family of rays can be chosen to agree with this global partial order.

In a general half-grid-like end~$\epsilon$ a similar thing will be true, but only for a subset of the rays in the end which we call the core rays. 

Let us fix for the rest of this section a graph~$\Gamma$ and a half-grid-like end~$\epsilon$. 
By Lemma~\ref{l:halfgridstructure}, there is some~${N \in \Nbb}$ such that all but at most $N$ vertices of the ray graph of any large enough family of disjoint $\epsilon$-rays lie on the \emph{central path}. 

\begin{definition}[Core rays]
    \label{d:coreray}
    Let~$R$ be an $\epsilon$-ray. 
    We say~$R$ is a \emph{core ray (of~$\epsilon$)} if there is a finite family~${\mathcal{R} = ( R_i \colon i \in I)}$ of disjoint $\epsilon$-rays with~${R = R_c}$ for some~${c \in I}$ such that~${\RG(\mathcal{R})- c}$ has precisely two components, each of size at least~${N+1}$.
\end{definition}

Note that, by Lemma~\ref{l:halfgridstructure}, such a ray $R_c$ is an inner vertex of the central path of $\RG(\mathcal{R})$. In order to define our partial order on the core rays, we will need to consider what it means for a ray to lie `between' two other rays.

\begin{definition}
    Given three $\epsilon$-rays $R,S,T$ such that $R,S,T$ have disjoint tails, we say that~$S$ \emph{separates~$R$ from~$T$} if the tails of~$R$ and~$T$ disjoint from~$S$ belong to different ends of~${\Gamma-S}$.
\end{definition}

\begin{lemma}
    \label{l:raygraph_separates}
    Let $\mathcal{R}={(R_i \colon i \in I)}$ be a finite family of disjoint $\epsilon$-rays and let~${i_1, i_2, j \in I}$. Then~$i_1$ and~$i_2$ belong to different components of~${\RG(\mathcal{R})-j}$ if and only if~$R_j$ separates~$R_{i_1}$ from~$R_{i_2}$. 
\end{lemma}

\begin{proof}
    Suppose that~$R_{i_1}$ and~$R_{i_2}$ belong to the same end of~${\Gamma-V(R_j)}$, and let~$\mathcal{R}'$ be the subset of~${\mathcal{R} \setminus \{R_j\}}$ which belong to this end.
    
    Then, $\mathcal{R}'$ is a disjoint family of rays in the same end of~${\Gamma-V(R_j)}$ and so by Lemma~\ref{rem:tail-raygraph} the ray graph~${\RG_{{\Gamma-V(R_j)}}(\mathcal{R}')}$ is connected.
    However, it is apparent that~${\RG_{{\Gamma-V(R_j)}}(\mathcal{R}')}$ is a subgraph of~${\RG_{\Gamma}(\mathcal{R})}$,
    and so~$i_1$ and~$i_2$ belong to the same component of~${\RG(\mathcal{R})-j}$.
    
    Conversely, suppose~$i_1$ and $i_2$ belong to the same component of~${\RG(\mathcal{R})-j}$. Then, it is clear that for any two adjacent vertices~$k$ and~$\ell$ in~${\RG(\mathcal{R})-j}$ the rays~$R_k$ and~$R_\ell$ are equivalent in $\Gamma - R_j$, and hence~$R_{i_1}$ and~$R_{i_2}$ belong to a common end of $\Gamma - R_j$. It follows that~$R_j$ does not separate~$R_{i_1}$ from~$R_{i_2}$.
\end{proof}

\begin{lemma}
    \label{l:ray_tripple_seps}
    If $R,S,T$ are $\epsilon$-rays and~$S$ separates~$R$ from~$T$, then~$T$ does not separate~$R$ from~$S$ and~$R$ does not separate~$S$ from~$T$.
\end{lemma}

\begin{proof}
    As~$R$ and~$T$ both belong to~$\epsilon$, there are infinitely many disjoint paths between them. 
    As~$S$ separates~$R$ from~$T$, we know that~$S$ must meet infinitely many of these paths. 
    Hence, there are infinitely many disjoint paths from~$S$ to~$R$, all disjoint from~$T$. 
    Similarly, there are infinitely many disjoint paths from~$S$ to~$T$, all disjoint from~$R$. 
    Hence~$T$ does not separate~$R$ from~$S$ and~$R$ does not separate~$S$ from~$T$. 
\end{proof}

\begin{lemma}
    \label{l:centraltwoend}
   Let~$R$ be a core ray of~$\epsilon$. 
   Then in~${\Gamma - V(R)}$ the end~$\epsilon$ splits into precisely two different ends. 
   (That is, there are two ends~$\epsilon'$ and~$\epsilon''$ of~${\Gamma - V(R)}$ such that every $\epsilon$-ray in~$\Gamma$ which is disjoint from~$R$ is in~$\epsilon'$ or~$\epsilon''$ in~${\Gamma - V(R)}$.)
\end{lemma}

\begin{proof}
    Let ${\mathcal{R} = ( R_i \colon i \in I)}$ be a finite family of disjoint $\epsilon$-rays witnessing that~${R = R_c}$ for some~${c \in I}$ is a core ray. 
    Then there are precisely two ends $\epsilon'$ and $\epsilon''$ in~${\Gamma-V(R)}$ that contain rays in~$\mathcal{R}$, since connected components of ${\RG(\mathcal{R})-c}$ are equivalent sets of rays in~${\Gamma-V(R)}$ and moreover, the two connected components do not contain rays belonging to the same end of~${\Gamma-V(R)}$ by Lemma~\ref{l:raygraph_separates}. 
    
    Suppose there is a third end in~${\Gamma-V(R)}$ that contains an $\epsilon$-ray~$S$. 
    We first claim that there is a tail of~$S$ which is disjoint from~${\bigcup \mathcal{R}}$. 
    Indeed, clearly~$S$ is disjoint from~$R$, and if~$S$ meets~${\bigcup \mathcal{R}}$ infinitely often then it would meet some~${R_i \in \mathcal{R}}$ infinitely often, and hence lie in the same end of~${\Gamma-V(R)}$ as~$R_i$. 
    So let~$S'$ be a tail of~$S$ which is disjoint from~${\bigcup \mathcal{R}}$.
    
    Let us consider the family~${\mathcal{R}' := \mathcal{R} \cup \{S'\}}$, 
    where the ray~$S'$ is indexed by some additional index~$s$. 
    Since~$S'$ is an $\epsilon$-ray, the ray graph~$\RG(\mathcal{R'})$ is connected. Furthermore, since the identity on~$I$ is clearly a transition function from~$\mathcal{R}$ to~$\mathcal{R}'$, by Lemma~\ref{l:halfgridstructure}, $c$ is an inner vertex of the central path of~$\RG(\mathcal{R}')$, and hence has degree two.
    
    We claim that~$s$ is adjacent to some~${i \neq c}$ in~$\RG(\mathcal{R}')$. 
    Indeed, if not, then~$s$ must be a leaf of~$\RG(\mathcal{R}')$ adjacent to~$c$. In which case, there must be some neighbour~$i$ of~$c$ in~$\RG(\mathcal{R})$ which is not adjacent to~$c$ in~$\RG(\mathcal{R}')$. However, then~$s$ must be adjacent to~$i$ in~$\RG(\mathcal{R}')$. 
    
    However, then clearly~$s$ lies in the same end of~${\Gamma-V(R)}$ as~$R_i$, and hence in either~$\epsilon'$ or~$\epsilon''$. 
\end{proof}

Hence, every core ray~$R$ splits~$\epsilon$ into two ends. 
We would like to use this partition to define our partial order on core rays; the core rays in one end will be less than~$R$ and the core rays in the other end will be greater than~$R$. 
However, if we want this partial order to agree with the correct orientation of the central path for any disjoint family of rays in~$\epsilon$, then every family of rays~${(R_i \colon i \in I)}$ in whose ray graph~${R=R_c}$ is a vertex of the central path will choose which end of~${\Gamma-V(R)}$ is less than~$R$ and which is greater than~$R$, and we must make sure that this choice is consistent. 

So, given a finite family of disjoint $\epsilon$-rays~${\Rcal=(R_i \colon i \in I)}$ in whose ray graph~${R=R_c}$ is a vertex of the central path, we denote by~${\top(R,\Rcal)}$ the end of~${\Gamma-V(R)}$ containing rays~$R_i$ satisfying~${i<c}$, where~$<$ refers to the correct orientation of the vertices of the central path, 
and with~${\bot(R,\Rcal)}$ the end containing rays~$R_i$ satisfying~${i>c}$. 
We will show that the labelling~$\top$ and~$\bot$ is in fact independent of the choice of family~${\Rcal}$. 

\begin{definition}
    Given two (possibly infinite) vertex sets~$X$ and~$Y$ in~$\Gamma$, 
    we say that an end~$\epsilon$ of~${\Gamma-X}$ is a \emph{sub-end} of an end~$\epsilon'$ of~${\Gamma-Y}$ if every ray in~$\epsilon$ has a tail in~$\epsilon'$.
\end{definition}

\begin{lemma}
    \label{l:end_distribution}
    Let~$R$ and~$S$ be disjoint core rays of~$\epsilon$. 
    Let us suppose that~$\epsilon$ splits in~${\Gamma-V(S)}$ into~$\epsilon'_S$ and~$\epsilon''_S$ and in~${\Gamma-V(R)}$ into~$\epsilon'_R$ and~$\epsilon''_R$. 
    If~$R$ belongs to~$\epsilon'_S$ and~$S$ belongs to~$\epsilon'_R$, then~$\epsilon''_S$ is a sub-end of~$\epsilon'_R$ and~$\epsilon''_R$ is a sub-end of~$\epsilon'_S$.
\end{lemma}

\begin{proof}
    Let~$T$ be a ray in~$\epsilon''_S$. 
    As~$R$ belongs to a different end of~${\Gamma-V(S)}$ than~$T$, there is a tail~$T'$ of~$T$ which is disjoint from~$R$. 
    As~$S$ separates~$R$ from~$T'$, we know, by Lemma~\ref{l:ray_tripple_seps}, that~$R$ does not separate~$S$ from~$T'$, hence~$T'$ belongs to~$\epsilon'_R$. 
    Hence,~$\epsilon_S''$ is a sub-end of~$\epsilon_R'$. 
    Proving that~$\epsilon''_R$ is a sub-end of~$\epsilon'_S$ works analogously. 
\end{proof}

\begin{lemmadef}
    \label{l:unique_above}
    Let ${\mathcal{R}_1 = ( R_i \colon i \in I_1 )}$ and ${\mathcal{R}_2 = ( R_i \colon i \in I_2 )}$ be two finite families of disjoint $\epsilon$-rays, such for some~${c \in I_1 \cap I_2}$ the ray~$R_c$ lies on the central path of both $\RG(\mathcal{R}_1)$ and $\RG(\mathcal{R}_2)$. 
    Then~${\top(R_c,\Rcal_1) = \top(R_c,\Rcal_2)}$ and~${\bot(R_c,\Rcal_1) = \bot(R_c,\Rcal_2)}$. 
    
    We therefore write~${\top(\epsilon,R_c)}$ for the end~${\top(R_c,\Rcal_1)}$ 
    and~$\bot(\epsilon,R_c)$ accordingly, i.e.\ ${\top(\epsilon, R_c)}$ is the end of~${\Gamma - V(R_c)}$ containing rays that appear on the central path of some ray graph before~$R_c$ according to the correct orientation and~${\bot(\epsilon, R_c)}$ is the end of~${\Gamma - V(R_c)}$ containing rays that appear on the central path of some ray graph after~$R_c$ according to the correct orientation. 
    Note that~${\top(\epsilon,R_c) \cap \bot(\epsilon,R_c) = \emptyset}$. 
\end{lemmadef}

\begin{proof}
    Let~$\epsilon'$ and~$\epsilon''$ be the two ends of~${\Gamma - V(R_c)}$ and let~$\Rcal_1'$ and~$\Rcal_2'$ be the set of rays in~$\Rcal_1$ and~$\Rcal_2$ respectively that belong to~$\epsilon'$, and similarly~$\Rcal_1''$ and~$\Rcal_2''$ be the set of rays in~$\Rcal_1$ and~$\Rcal_2$ respectively that belong to $\epsilon''$. 
    Let~$\Scal'$ be the larger of~$\Rcal_1'$ and~$\Rcal_2'$, and similarly~$\Scal''$ the larger of~$\Rcal_1''$ and~$\Rcal_2''$.
    
    Let us consider the family of rays~${\Scal := \Scal' \cup \{R_c\} \cup \Scal''}$. 
    Since the rays in~$\Scal'$ and~$\Scal''$ belong to different ends of~${\Gamma - V(R_c)}$, we may, after replacing some of the rays with tails, assume that~$\Scal$ is a family of disjoint rays. 
    We claim that there is a transition function~$\sigma_1$ from~$\Rcal_1$ to~$\Scal$ which maps~$R_c$ to itself, $\Rcal_1'$ to~$\Scal'$, and~$\Rcal_1''$ to~$\Scal''$. 
    
    Indeed, let us take a finite separator~$X$ which separates~$\epsilon'$ and~$\epsilon''$ in~${\Gamma - V(R_c)}$. 
    By Lemma~\ref{l:trans-link}, there is a finite set~$Y$ such that any linkage after~$Y$ from~$\Rcal_1$ to~$\Scal$ is transitional. 
    Then, since the rays in~$\Rcal_1'$ and~$\Scal'$ belong to the same end of~${\Gamma - V(R_c)}$ and~${|\Rcal_1'| \leq |\Scal'|}$, there is a linkage after~${X \cup Y}$ from~$\Rcal_1'$ to~$\Scal'$ in~${\Gamma - V(R_c)}$, and similarly there is a linkage after~${X \cup Y}$ from~$\Rcal_1''$ to~$\Scal''$ in~${\Gamma - V(R_c)}$. 
    If we combine these two linkages with a trivial linkage from~$R_c$ to itself after~${X \cup Y}$, we obtain a transitional linkage which induces an appropriate transition function.
    
    The same argument shows that there is a transition function~$\sigma_2 $ from~$\Rcal_2$ to~$\Scal$ which maps~$R_c$ to itself, $\Rcal_2'$ to~$\Scal'$, and~$\Rcal_2''$ to~$\Scal''$. 
    By Lemma~\ref{l:halfgridstructure}, both transition functions map vertices of the central path to vertices of the central path and preserve the correct orientation. 
    In particular, $R_c$ lies on the central path of $\RG(\mathcal{S})$.
    
    Moreover, both~$\sigma_1$ and~$\sigma_2$ map $\epsilon'$-rays to $\epsilon'$-rays and $\epsilon''$-rays to $\epsilon''$-rays. 
    Therefore, if ${\epsilon' = \top(R_c,\Scal)}$, then~$\sigma_1$ shows that ${\epsilon' = \top(R_c,\Rcal_1)}$ and~$\sigma_2$ shows that ${\epsilon' = \top(R_c,\Rcal_2)}$, and similarly if~${\epsilon' = \bot(R_c,\Scal)}$.
\end{proof}

\begin{lemmadef}
    \label{def:core-order}
    Let~${\core(\epsilon)}$ denote the set of core rays in~$\epsilon$. 
    We define a partial order~$\leq_\epsilon$ on~${\core(\epsilon)}$ by
    \begin{align*}
        R \leq_\epsilon S 
            &\text{ if and only if either } R = S, \\
            &\text{ or } R \text{ and } S \text{ have disjoint tails $xR$ and $yS$ and } xR \in \top(\epsilon, yS)
    \end{align*}
    for~${R, S \in \core(\epsilon)}$.
\end{lemmadef}

\begin{proof}
    We must show that~$\leq_\epsilon$ is indeed a partial order. 
    For the anti-symmetry, let us suppose that~$R$ and~$S$ are disjoint rays in~${\core(\epsilon)}$ such that~${R \leq_\epsilon S}$ and~${S \leq_\epsilon R}$, so that~${R \in \top(\epsilon, S)}$ and~${S \in \top(\epsilon, R)}$. 
    Let~$\mathcal{R}_S$ be a family of rays witnessing that~$S$ is a core ray and~$\mathcal{R}_R$ a family witnessing that~$R$ is a core ray. 
    By Lemma~\ref{l:end_distribution}, ${\bot(\epsilon,S)}$ is a sub-end of~${\top(\epsilon,R)}$ and~${\bot(\epsilon,R)}$ is a sub-end of~${\top(\epsilon,S)}$. 
    Let~$\mathcal{R}_{\bot(S)}$ be the subset of~$\mathcal{R}_S$ of rays which belong to~${\bot(\epsilon, S)}$. 
    Let~$\mathcal{R}_{\bot(R)}$ be defined accordingly.
    By replacing rays with tails, we may assume that all rays in~${\mathcal{R} := \mathcal{R}_{\bot(S)} \cup \mathcal{R}_{\bot(R)} \cup \{R\} \cup\{S\}}$ are pairwise disjoint. Note that, by the comment after Definition~\ref{d:coreray}, both $R$ and~$S$ are inner vertices of the central path of~${\RG(\mathcal{R})}$. 
    Thus, either~${S \in \bot(\epsilon,R)}$ or~${R \in \bot(\epsilon,S)}$, contradicting Lemma~\ref{l:unique_above}.
    
    For the transitivity, let us suppose that $R,S,T$ are rays in~${\core(\epsilon)}$, such that~${R \leq_\epsilon S}$ and~${S \leq_\epsilon T}$. 
    We may assume that~$R$ and~$S$, and~$S$ and~$T$ are disjoint. 
    As~$\leq_\epsilon$ is anti-symmetric, we have~${T \not \leq_\epsilon S}$, hence~${T \in \bot(\epsilon,S)}$. 
    Thus,~$R$ and~$T$ belong to different ends of~${\Gamma-V(S)}$, and we may assume that they are also disjoint. 
    As~$S$ therefore separates~$R$ from~$T$, by Lemma~\ref{l:ray_tripple_seps}, we know that~$T$ does not separate~$S$ from~$R$. 
    Thus, $R$ and~$S$ belong to the same end of~${\Gamma-V(T)}$. 
    Hence~${R \in \top(\epsilon,T)}$. 
\end{proof}

\begin{remark}
    \label{rem:core-remarks}
    Let~${R, S \in \core(\epsilon)}$ and let~$\mathcal{R}$ be a finite family of disjoint $\epsilon$-rays.
    \begin{enumerate}
        \item \label{rem:core-tail} Any ray which shares a tail with $R$ is also a core ray of~$\epsilon$.
        \item \label{rem:core-comparable} If~$R$ and~$S$ are disjoint, then~$R$ and~$S$ are comparable under~$\leq_\epsilon$.
        \item \label{rem:core-correct} If~$R$ and~$S$ are on the central path of~${\RG(\Rcal)}$, then~${R \leq_\epsilon S}$ if and only if~$R$ appears before~$S$ in the correct orientation of the central path of~$\RG(\Rcal)$.
        \item \label{rem:non-core-bounded} The maximum number of disjoint rays in ${\epsilon \setminus \core(\epsilon)}$ is bounded by~${2N+2}$.
    \end{enumerate}
\end{remark}

\begin{lemma}
    \label{lem:core-exchange}
    Let~${R, S \in \core(\epsilon)}$ and let~${Z \subseteq V(\Gamma)}$ be a finite set such that~${\top(\epsilon, S)}$ and~${\bot(\epsilon, S)}$ are separated by~$Z$ in~${\Gamma - V(S)}$.
    Let~${H \subseteq \Gamma - Z}$ be a connected subgraph which is disjoint to~$S$ and contains~$R$
    and let~${T \subseteq H}$ be some core $\epsilon$-ray.
    Then~$S$ is in the same relative $\leq_\epsilon$-order to~$T$ as to~$R$.
\end{lemma}

\begin{proof}
    Assume~${S \leq_\epsilon R}$ and hence~${R \in \top(\epsilon, S)}$. 
    Since~$H$ is connected, we obtain that ${T \in \top(\epsilon, S)}$ as well and hence~${S \leq_\epsilon T}$. 
    The other case is analogous.
\end{proof}

Since, by \ref{rem:core-remarks}~\ref{rem:core-correct}, the order $\leq_\epsilon$ will agree with correct order on the central path, which is preserved by transition functions by Lemma~\ref{l:halfgridstructure}, the order $\leq_\epsilon$ will also be preserved by transition functions, as long as they map core rays to core rays. In order to guarantee that this holds, before linking a family of core rays $\mathcal{R}$ we will first enlarge it slightly by adding some `buffer' rays.

\begin{lemmadef}
    \label{def:central-core}
    Let~$\mathcal{R} = (R_i \colon i \in I)$ be a finite family of disjoint core $\epsilon$-rays. 
    Then there exists a finite family~$\overline{\mathcal{R}} \supset \mathcal{R}$ of disjoint $\epsilon$-rays such that
    \begin{itemize}
    \item For each $i\in I$, the graph $\RG(\overline{\mathcal{R}})-i$ has precisely two components, each of size at least $N+1$;
    \item Each $i\in I$ is an inner vertex of the central path of~$\RG(\overline{\mathcal{R}})$;
    \item $|\overline{\mathcal{R}}| = |\mathcal{R}|+2N+2$.
    \end{itemize}
    Even though such a family is not unique, we denote by~${\overline{\mathcal{R}}}$ an arbitrary such family.
\end{lemmadef}

\begin{proof}
    By Remark~{\ref{rem:core-remarks}\ref{rem:core-comparable}}, the rays in~$\mathcal{R}$ are linearly ordered by~$\leq_\epsilon$.
    Let~$R$ denote the $\leq_\epsilon$-greatest and~$S$ denote the $\leq_\epsilon$-smallest element of~$\mathcal{R}$.

    As in the proof of Lemma~\ref{def:core-order}, let~$\Scal_R$ and~$\Scal_S$ be families of disjoint rays witnessing that~$R$ and~$S$ are core rays, and let~$\Scal_{\bot(R)}$ be the subset of rays of~$\Scal_R$ belonging to~${\bot(\epsilon,R)}$ and~$\Scal_{\top(S)}$ be the subset of rays of~$\Scal_S$ belonging to~${\top(\epsilon,S)}$. Note that, by definition both $\Scal_{\bot(R)}$ and $\Scal_{\top(S)}$ contain at least $N+1$ rays, and we may in fact assume without loss of generality that they both contain exactly $N+1$ rays.
    
    Now~${\mathcal{S}_{\bot(R)} \subseteq \bot(\epsilon, R)}$ and~${R' \in \top(\epsilon, R)}$ for every~${R' \in \mathcal{R} \setminus \{R\}}$, and each ray in~$\mathcal{S}_{\bot(R)}$ has a tail disjoint to~${\bigcup \mathcal{R}}$. 
    Analogously, ${\mathcal{S}_{\top(S)} \subseteq \top(\epsilon, S)}$ and~${R' \in \bot(\epsilon, S)}$ for every~${R' \in \mathcal{R} \setminus \{S\}}$ and each ray in~$\mathcal{S}_{\top(S)}$ has a tail disjoint to~${\bigcup \mathcal{R}}$. 
    Now, ${\mathcal{S}_{\top(S)} \subseteq \top(\epsilon,R)}$ and ${\mathcal{S}_{\bot(R)} \subseteq \bot(\epsilon,S)}$ by Lemma~\ref{l:end_distribution}, 
    yielding that tails of rays in~$\mathcal{S}_{\top(S)}$ are necessarily disjoint from tails in~$\mathcal{S}_{\bot(R)}$.
    
    Let~$\overline{\mathcal{R}}$ be the union of $\mathcal{R}$ with appropriate tails of each ray in $\mathcal{S}_{\bot(R)} \cup \mathcal{S}_{\top(S)}$. Note that $|\overline{\mathcal{R}}| = |\mathcal{R}|+2N+2$. For any ray $R_i \in \mathcal{R}$, we first note that that $S \leq_\epsilon R_i$ and so $S \in \top(\epsilon,R_i)$ and $R_i \in \bot(\epsilon, S)$. Then, since $\mathcal{S}_{\top(S)} \subseteq \top(\epsilon, S)$ it follows from Lemma~\ref{l:end_distribution} that $\mathcal{S}_{\top(S)} \subseteq \top(\epsilon, R_i)$, and hence one of the components of~${\RG(\overline{\mathcal{R}}) - i}$ has size at least~${N+1}$. 
    A similar argument shows that a second component has size at least~${N+1}$, and finally, since~$R_i$ is a core ray, by Lemma~\ref{l:centraltwoend}, there are no other components of~${\RG(\overline{\mathcal{R}}) - i}$. Finally, by the comment after Definition~\ref{d:coreray}, it follows that $R_i$ is an inner vertex of the central path of this ray graph.
\end{proof}

\begin{lemma}[{\cite[Lemma 7.10]{BEEGHPTII}}]
    \label{l:transitionseparate}
    Let~$\mathcal{R}$ and~$\mathcal{T}$ be families of disjoint rays, each of size at least~${N+3}$, and let~$\sigma$ be a transition function from~$\Rcal$ to~$\Tcal$. 
    Let~${x \in \RG(\mathcal{R})}$ be an inner vertex of the central path. 
    If~${v_1,v_2\in \RG(\mathcal{R})}$ lie in different components of~${\RG(\mathcal{R})-x}$, then~$\sigma(v_1)$ and~$\sigma(v_2)$ lie in different components of~${\RG(\Tcal)-\sigma(x)}$. 
    Moreover, $\sigma(x)$ is an inner vertex of the central path of~$\RG(\Tcal)$.
\end{lemma}

\begin{definition}
    Let~$\mathcal{R}$, $\mathcal{S}$ be finite families of disjoint $\epsilon$-rays and let~$\mathcal{R}'$ be a subfamily of~$\mathcal{R}$ consisting of core rays.
    A linkage~$\mathcal{P}$ between~$\mathcal{R}$ and~$\mathcal{S}$ is \emph{preserving on~$\mathcal{R}'$} if~$\mathcal{P}$ links~$\mathcal{R}'$ to core rays and preserves the order~$\leq_\epsilon$. 
\end{definition}

\begin{lemma}\label{l:core-preserving}
    Let~${\mathcal{R} = (R_i \colon i \in I)}$ be a finite family of disjoint core $\epsilon$-rays and let ${\mathcal{S} = (S_j \colon j \in J)}$ be a finite family of disjoint $\epsilon$-rays. 
    Let ${\overline{\mathcal{R}}=(R_i \colon i \in \overline{I})}$ be as in Lemma~\ref{def:central-core} and let~$\mathcal{P}$ be a linkage from~$\overline{\mathcal{R}}$ to~$\mathcal{S}$. If~$\mathcal{P}$ is transitional, then it is preserving on~$\mathcal{R}$.
\end{lemma}

\begin{proof}
    We first note that, by Lemma~\ref{l:halfgridstructure}, if~$\mathcal{P}$ links the rays in~$\mathcal{R}$ to core rays, then it will be preserving.
    
    So, let~${\sigma: \overline{I} \rightarrow J}$ be the transition function induced by~$\mathcal{P}$. 
    For each~${i \in I}$, since~$i$ is an inner vertex of the central path of~${\RG(\overline{\mathcal{R}})}$, by Lemma~\ref{l:transitionseparate}, $\sigma(i)$ is an inner vertex of the central path of~${\RG(\mathcal{S})}$. 
    Since the central path is a bare path, it follows that~${\RG(\mathcal{S}) - \sigma(i)}$ has precisely two components.
    
    Furthermore, by Lemma~\ref{def:central-core}, the graph~${\RG(\overline{\mathcal{R}}) - i}$ has precisely two components, each of size at least~${N+1}$, and so by Lemma~\ref{l:transitionseparate} the two components of~${\RG(\mathcal{S}) - \sigma(i)}$ each have size at least~${N+1}$. 
    Hence, the family~$\mathcal{S}$ witnesses that~$S_{\sigma(i)}$ is a core ray.
\end{proof}

\begin{definition}
    If~$\mathcal{P}$ is a linkage from~$\mathcal{R}$ to~$\mathcal{S}$, then a \emph{sub-linkage} of~$\Pcal$ is just a subset of~$\Pcal$, considered as a linkage from the corresponding subset of~$\mathcal{R}$ to~$\mathcal{S}$.
\end{definition}

\begin{remark}
    \label{rem:transitional-sub}
    A sub-linkage of a transitional linkage is transitional.
\end{remark}

The following remarks are a direct consequence of the definitions and Lemma~\ref{l:halfgridstructure}.

\begin{remark}
    \label{rem:core-preserving}
    Let $\mathcal{R}$ be a finite family of disjoint core $\epsilon$-rays and let $\mathcal{S}$ and $\mathcal{T}$ be finite families of disjoint $\epsilon$-rays. Let $\overline{\mathcal{R}}$ be as in Lemma~\ref{def:central-core} and let~$\mathcal{P}_1$ and~$\mathcal{P}_2$ be linkages from~$\overline{\mathcal{R}}$ to~$\mathcal{S}$ and from~${(\overline{\mathcal{R}} \circ_{\mathcal{P}_1} \mathcal{S})}$ to~$\mathcal{T}$ respectively.
    \begin{enumerate}
        \item \label{rem:core-preserving-sub} If~$\mathcal{P}_1$ is preserving on~$\mathcal{R}$, then any~${\mathcal{P}_1' \subseteq \mathcal{P}_1}$ as a linkage between the respective subfamilies is preserving on the respective subfamily of~$\mathcal{R}$.
        \item \label{rem:core-preserving-concat} If~$\mathcal{P}_1$ is preserving on~$\mathcal{R}$ and~$\mathcal{P}_2$ is preserving on~${\mathcal{R} \circ_{\mathcal{P}_1} \mathcal{S}}$, then the concatenation~${\mathcal{P}_1 + \mathcal{P}_2}$ is preserving on~$\mathcal{R}$.
    \end{enumerate}
\end{remark}

\begin{lemma}
   \label{lem:core-preserving2}
   Let~$\mathcal{R}$ and~$\mathcal{S}$ be finite families of disjoint core rays of~$\epsilon$ and let~${\mathcal{S}' \subseteq \mathcal{S}}$ be a subfamily of~$\mathcal{S}$ with~${|\mathcal{R}| = |\mathcal{S}'|}$.
   Then there is a transitional linkage from~$\overline{\mathcal{R}}$ to~$\overline{\mathcal{S}}$ which is preserving on~$\mathcal{R}$ and links the rays in~$\mathcal{R}$ to rays in~$\mathcal{S}'$. 
\end{lemma}

\begin{proof}
    Consider~${\mathcal{T} := (\overline{\mathcal{S}} \setminus \mathcal{S}) \cup \mathcal{S}' \subseteq \overline{\mathcal{S}}}$. 
    It is apparent that the family~$\mathcal{T}$ satisfies the conclusions of Lemma~\ref{def:central-core} for~$\mathcal{S}'$.
    
    Let~$\sigma$ be some transition function between~$\overline{\mathcal{R}}$ and~$\mathcal{T}$ and let~$\mathcal{P}$ be a linkage inducing this transition function. 
    By Lemma~\ref{l:core-preserving} this linkage is preserving on~$\mathcal{R}$. 
    Note that, since~$\sigma$ is a transition function from~$\overline{\mathcal{R}}$ to~$\mathcal{T}$, it is also a transition function from~$\overline{\mathcal{R}}$ to~$\overline{\mathcal{S}}$, and so~$\mathcal{P}$ is also a preserving, transitional linkage from~$\overline{\mathcal{R}}$ to~$\overline{\mathcal{S}}$. 
    We claim further that~$\mathcal{P}$ links the rays in~$\mathcal{R}$ to the rays in~$\mathcal{S}'$.
    
    Indeed, since $|\overline{\mathcal{R}}| = |\mathcal{T}| = |\mathcal{R}| + 2N+2$, we may assume for a contradiction that there is some $R_i \in \mathcal{R}$ such that $S_{\sigma(i)} \not\in \mathcal{S}'$. Note that, since $i$ is an inner vertex of the central path of $\RG(\overline{\mathcal{R}})$, by Lemma~\ref{l:transitionseparate} $\sigma(i)$ is an inner vertex of the central path of $\RG(\mathcal{T})$, and so in particular $\RG(\mathcal{T}) - \sigma(i)$ has precisely two components.

    Since for each~${S_j \in \mathcal{S'}}$, $j$ lies on the central path of~$\RG(\mathcal{T})$, if ${S_{\sigma(i)} \not\in \mathcal{S'}}$ then it is clear that~${\RG(\mathcal{T}) \setminus \sigma(i)}$ contains one component of size at least~${|\mathcal{S}'| + N+1 = |\mathcal{R}|+N+1}$. 
    However, since~$i$ is an inner vertex of the central path of~$\RG(\overline{\mathcal{R}})$, by Lemma~\ref{l:transitionseparate} and Lemma~\ref{l:core-preserving} there must be two components of~${\RG(\mathcal{T}) \setminus \sigma(i)}$ of size at least~${N+1}$, a contradiction.
\end{proof}

\section{%
\texorpdfstring{$G$-tribes and concentration of $G$-tribes towards an end}%
{G-tribes and concentration of G-tribes towards an end}%
}
\label{s:tribes}

To show that a given graph~$G$ is $\preceq$-ubiquitous, we shall assume that~${n G \preceq \Gamma}$ for every~${n \in \Nbb}$ and need to show that this implies~${\aleph_0 G \preceq \Gamma}$. 
To this end we use the following notation for such collections of~$nG$ in~$\Gamma$, which we introduced in~\cite{BEEGHPTI} and~\cite{BEEGHPTII}.

\begin{definition}[\Gtribe s]
    Let~$G$ and~$\Gamma$ be graphs.
    \begin{itemize}
        \item A \emph{\Gtribe\ in~$\Gamma$ (with respect to the minor relation)} is a family~$\mathcal{F}$ of finite collections~$F$ of disjoint subgraphs~$H$ of~$\Gamma$, such that each \emph{member}~$H$ of~$\Fcal$ is an~$IG$.
        \item A \Gtribe\ $\mathcal{F}$ in~$\Gamma$ is called \emph{thick} if for each~${n \in \Nbb}$, there is a \emph{layer}~${F \in \mathcal{F}}$ with~${|F| \geq n}$; otherwise, it is called \emph{thin}.
        \item A $G$-tribe~$\Fcal$ is \emph{connected} if every member~$H$ of~$\Fcal$ is connected. 
            Note that, this is the case precisely if~$G$ is connected.
        \item A \Gtribe\ $\mathcal{F}'$ in~$\Gamma$ is a \emph{\Gsubtribe}\footnote{When~$G$ is clear from the context we will often refer to a $G$-subtribe as simply a subtribe.}
            of a \Gtribe\ $\mathcal{F}$ in~$\Gamma$, denoted by~${\Fcal' \preceq \Fcal}$, if there is an injection~${\Psi \colon \Fcal' \to \Fcal}$ such that for each~${F' \in \mathcal{F}'}$, there is an injection ${\varphi_{F'} \colon F' \to \Psi(F')}$ with~${V(H') \subseteq V(\varphi_{F'}(H'))}$ for every~${H' \in F'}$. 
            The \Gsubtribe~$\Fcal'$ is called \emph{flat}, denoted by~${\Fcal' \subseteq \Fcal}$, if there is such an injection~$\Psi$ satisfying~${F' \subseteq \Psi(F')}$.
        \item A thick \Gtribe\ $\mathcal{F}$ in~$\Gamma$ is \emph{concentrated at an end~$\epsilon$} of~$\Gamma$ 
            if for every finite vertex set~$X$ of~$\Gamma$, the \Gtribe\ ${\Fcal_X = \{F_X \colon F \in \Fcal \}}$ consisting of the layers
            \[
                {F_X = \{H \in F \colon H \not\subseteq C(X,\epsilon)\} \subseteq F}
            \] 
            is a thin subtribe of~$\Fcal$.
    \end{itemize}
\end{definition}

We note that, if~$G$ is connected, every thick $G$-tribe~$\Fcal$ contains a thick subtribe~$\Fcal'$ such that every~${H \in \bigcup\Fcal'}$ is a tidy~$IG$. 
We will use the following lemmas from~\cite{BEEGHPTI}.

\begin{lemma}[Removing a thin subtribe, {\cite[Lemma 5.2]{BEEGHPTI}}]
    \label{l:removethin}
    Let~$\Fcal$ be a thick $G$-tribe in~$\Gamma$ and let~$\Fcal'$ be a thin subtribe of~$\Fcal$, witnessed by~${\Psi \colon \Fcal'\to \Fcal}$ and~${(\varphi_{F'} \colon F' \in \mathcal{F}')}$. 
    For~${F \in \Fcal}$, if~${F \in \Psi(\Fcal')}$, let~${\Psi^{-1}(F) = \{F'_F\}}$ and set~${\hat{F} = \varphi_{F'_F}(F'_F)}$. 
    If~${F \notin \Psi(\Fcal')}$, set~${\hat{F} = \emptyset}$. 
    Then 
    \[
        \Fcal'' := \{F\setminus \hat{F}\colon F\in \Fcal\}
    \] 
    is a thick flat $G$-subtribe of~$\Fcal$.
\end{lemma}

\begin{lemma}[Pigeon hole principle for thick $G$-tribes, {\cite[Lemma 5.3]{BEEGHPTI}}]
    \label{Lem_finitechoice}
    Let~${k \in \Nbb}$ and let~${c \colon \bigcup \mathcal{F} \to [k]}$ be a $k$-colouring of the members of some thick $G$-tribe~$\mathcal{F}$ in~$\Gamma$. 
    Then there is a monochromatic, thick, flat $G$-subtribe~$\mathcal{F}'$ of~$\mathcal{F}$.
\end{lemma}

\begin{lemma}[{\cite[Lemma 5.4]{BEEGHPTI}}]
    \label{l:concentrated}
    Let~$G$ be a connected graph and~$\Gamma$ a graph containing a thick connected $G$-tribe~$\mathcal{F}$. 
    Then either ${\aleph_0 G \preceq \Gamma}$, or there is a thick flat subtribe~$\Fcal'$ of~$\Fcal$ and an end~$\epsilon$ of~$\Gamma$ such that~$\Fcal'$ is concentrated at~$\epsilon$.
\end{lemma}

\begin{lemma}[{\cite[Lemma 5.5]{BEEGHPTI}}]
    \label{lem_subtribesinheritconcentration}
    Let~$G$ be a connected graph and~$\Gamma$ a graph containing a thick connected $G$-tribe~$\mathcal{F}$ concentrated at an end~$\epsilon$ of~$\Gamma$. 
    Then the following assertions hold: 
    \begin{enumerate}
        \item For every finite set~$X$, the component~${C(X,\epsilon)}$ contains a thick flat $G$-subtribe of~$\Fcal$.
        \item Every thick subtribe~$\Fcal'$ of~$\Fcal$ is concentrated at~$\epsilon$.
    \end{enumerate}
\end{lemma} 

The following lemma from~\cite{BEEGHPTII} shows that we can restrict ourself to thick $G$-tribes which are concentrated at thick ends.

\begin{lemma}[{\cite[Lemma 8.7]{BEEGHPTII}}]
    \label{l:concentratedatthin}
    Let~$G$ be a connected graph and~$\Gamma$ a graph containing a thick $G$-tribe $\Fcal$ concentrated at an end~${\epsilon \in \Omega(\Gamma)}$ which is thin.
    Then~${\aleph_0 G \preceq \Gamma}$.
\end{lemma}

Given an extensive tree-decomposition ${(T,\mathcal{V})}$ of~$G$, broadly our strategy will be to obtain a family of disjoint~$IG$'s by choosing a sequence of initial subtrees ${T_0 \subseteq T_1 \subseteq \ldots}$ such that~${\bigcup T_i = T}$ and constructing inductively a family of finitely many~${IG(T_{k+1})}$'s which extend the~${IG(T_{k})}$'s built previously (cf.~Definition~\ref{defn_treestuff}). 
The extensiveness of the tree-decomposition will ensure that, at each stage, there will be some edges in~${\partial(T_i) = E(T_i, T \setminus T_i)}$, each of which has in~$G$ a family of rays~$\mathcal{R}_e$ along which the graph $G$ displays self-similarity.

In order to extend our~$IG(T_k)$ at each step, we will want to assume that the~$IG$s in our thick $G$-tribe~$\mathcal{F}$ lie in a `uniform' manner in the graph~$\Gamma$ in terms of these rays~$\mathcal{R}_e$.

More specifically, for each edge~${e \in \partial(T_i)}$, the rays~$\mathcal{R}_e$ provided by the extensive tree-decomposition in Definition~\ref{d:extensive} tend to a common end~$\omega_e$ in~$G$, and for each~${H \in \bigcup \mathcal{F}}$, the corresponding rays in~$H$ converge to an end~${H(\omega_e) \in \Omega(\Gamma)}$ (cf.~Definition~\ref{def:H(omega)}), which might either be the end~$\epsilon$ of~$\Gamma$ at which~$\mathcal{F}$ is concentrated, or another end of~$\Gamma$. 
We would like that our \Gtribe~$\mathcal{F}$ makes a consistent choice across all members $H$ of $\mathcal{F}$   of whether~$H(\omega_e)$ is~$\epsilon$, for each~${e \in \partial(T_i)}$. 

Furthermore, if~${H(\omega_e) = \epsilon}$ for every~${H \in \bigcup \mathcal{F}}$, then this imposes some structure on the end~$\omega_e$ of~$G$.
More precisely, by~\cite[Lemma~10.1]{BEEGHPTII}, we may assume that~${\RG_H(H^{\downarrow}(\mathcal{R}_e))}$ is a path for each member~$H$ of the $G$-tribe~$\mathcal{F}$, or else we immediately find that $\aleph_0 G \preccurlyeq \Gamma$ and are done.

By moving to a thick subtribe, we may assume that every $\epsilon$-ray in $H$ is a core ray for every~${H \in \bigcup \mathcal{F}}$, in which case~$\leq_\epsilon$ imposes a linear order on every family of rays~$H^{\downarrow}(\mathcal{R}_e)$, which induces one of the two distinct orientations of the path~${\RG_H(H^{\downarrow}(\mathcal{R}_e))}$. 
We will also want that our tribe~$\mathcal{F}$ induces this orientation in a consistent manner. 

Let us make the preceding discussion precise with the following definitions:

\begin{definition}
    \label{d:tribes}
    Let~$G$ be a connected locally finite graph with an extensive tree-\linebreak decomposition~${(T,\mathcal {V})}$ and~$S$ be an initial subtree of~$T$.
    Let~${H \subseteq \Gamma}$ be a tidy~$IG$, $\mathcal{H}$ be a set of tidy~$IG$s in~$\Gamma$, and~$\epsilon$ an end of~$\Gamma$.
    \begin{itemize}
        \item 
            Given an end~$\omega$ of~$G$, we say that~$\omega$ \emph{converges to~$\epsilon$ according to~$H$} if $H(\omega) = \epsilon$ (cf.~Definition~\ref{def:H(omega)}). 
            The end~$\omega$ \emph{converges to~$\epsilon$ according to~$\Hcal$} if it converges to~$\epsilon$ according to every element of~$\Hcal$. 
            
            We say that~$\omega$ is \emph{cut from~$\epsilon$ according to~$H$} if~$H(\omega) \neq \epsilon$. 
            The end~$\omega$ is \emph{cut from~$\epsilon$ according to~$\Hcal$} if it is cut from~$\epsilon$ according to every element of~$\Hcal$. 
            
            Finally, we say that~$\Hcal$ \emph{determines whether~$\omega$ converges to~$\epsilon$} if either~$\omega$ converges to~$\epsilon$ according to~$\Hcal$ or~$\omega$ is cut from~$\epsilon$ according to~$\Hcal$. 
        
        \item 
            Given~${E \subseteq E(T)}$, we say~$\Hcal$ \emph{weakly agrees about~$E$} if for each~${e \in E}$, the set~$\Hcal$ determines whether~$\omega_e$ (cf.~Definition~\ref{d:extensive}) converges to~$\epsilon$.
            If~$\Hcal$ weakly agrees about~${\partial(S)}$ we let
            \begin{align*}
                \partial_\epsilon (S) &:= \{e \in \partial (S) \colon \omega_e \text{ converges to~$\epsilon$ according to } \Hcal \}\,, \\
                \partial_{\neg \epsilon} (S)  &:=  \{e \in \partial (S) \colon \omega_e \text{ is cut from~$\epsilon$ according to } \Hcal \} \, ,
            \end{align*}
            and write 
            \begin{align*}
                S^{\neg \epsilon} &\text{ for the component of the forest } T - \partial_\epsilon (S) \text{ containing the root of } T\,, \\
                S^{\epsilon} &\text{ for the component of the forest } T - \partial_{\neg \epsilon} (S) \text{ containing the root of } T\,.
            \end{align*}
            Note that~${S = S^{\neg \epsilon} \cap S^{\epsilon}}$. 
            
        \item 
            We say that~$\Hcal$ is \emph{well-separated from~$\epsilon$ at~$S$} if~$\Hcal$ weakly agrees about~${\partial(S)}$ and~${H(S^{\neg \epsilon})}$ can be separated from~$\epsilon$ in~$\Gamma$ for all elements~${H \in \Hcal}$, i.e.~for every~$H$ there is a finite~${X \subseteq V(\Gamma)}$ such that~${H(S^{\neg \epsilon}) \cap C_\Gamma(X,\epsilon) = \emptyset}$.
    \end{itemize}
    
    In the case that~$\epsilon$ is half-grid-like, we say that~$\Hcal$ \emph{strongly agrees} about~$\partial(S)$ if 
    \begin{itemize}
        \item it weakly agrees about~${\partial(S)}$; 
        \item for each~${H \in \Hcal}$, every $\epsilon$-ray~${R \subseteq H}$ is in~${\core(\epsilon)}$; 
        \item for every~${e \in \partial_\epsilon(S)}$, there is a linear order~$\leq_{\mathcal{H},e}$ on~${S(e)}$ (cf.~Definition~\ref{d:extensive}), such that the order induced on~${H^{\downarrow}(\mathcal{R}_e)}$ by~$\leq_{\mathcal{H},e}$ agrees with~$\leq_\epsilon$ on~${H^{\downarrow}(\mathcal{R}_e)}$ for all~${H \in \Hcal}$. 
    \end{itemize}
    
    If~$\Fcal$ is a thick $G$-tribe concentrated at an end~$\epsilon$, we use these terms in the following way:
    \begin{itemize}
        \item Given~${E \subseteq E(T)}$, we say that~$\Fcal$ \emph{weakly agrees about~$E$} if~${\bigcup \Fcal}$ weakly agrees about~$E$ w.r.t.~$\epsilon$.
        \item We say that~$\Fcal$ is \emph{well-separated from~$\epsilon$ at~$S$} if~${\bigcup \Fcal}$ is.
        \item We say that~$\Fcal$ \emph{strongly agrees} about~${\partial(S)}$ if~${\bigcup \Fcal}$ does.
    \end{itemize}
\end{definition}

For ease of presentation, when a $G$-tribe $\mathcal{F}$ strongly agrees about $\partial(S)$ we will write $\leq_{\mathcal{F},e}$ for $\leq_{\bigcup \mathcal{F},e}$.

\begin{remark}
    \label{r:hereditaryprop}
    The properties of weakly agreeing about~$E$, being well-separated from~$\epsilon$, and strongly agreeing about~$\partial(S)$ are all preserved under taking subsets, and hence under taking flat subtribes. 
\end{remark}

Note that by the pigeon hole principle for thick $G$-tribes, given a finite edge set~${E \subseteq E(T)}$, any thick $G$-tribe~$\Fcal$ concentrated at~$\epsilon$ has a thick (flat) subtribe which weakly agrees about~$E$.

The next few lemmas show that, with some slight modification, we may restrict to a further subtribe which strongly agrees about~$E$ and is also well-separated from~$\epsilon$. 

\begin{definition}[{\cite[Lemma 3.5]{BEEGHPTII}}]
    \label{d:linear}
    Let~$\omega$ be an end of a graph~$G$. 
    We say~$\omega$ is \emph{linear} if~${\RG(\mathcal{R})}$ is a path for every finite family~$\mathcal{R}$ of disjoint $\omega$-rays.
\end{definition}

\begin{lemma}[{\cite[Lemma 10.1]{BEEGHPTII}}]
    \label{l:linearsubtribe}
    Let~$\epsilon$ be a non-pebbly end of~$\Gamma$ and let~$\Fcal$ be a thick $G$-tribe, such that for every~${H \in \bigcup \Fcal}$, there is an end~${\omega_H \in \Omega(G)}$ such that~${H(\omega_H) = \epsilon}$. 
    Then there is a thick flat subtribe~$\Fcal'$ of~$\Fcal$ such that~$\omega_H$ is linear for every~${H \in \bigcup \Fcal'}$.
\end{lemma}

\begin{corollary}
    \label{c:linearend}
    Let~$G$ be a connected locally finite graph with an extensive tree-decomposi{-}tion ${(T,\mathcal {V})}$, $S$ an initial subtree of~$T$, and let~$\Fcal$ be a thick $G$-tribe which is concentrated at a non-pebbbly end~$\epsilon$ of a graph~$\Gamma$ and weakly agrees about~${\partial(S)}$. 
    Then~$\omega_e$ is linear for every~${e \in \partial_\epsilon(S)}$.
\end{corollary}

\begin{proof}
    For any~${e \in \partial_\epsilon(S)}$ apply Lemma~\ref{l:linearsubtribe} to~$\Fcal$ with~${\omega_H = \omega_e}$ for each~${H \in \bigcup \Fcal}$. 
\end{proof}

\begin{lemma}\label{lem:stronglyagreesubtribe}
    Let~$G$ be a connected locally finite graph with an extensive tree-decomposition ${(T, \mathcal{V})}$ and let~$S$ be an initial subtree of~$T$ with~${\partial(S)}$ finite.
    Let~$\Fcal$ be a thick $G$-tribe in a graph~$\Gamma$, which weakly agrees about~${\partial(S) \subseteq E(T)}$, concentrated at a half-grid-like end~$\epsilon$ of~$\Gamma$. 
    Then~$\Fcal$ has a thick flat subtribe~$\Fcal'$ so that~$\Fcal'$ strongly agrees about~${\partial(S)}$. 
\end{lemma}

\begin{proof}
    Since~$\epsilon$ is half-grid-like, there is some~${N \in \Nbb}$ 
    as in Lemma~\ref{l:halfgridstructure}. 
    Then, by 
    Remark~\ref{rem:core-remarks}\ref{rem:non-core-bounded}, 
    given any family of 
    disjoint $\epsilon$-rays, at least~${m-2N-2}$ of them are core rays. 
    Thus, since all members of a layer~$F$ of~$\Fcal$ are disjoint, at least~${|F|-2N-2}$ members of~$F$ do not contain any $\epsilon$-ray which is not core. 
    Thus, there is a thick flat subtribe~$\Fcal^*$ of~$\Fcal$ such that all $\epsilon$-rays in members of~$\Fcal^*$ are core. 
    
    Given a member~$H$ of~$\Fcal^*$ and~${e \in \partial_\epsilon(S)}$, we consider the order~$\leq_{H,e}$ induced on~${S(e)}$ by the order~$\leq_\epsilon$ on~${H^{\downarrow}(\mathcal{R}_e)}$. 
    Let~$O_e$ be the set of potential orders on~${S(e)}$ which is finite since~${S(e)}$ is finite\footnote{Note that there are in fact at most two orders of~${S(e)}$ induced by one of the members of~$\Fcal^*$ since~$\omega_e$ is linear by Corollary~\ref{c:linearend}.}. 
    Consider the colouring~${c \colon \bigcup \Fcal^* \to \prod_{e \in \partial_\epsilon(S)} O_e}$ where we map every~$H$ to the product of the orders~$\leq_{H,e}$ it induces. 
    By the pigeon hole principle for thick G-tribes, Lemma~\ref{Lem_finitechoice}, there is a monochromatic, thick, flat $G$-subtribe~$\Fcal'$ of~$\Fcal^*$. 
    We can now set $\le _{\Fcal',e} := \leq_{H,e}$ for some~${H \in \Fcal'}$. 
    Then, by Remark~\ref{r:hereditaryprop}, this order~$\leq_{\Fcal',e}$ witnesses that~$\Fcal'$ is a thick flat subtribe of~$\Fcal$ which strongly agrees about~${\partial(S)}$. 
\end{proof}

\begin{lemma}
    \label{l:push-away}
    Let~$G$ be a connected locally finite graph with an extensive tree-decomposi{-}tion~${(T, \mathcal{V})}$. 
    Let~${H \subseteq \Gamma}$ be a tidy~$IG$ and~$\epsilon$ an end of~$\Gamma$. 
    Let~$e$ be an edge of~$T$ such that~${H(\omega_e) \neq \epsilon}$. 
    Then there is a finite set~${X \subseteq V(\Gamma)}$ 
    such that for every finite~${X' \supseteq X}$, there exists a push-out ${H_e = H(G[A(e)]) \oplus_{H^\downarrow(\Rcal_e)} W_e}$ of~$H$ along~$e$ to some depth~${n \in \Nbb}$ so that ${C_\Gamma(X', H(\omega_e)) \neq C_\Gamma(X', \epsilon)}$ 
    and ${W_e \subseteq C_\Gamma(X', H(\omega_e))}$.
\end{lemma}

\begin{proof}
    Let~${X \subseteq V(\Gamma)}$ be a finite vertex set such that~${C_\Gamma(X, H(\omega_e)) \neq C_\Gamma(X, \epsilon)}$. Then ${C_\Gamma(X', H(\omega_e)) \neq C_\Gamma(X', \epsilon)}$ holds for any finite vertex set~${X' \supseteq X}$. Furthermore, since~$X'$ is finite, there are only finitely many~${v \in V(G)}$ whose branch sets~${H(v)}$ meet~$X'$. By extensiveness, every vertex of~$G$ is contained in only finitely many parts of the tree-decomposition, and so there exists an $n \in \Nbb$ such that whenever $e' \in E(T_{e^+})$ is such that ${\dist(e^-,e'^-)\geq n}$, then
    \[
        H(G[B(e')]) \cap X' = \emptyset, \; \text{and so} \; H(G[B(e')]) \subseteq C_\Gamma(X', H(\omega_e)).
    \]
    
    Since ${(T,\mathcal{V})}$ is an extensive tree-decomposition, there is a witness~$W$ of the self-similarity of~${B(e)}$ at distance at least~$n$. 
    Then by Definition~\ref{d:pushout} and Lemma~\ref{lem:pushout2}, there is a push-out ${H_e = H(G[A(e)]) \oplus_{H^\downarrow(\Rcal_e)} H(W)}$ of~$H$ along~$e$ to depth~$n$.
    Let~${W_e = H(W)}$, then by Definition~\ref{d:pushout},
    $V(W_e) \subseteq V(H(G[B(e')]))\subseteq C_\Gamma(X', H(\omega_e))$.
\end{proof}

\begin{lemma}
	\label{l:induction-start}
    Let~$G$ be a connected locally finite graph with an extensive tree-decomposition ${(T,\mathcal {V})}$ with root~$r \in T$. 
    Let~$\Gamma$ be a graph and~$\Fcal$ a thick $G$-tribe concentrated at a half-grid-like end~$\epsilon$ of~$\Gamma$. 
    Then there is a thick subtribe~$\Fcal'$ of~$\Fcal$ such that 
    \begin{enumerate}[label=(\arabic*)]
    	\item \label{item:ind-start-1} $\Fcal'$ is concentrated at~$\epsilon$.
        \item \label{item:ind-start-2} $\Fcal'$ strongly agrees about~${\partial (\{r\})}$.
        \item \label{item:ind-start-3} $\Fcal'$ is well-separated from~$\epsilon$ at~${\{r\}}$.
    \end{enumerate}
\end{lemma}

\begin{proof}
    Since~$T$ is locally finite, also ${d(r)}$ is finite, and, by choosing a thick flat subtribe of~$\Fcal$, we may assume that~$\Fcal$ weakly agrees about~${\partial(\{r\})}$. 
    Moreover, by Lemma~\ref{lem:stronglyagreesubtribe}, we may even assume that~$\Fcal$ strongly agrees about~${\partial(\{r\})}$. Using Lemma~\ref{lem_subtribesinheritconcentration}(2), this $\Fcal$ would then satisfy \ref{item:ind-start-1} and \ref{item:ind-start-2}. So, it remains to arrange for \ref{item:ind-start-3}:
    
    For every member~$H$ of~$\Fcal$, and for every~${e \in \partial_{\neg \epsilon}(\{r\})}$, there exists, by Lemma~\ref{l:push-away}, a finite set~${X_e \subseteq V(\Gamma)}$, such that for every finite vertex set~${X' \supseteq X_e}$ there is a push-out ${H_e = H(G[A(e)]) \oplus_{H^\downarrow(\Rcal_e)} W_e}$ of~$H$ along~$e$, so that
    ${C_\Gamma(X', H(\omega_e)) \neq C_\Gamma(X', \epsilon)}$ and ${W_e \subseteq C_\Gamma(X', H(\omega_e))}$.
    Let~$X$ be the union of all these~$X_e$ together with~${H(\{r\})}$. 
    For each~${e \in \partial_{\neg \epsilon}(\{r\})}$, let~$H_e$ be the push-out whose existence is guaranteed by the above with respect to this set~$X$.
    
    Let us define an~$IG$ 
    \[
        H' := \bigcup_{e \in \partial_{\neg \epsilon} (\{r\})} \mkern-18mu H_e \left( \{r\}^{\epsilon} \cup T_{e^+} \right).
    \]
    It is straightforward, although not quick, to check that this is indeed an~$IG$ and so we will not do this in detail. 
    Briefly, this can be deduced from multiple applications of Definition~\ref{d:amalgamation}, and, since each ${H_e(G[A(e)])}$ extends~${H(G[A(e)])}$ fixing~${A(e) \setminus S(e)}$, all that we need to check is that the extra vertices added to the branch sets of vertices in~${S(e)}$ are distinct for each edge~$e$. 
    However, this follows from Definition~\ref{d:pushout}, since these vertices come from~${\bigcup H^\downarrow(\Rcal_e)}$ and the rays~$R_{e,s}$ and~$R_{e',s'}$ are disjoint except in their initial vertex when~${s = s'}$. 
    Let~$\Fcal'$ be the tribe given by~${\{F' \colon F \in \Fcal \}}$, where~${F' = \{ H' \colon H \in F\}}$ for each~${F \in \Fcal}$. 
    We claim that~$\Fcal'$ satisfies the conclusion of the lemma.
    
    Firstly, by Lemma~\ref{lem_subtribesinheritconcentration}(2), $\Fcal'$ is concentrated at~$\epsilon$, i.e.~\ref{item:ind-start-1} holds. Next, we claim that~$\mathcal{F}'$ strongly agrees about~${\partial(\{r\})}$. 
    Indeed, by construction for each~${e \in \partial_{\neg \epsilon}(\{r\})}$ we have ${W_e \subseteq C_\Gamma(X, H(\omega_e))}$, and hence~$\omega_e$ is cut from~$\epsilon$ according to~$H'$. 
    Furthermore, by construction ${H(\{r\}^{\epsilon}) \setminus X = H'(\{r\}^{\epsilon}) \setminus X}$ and so~$\omega_e$ converges to~$\epsilon$ according to~$H'$ for every~${e \in \partial_{\epsilon}(\{r\})}$. 
    In fact, ${H^{\downarrow}(\mathcal{R}_e) = H'^{\downarrow}(\mathcal{R}_e)}$ for every~${e \in \partial_{ \epsilon}(\{r\})}$. 
    Finally, since ${H' \subseteq H}$, and~$\Fcal$ strongly agrees about~${\partial(\{r\})}$, it follows that every $\epsilon$-ray in~$H'$ is in $\core(\epsilon)$, and so \ref{item:ind-start-2} holds.
    
    It remains to show that~$\Fcal'$ is well-separated from~$\epsilon$ at~${\{r\}}$.
    However, $H'(\{r\}^{\neg \epsilon}) \setminus \bigcup_{e \in \partial_{\neg \epsilon}(\{r\})} W_e$ is finite, and each $W_e$ is separated from $\epsilon$ by $X$. Hence, there is some finite set $Y$ separating $H'(\{r\}^{\neg \epsilon})$ from $\epsilon$, and so \ref{item:ind-start-3} holds.
\end{proof}

\begin{lemma}[Well-separated push-out]
    \label{l:separatedpushout}
    Let~$G$ be a connected locally-finite graph with an extensive tree-decomposition~${(T, \mathcal{V})}$. 
    Let~${H \subseteq \Gamma}$ be a tidy~$IG$ and~$\epsilon$ an end of~$\Gamma$. 
    Let~$S$ be a finite initial subtree of~$T$, such that~${\{H\}}$ is well-separated from~$\epsilon$ at~$S$,
    and let~${f \in \partial_\epsilon(S)}$.
    Then there exists exists a push-out~$H'$ of~$H$ along~$f$ to depth~$0$ (see Definition~\ref{d:pushout}) such that~${\{H'\}}$ is well-separated from~$\epsilon$ at~${\tilde S := S + f \subseteq T}$.
\end{lemma}

\begin{proof}
    Let ${X' \subseteq V(\Gamma)}$ be a finite set with~${H(S^{\neg\epsilon}) \cap C_\Gamma(X', \epsilon) = \emptyset}$. 
    If ${\partial_{\neg \epsilon} (\tilde S) \setminus \partial(S) = \emptyset}$, then~${H' = H}$ satisfies the conclusion of the lemma, hence we may assume that~${\partial_{\neg \epsilon} (\tilde S) \setminus \partial(S)}$ is non-empty.
    
    By applying Lemma~\ref{l:push-away} to every~${e \in \partial_{\neg \epsilon} (\tilde S) \setminus \partial(S)}$, we obtain a finite set~${X \supseteq X'}$ and a family ${(H_e \colon e \in  \partial_{\neg \epsilon} (\tilde S) \setminus \partial(S))}$ where each~${H_e = H(G[A(e)]) \oplus_{H^\downarrow(\Rcal_e)} W_e}$ is a push-out of~$H$ along~$e$ such that~${W_e \subseteq C_{\Gamma}(X,H(\omega_e)) \neq C_{\Gamma}(X,\epsilon)}$.
    
    Let
    \[
        H' :=  \mkern-24mu \bigcup_{e \in \partial_{\neg \epsilon} (\tilde S)\setminus \partial(S)} \mkern-24mu H_e \left( S^{\epsilon} \cup T_{e^+} \right).
    \]
    As before, it is straightforward to check that~$H'$ is an~$IG$, and that~$H'$ is a push-out of~$H$ along~$f$ to depth~$0$. 
    We claim that~$H'$ is well-separated from~$\epsilon$ at~${\tilde S}$. 
    
    Since $X'$ separates $H(S^{\neg \epsilon})$ from $\epsilon$, and $\partial_{\neg \epsilon} (\tilde S)\setminus \partial(S)$ is finite, it will be sufficient to show that for each $e \in \partial_{\neg \epsilon} (\tilde S)\setminus \partial(S)$, there is a finite set $X_e$ which separates $H'(G[B(e)])$ from $\epsilon$ in $\Gamma$. 
    However, by construction $X$ separates $W_e$ from $\epsilon$, and $H'(G[B(e)]) \setminus W_e$ is finite, and so the claim follows.
\end{proof}

The following lemma contains a large part of the work needed for our inductive construction. The idea behind the statement is the following: 
At step~$n$ in our construction, we will have a thick $G$-tribe~$\Fcal_n$ which agrees about~${\partial(T_n)}$, where~$T_n$ is an initial subtree of the decomposition tree~$T$ with finite~${\partial(T_n)}$, which will allow us to extend our~${IG(T_n)}$'s to~${IG(T_{n+1})}$'s, where~$T_{n+1}$ is a larger initial subtree of~$T$, again with finite~${\partial(T_{n+1})}$. 
In order to perform the next stage of our construction, we will need to `refine'~$\Fcal_n$ to a thick $G$-tribe~$\Fcal_{n+1}$ which agrees about~${\partial(T_{n+1})}$.

This would be a relatively simple application of the pigeon hole principle for $G$-tribes, Lemma~\ref{Lem_finitechoice}, except that, in our construction, we cannot extend by a member of~$\Fcal_{n+1}$ naively. 
Indeed, suppose we wish to use an~$IG$, say~$H$, to extend an~$IG(T_{n})$ to an~$IG(T_{n+1})$. 
There is some subgraph, ${H(T_{n+1}\setminus T_n)}$, of~$H$ which is an~${IG(T_{n+1}\setminus T_n)}$, however in order to use this to extend the~${IG(T_{n})}$ we first have to link the branch sets of the boundary vertices to this subgraph, and there may be no way to do so without using other vertices of~${H(T_{n+1}\setminus T_n)}$.

For this reason, we will ensure the existence of an `intermediate $G$-tribe'~$\Fcal^*$, which has the property that for each member~$H$ of~$\Fcal^*$, there are push-outs at arbitrary depth of~$H$ which are members of~$\Fcal_{n+1}$. 
This allows us to first link our~${IG(T_{n})}$ to some~${H \in \Fcal^*}$ and then choose a push-out~${H' \in \Fcal_{n+1}}$ of~$H$, such that~${H'(T_{n+1}\setminus T_n)}$ avoids the vertices we used in our linkage. 

\begin{lemma}[$G$-tribe refinement lemma]
    \label{lem:refinement-lemma}
    Let~$G$ be a connected locally finite graph with an extensive tree-decomposition~${(T,\mathcal {V})}$, let~$S$ be an initial subtree of~$T$ with~${\partial(S)}$ finite, and let~$\Fcal$ be a thick $G$-tribe of a graph~$\Gamma$ such that
    \begin{enumerate}[label=(\arabic*)]
    	\item \label{itemconcentratedF} $\Fcal$ is concentrated at a half-grid-like end~$\epsilon$; 
        \item \label{itemagreeF} $\Fcal$ strongly agrees about~${\partial (S)}$; 
        \item \label{itemSeparationF} $\Fcal$ is well-separated from~$\epsilon$ at~$S$. 
    \end{enumerate}
    Suppose~${f \in \partial_\epsilon(S)}$ and let~${\tilde S := S + f \subseteq T}$. 
    Then there is a thick flat subtribe~$\Fcal^*$ of~$\Fcal$ and a thick $G$-tribe~$\Fcal'$ in~$\Gamma$ with the following properties:
    \begin{enumerate}[label=(\roman*)]
    	\item \label{itemconcentrated} $\Fcal'$ is concentrated at~$\epsilon$.
        \item \label{itemagree} $\Fcal'$ strongly agrees about~${\partial (\tilde{S})}$.  
        \item \label{itemSeparation} $\Fcal'$ is well-separated from~$\epsilon$ at~$\tilde{S}$.
        \item \label{itemconsistentagree} ${\Fcal' \cup \Fcal}$ strongly agrees about~${\partial (S) \setminus \{f\}}$.
        \item \label{itemnested} $S^{\neg \epsilon}$ w.r.t.~$\Fcal$ is a subtree of~${\tilde{S}^{\neg \epsilon}}$ w.r.t.~$\Fcal'$. 
        \item \label{consistentpushingalong} For every~${F \in \Fcal^*}$ and every~${m \in \Nbb}$, there is an~${F' \in \Fcal'}$ such that for all~${H \in F}$, there is an~${H' \in F'}$ which is a push-out of~$H$ to depth~$m$ along~$f$. 
        \end{enumerate}
\end{lemma}

\begin{proof}
    For every member~$H$ of~$\Fcal$, consider a sequence~${(H^{(i)} \colon i \in \Nbb)}$, where~${H^{(i)}}$ is a push-out of~$H$ along~$f$ to depth at least~$i$. 
    After choosing a subsequence of~${(H^{(i)} \colon i \in \Nbb)}$ and relabelling (monotonically), we may assume that for each~$H$, the set ${\{H^{(i)} \colon i \in \Nbb\}}$ weakly agrees on~${\partial(\tilde S)}$, i.e.~for every~${e \in \partial(\tilde S)}$ either~${{H^{(i)}}^{\downarrow}(R) \in \epsilon}$ 
    for every~${R \in \omega_e}$ and all~$i$ or~${{H^{(i)}}^{\downarrow}(R) \notin \epsilon}$ for every~${R \in \omega_e}$ and all~$i$. 
    Note that a monotone relabelling preserves the property of~${H^{(i)}}$ being a push-out of~$H$ along~$f$ to depth at least~$i$. 

    This uniform behaviour of~${(H^{(i)} \colon i \in \Nbb)}$ on~${\partial(\tilde S)}$ for each member~$H$ of~$\Fcal$ gives rise to a finite colouring~${c \colon \bigcup \Fcal \to 2^{\partial(\tilde S)}}$. 
    By Lemma~\ref{Lem_finitechoice}, we may choose a thick flat subtribe~${\Fcal_1 \subseteq \Fcal}$ 
    such that~$c$ is constant on~${\bigcup \Fcal_1}$. 
    
    Recall that, by Corollary~\ref{c:linearend}, for every~${e \in \partial_\epsilon(\tilde S)}$ (w.r.t.~$\Fcal_1$), the ray graph~${\RG_G(\Rcal_e)}$ is a path. 
    We pick an arbitrary orientation of this path 
    and denote by~$\le_e$ the corresponding linear order on~$\Rcal_e$.
    
    Note that, since $\Fcal_1$ is a flat subtribe of $\mathcal{F}$ which strongly agrees about $\partial(S)$, every $\epsilon$-ray in every member~${H \in \bigcup\Fcal_1}$ is core. 
    Let us define, for each member~${H \in \bigcup\Fcal_1}$,
    \begin{align*}
        d_H &\colon \{H^{(i)} \colon i \in \Nbb \} \to \{-1,1\}^{\partial_\epsilon(\tilde S)}, \\
        \intertext{where}
        d_H(H^{(i)})_e &= 
        \begin{cases} 
            1 & \text{if $\leq_{\epsilon}$ agrees with the $\leq_e$}, \\ 
            -1 & \text{if $\leq_{\epsilon}$ agrees with the reverse order $\geq_e$ of $\leq_e$}.
        \end{cases}
    \end{align*}
    Since~$d_H$ has finite range, we may assume by Lemma~\ref{Lem_finitechoice}, after choosing a subsequence and relabelling, 
    that~$d_H$ is constant on ${\{H^{(i)}\colon i\in \Nbb\}}$ and that~$H^{(i)}$ is still a push-out of~$H$ along~$f$ to depth at least~$i$. 
    
    Now, consider ${d \colon \bigcup\Fcal_1 \to \{-1,1\}^{\partial_\epsilon(\tilde S)}}$,
    with ${d(H) = d_H(H^{(1)})}$ (${ = d_H(H^{(i)})}$ for all~${i \in \mathbb{N}}$).
    Again, we may choose a thick flat subtribe ${\Fcal_2 \subseteq \Fcal_1}$ such that~$d$ is constant on~$\Fcal_2$. 
    
    Since~$\mathcal{F}$ is well-separated from~$\epsilon$ at~$S$, we get that ${\{ H^{(i)} \colon H \in \mathcal{F} \}}$ is well-separated from~$\epsilon$ at~$S$. 
    So, we can now apply Lemma~\ref{l:separatedpushout} to each~$H^{(i)}$ to obtain~$H'^{(i)}$, yielding a collection which is well-separated from~$\epsilon$ at~$\tilde{S}$. 
    Note that~$H'^{(i)}$ is still a push-out of~$H$ along~$f$ to depth at least~$i$. 
    
    Now, let~${\Fcal^* = \Fcal_2}$ and~${\Fcal' = \{ \{H'^{(i)} \colon H \in F\} \colon i \in \Nbb, F\in \Fcal^* \}}$. 
    Let us verify that these satisfy~\ref{itemconcentrated}--\ref{consistentpushingalong}.
    $\Fcal^*$ is concentrated at~$\epsilon$ because it is a thick flat subtribe of~$\Fcal$ by Lemma~\ref{lem_subtribesinheritconcentration}. 
    By a comparison, layer by layer, since all members of~$\Fcal'$ are push-outs of members of~$\Fcal^*$ along~$f$, the tribe~$\Fcal'$ is also concentrated at~$\epsilon$, satisfying~\ref{itemconcentrated}. 
    
    Property~\ref{itemagree} is satisfied: 
    Since~$c$ and~$d$ are constant on~${\bigcup\Fcal_2}$ the collection of the~$H^{(i)}$ (for~${H \in \bigcup\Fcal_2}$) strongly agrees on~${\partial(\tilde S)}$, since we have chosen an appropriate subsequence in which~${d_H(H^{(i)})}$ is constant. 
    The~$H'^{(i)}$ are constructed such that this property is preserved.
    Property~\ref{itemSeparation} is immediate from the choice of~$H'^{(i)}$.
    Properties~\ref{itemconsistentagree} and~\ref{itemnested} follow from~\ref{itemagreeF} and the fact that every member of~$\Fcal'$ is a push-out of a member of~$\Fcal$ along~$f$.
    Property~\ref{consistentpushingalong} is immediate from the construction of~$\Fcal'$\!. 
\end{proof}

\section{The inductive argument}
\label{sec:countable-subtrees}

In this section we prove Theorem~\ref{t:nice}, our main result. 
Given a locally finite connected graph $G$ which admits an extensive tree-decomposition~${(T,\mathcal{V})}$ and a graph~$\Gamma$ which contains a thick $G$-tribe~$\mathcal{F}$, our aim is to construct an infinite family~${(Q_i \colon i \in \Nbb)}$ of disjoint $G$-minors in~$\Gamma$ inductively. 

Our work so far will allow us to make certain assumptions about~$\Fcal$. 
For example, by Lemma~\ref{l:concentrated}, we may assume that~$\Fcal$ is concentrated at some end~$\epsilon$ of~$\Gamma$, which, by Lemma~\ref{l:concentratedatthin}, we may assume is a thick end, and, by Lemma~\ref{c:pebblyubiq}, we may assume is not pebbly. 
Hence, by Theorem~\ref{t:trichotomy}, we may assume that~$\epsilon$ is either half-grid-like or grid-like. 

At this point our proof will split into two different cases, depending on the nature of~$\epsilon$. 
As we mentioned before, the two cases are very similar, with the grid-like case being significantly simpler. Therefore, we will first prove Theorem~\ref{t:nice} in the case where~$\epsilon$ is half-grid-like, and then in Section~\ref{s:gridlike} we will shortly 
sketch the differences for the grid-like case.

So, to briefly recap, in the following section we will be working under the standing assumptions that there is a thick $G$-tribe~$\mathcal{F}$ in~$\Gamma$, and an end~$\epsilon$ of $\Gamma$ such that 
\begin{itemize}
    \item[--] $\mathcal{F}$ is concentrated at $\epsilon$; 
    \item[--] $\epsilon$ is thick;
    \item[--] $\epsilon$ is half-grid-like.
\end{itemize}

\subsection{The half-grid-like case}
\label{sec:half-grid-like}

As explained in Section~\ref{s:sketch}, our strategy will be to take some sequence of initial subtrees ${S_1 \subseteq S_2 \subseteq S_3 \ldots}$ of~$T$ such that~${\bigcup_{i \in \Nbb} S_i = T}$, and to inductively build a collection of~$n$ inflated copies of~${G(S_n)}$, at each stage extending the previous copies. 
However, in order to ensure that we can continue the construction at each stage, we will require the existence of additional structure.

Let us pick an enumeration ${\{ t_i \colon i \geq 0 \}}$ of~${V(T)}$ such that~$t_0$ is the root of~$T$ and ${T_n := T[\{ t_i \colon  0\leq i \leq n \}]}$ is connected for every~${n \in \Nbb}$. 
We will not take the~$S_n$ above to be the subtrees~$T_n$, but instead the subtrees~$T_n^{\neg \epsilon}$ with respect to some tribe~$\Fcal_n$ that weakly agrees about~${\partial(T_n)}$. 
This will ensure that every edge in the boundary~${\partial(S_n)}$ will be in~${\partial_\epsilon(T_n)}$. 
For every edge~${e \in E(T)}$, let us fix a family~${\mathcal{R}_e = ( R_{e,s} \colon s \in S(e) )}$ of disjoint rays witnessing the self-similarity of the bough~${B(e)}$ towards an end~$\omega_e$ of~$G$, where~${\init(R_{e,s}) = s}$. 
By taking ${S_n = T_n^{\neg \epsilon}}$, we guarantee that for each edge in~${e \in \partial(S_n)}$, ${s \in S(e)}$, and every~${H \in \bigcup \Fcal_n}$, the ray~${H^{\downarrow}(R_{e,s})}$ is an $\epsilon$-ray.

Furthermore, since~${\partial(T_n)}$ is finite, we may assume by Lemma~\ref{lem:stronglyagreesubtribe}, that~$\Fcal_n$ strongly agrees about~${\partial(T_n)}$. 
We can now describe the additional structure that we require for the induction hypothesis.

At each stage of our construction we will have built some inflated copies of~${G(S_n)}$, which we wish to extend in the next stage. 
However, $S_n$ will not in general be a finite subtree, and so we will need some control over where these copies lie in~$\Gamma$ to ensure we have not `used up' all of~$\Gamma$. 
The control we will want is that there is a finite set of vertices~$X$, which we call a \emph{bounder}, that separates all that we have built so far from the end~$\epsilon$. 
This will guarantee, since~$\Fcal$ is concentrated at~$\epsilon$, that we can find arbitrarily large layers of~$\Fcal$ which are disjoint from what we have built so far.

Furthermore, in order to extend these copies in the next step, we will need to be able to link the boundary of our inflated copies of~${G(S_n)}$ to this large layer of~$\Fcal$. 
To this end we, will also want to keep track of some structure which allows us to do this, which we call an \emph{extender}. Let us make the preceding discussion precise.

\begin{definition}[Bounder, extender]
    Let~$\Fcal$ be a thick $G$-tribe, which is concentrated at~$\epsilon$ and strongly agrees about~${\partial(S)}$ for some initial subtree~$S$ of~$T$, and let~${k \in \Nbb}$.
    Let ${\mathcal{Q} = (Q_i \colon i \in [k])}$ be a family of disjoint inflated copies of~${G(S^{\neg \epsilon}})$ in~$\Gamma$ (note, $S^{\neg \epsilon}$ depends on~$\Fcal$).
    \begin{itemize}
        \item A \emph{bounder} for~$\mathcal{Q}$ is a finite set~$X$ of vertices in~$\Gamma$ separating each~$Q_i$ in~$\mathcal{Q}$ from~$\epsilon$, i.e.~such that 
            \[
                C(X,\epsilon) \cap \bigcup_{i=1}^k Q_i  = \emptyset. 
            \]
        \item For ${A \subseteq E(T)}$, let~${I(A,k)}$ denote the set~${\{ (e,s,i) \colon e \in A, s \in S(e), i \in [k] \}}$.
        \item An \emph{extender} for~$\mathcal{Q}$ is a family ${\mathcal{E} = ( E_{e,s,i} \colon (e,s,i) \in I(\partial_\epsilon(S),k))}$ of $\epsilon$-rays in~$\Gamma$ such that the graphs in~${\mathcal{E}^{-} \cup \mathcal{Q}}$ are pairwise disjoint and such that~${\init(E_{e,s,i}) \in Q_i(s)}$ for every~${(e,s,i) \in I(\partial_\epsilon(S),k)}$ (using the notation as in Definition~\ref{def_concat}).
        \item Given an extender~$\Ecal$, an edge~${e \in \partial_{\epsilon}(S)}$, and~${i \in [k]}$, we let
            \[
                \Ecal_{e,i} := ( E_{e,s,i} \colon s \in S(e)).
            \]
    \end{itemize}
\end{definition}

Recall that, since~$\epsilon$ is half-grid like, there is a partial order~$\leq_\epsilon$ defined on the core rays of~$\epsilon$, see Lemma~\ref{def:core-order}. 
Furthermore, if~$\Fcal$ strongly agrees about~${\partial(S)}$ then, as in Definition~\ref{d:tribes}, for each~${e \in \partial_\epsilon(S)}$, there is a linear order~$\leq_{\Fcal,e}$ on~${S(e)}$.

\begin{definition}[Extension scheme]
    Under the conditions above, we call a tuple $(X,\Ecal)$ an \emph{extension scheme} for $\mathcal{Q}$ if the following holds:
    \begin{enumerate}[label=(ES\arabic*)]
        \item\label{item:ES-defs} $X$ is a bounder for~$\mathcal{Q}$ and~$\Ecal$ is an extender for~$\mathcal{Q}$;
        \item\label{item:ES-core} $\mathcal{E}$ is a family of core rays;
        \item\label{item:ES-correct} the order~$\leq_\epsilon$ on~$\Ecal_{e,i}$ (and thus on~$\Ecal_{e,i}^-$) agrees with the order induced by~$\leq_{\mathcal{F},e}$ on~$\Ecal_{e,i}^-$ for all~${e \in \partial_\epsilon(S)}$ and~${i \in [k]}$;
        \item\label{item:ES-interval} the sets~$\mathcal{E}_{e,i}^-$ are intervals 
            with respect to~$\leq_\epsilon$ on~${\mathcal{E}^-}$ for all~${e \in \partial_\epsilon(S)}$ and~${i \in [k]}$. 
    \end{enumerate}
\end{definition}

We will in fact split our inductive construction into two types of extensions, which we will do on odd and even steps respectively.

In an even step~${n = 2k}$, starting with a $G$-tribe~$\mathcal{F}_k$, $k$ disjoint inflated copies ${(Q_{i}^n \colon i \in [k])}$ of~${G(T_k^{\neg \epsilon})}$, and an appropriate extension scheme, we will construct~$Q_{k+1}^n$, a further disjoint inflated copy of~${G(T_k^{\neg \epsilon})}$, and an appropriate extension scheme for everything we built so far.

In an odd step~${n = 2k - 1}$ (for~${k \geq 1}$), starting with the same $G$-tribe~$\mathcal{F}_{k-1}$ from the previous step, $k$ disjoint inflated copies of~${G(T_{k-1}^{\neg \epsilon})}$, and an appropriate extension scheme, we will refine to a new $G$-tribe~$\mathcal{F}_{k}$, which strongly agrees on~${\partial(T_k)}$, extend each copy~$Q_i^n$ of~${G(T_{k-1}^{\neg \epsilon})}$ to a copy~$Q_i^{n+1}$ of~${G(T_{k}^{\neg \epsilon})}$ for~${i \in [k]}$, and construct an appropriate extension scheme for everything we built so far.

So, we will assume inductively that for some ${n \in \Nbb_0}$, with~${\rho := \lfloor n/2 \rfloor}$ and~${\sigma := \lceil n/2 \rceil}$ 
we have:
\begin{enumerate}[label=(I\arabic*)]
    \item \label{item:induction1} a thick $G$-tribe~$\mathcal{F}_{\rho}$ in~$\Gamma$ which
        \begin{itemize}
            \item is concentrated at~$\epsilon$;
            \item strongly agrees about~${\partial(T_\rho)}$;
            \item is well-separated from~$\epsilon$ at~$T_\rho$; 
            \item whenever~${k < l \leq \rho}$, the tree~$T_k^{\neg\epsilon}$ with respect to~$\mathcal{F}_k$ is a subtree of~$T_l^{\neg\epsilon}$ with respect to~$\mathcal{F}_l$.
        \end{itemize}
    \item \label{item:induction2} 
        a family ${\mathcal{Q}_n = ( Q_i^n \colon i \in [\sigma] )}$ of~$\sigma$ pairwise disjoint inflated copies of~${G(T^{\neg \epsilon}_{\rho})}$ (where $T^{\neg \epsilon}_{\rho}$ is considered with respect to~$\mathcal{F}_\rho$) in~$\Gamma$;\\
        if~${n \geq 1}$, we additionally require that~$Q^n_i$ extends~$Q^{n-1}_i$ for all~${i \leq \sigma-1}$;
    \item \label{item:induction3} an extension scheme~${(X_n,\Ecal_n)}$ for~$\mathcal{Q}_n$;
    \item \label{item:induction4} if~$n$ is even and~${\partial_\epsilon(T_\rho) \neq \emptyset}$, we require that there is a set~$\mathcal{J}_\rho$ of disjoint core $\epsilon$-rays disjoint to~$\mathcal{E}_n$, with ${|\mathcal{J}_\rho| \geq (|\partial_\epsilon(T_\rho)|+1) \cdot |\mathcal{E}_n|}$.
\end{enumerate}

Suppose we have inductively constructed~$\mathcal{Q}_n$ for all~${n \in \Nbb}$. 
Let us define~${H_i := \bigcup_{n \geq 2i-1} Q^n_i}$. 
Since $T_k^{\neg\epsilon}$ with respect to~$\mathcal{F}_k$ is a subtree of~$T_l^{\neg\epsilon}$ with respect to~$\mathcal{F}_l$ for all~${k<l}$, we have that ${\bigcup_{n \in \Nbb} T^{\neg \epsilon}_n = T}$ (where we considered~$T^{\neg \epsilon}_n$ w.r.t.~$\mathcal{F}_n$), and due to the extension property~\ref{item:induction2}, the collection ${(H_i \colon i \in \Nbb)}$ is an infinite family of disjoint $G$-minors, as required. 

So let us start the construction. 
To see that our assumptions can be fulfilled for the case~${n = 0}$, we first note that since~${T_0 = t_0}$, by Lemma~\ref{l:induction-start} there is a thick subtribe~$\Fcal_0$ of~$\Fcal$ which satisfies~\ref{item:induction1}. 
Let us further take~${\mathcal{Q}_0 = \mathcal{E}_0 = X_0 = \mathcal{J}_0 = \emptyset}$. 

\vspace{0.2cm}

The following notation will be useful throughout the construction. 
Given~${e \in E(T)}$ and some inflated copy~$H$ of~$G$, recall that~${H^{\downarrow}(\mathcal{R}_e)}$ denotes the family~${(H^{\downarrow}(R_{e,s}) \colon s \in S(e))}$. 
Given a $G$-tribe~$\mathcal{F}$, a layer~${F \in \mathcal{F}}$ and a family of disjoint rays~$\Rcal$ in~$G$ we will write ${F^{\downarrow}(\mathcal{R}) = ( H^{\downarrow}(R) \colon H \in F, R \in \Rcal)}$.

\vspace{.5cm}
\noindent{\bf Construction part 1: ${n=2k}$ is even}

\vspace{.5cm}
\noindent{\bf Case 1: ${\partial_\epsilon (T_k) = \emptyset}$.}

In this case, ${T^{\neg \epsilon}_k = T}$ 
and so picking any member~${H \in \Fcal_k}$ with ${H \subseteq C(X_n,\epsilon)}$ and setting ${Q_{k+1}^{n+1} = H(T^{\neg \epsilon}_k)}$ gives us a further inflated copy of~${G(T^{\neg \epsilon}_k)}$ disjoint from all the previous ones. 
We set~${Q^{n+1}_i = Q^{n}_i}$ for all~${i \in [k]}$ and~${\mathcal{Q}_{n+1} = ( Q^{n+1}_i \colon i \in [k+1])}$. 
Since~$\mathcal{F}_k$ is well-separated from~$\epsilon$ at~$T_k$, there is a suitable bounder~${X_{n+1}\supseteq X_n}$ for~$\mathcal{Q}_{n+1}$. 
Then ${(X_{n+1}, \emptyset)}$ is an extension scheme for~$\mathcal{Q}_{n+1}$ while~$\mathcal{F}_k$ remains unchanged.

\vspace{.5cm}
\noindent \textbf{Case 2: ${\partial_\epsilon (T_k) \neq \emptyset}$.} (See Figure~\ref{f:adding})

Consider the family ${\mathcal{R}^{-} := \bigcup \{ \mathcal{R}_e^{-} \colon e \in \partial_\epsilon(T_k)\}}$. 
Moreover, set~${\mathcal{C} := \mathcal{E}_n^- \cup \mathcal{J}_k}$ and consider~$\overline{\mathcal{C}}$ as in Definition~\ref{def:central-core}. 
Let ${Y \subseteq C(X_n,\epsilon)}$ be a finite subgraph, which is a transition box between~$\overline{\mathcal{E}_n^-}$ and~$\overline{\mathcal{C}}$ after~$X_n$ as in Lemma~\ref{l:transition-box}. 
Let~$\mathcal{F}'$ be a flat thick $G$-subtribe of~$\mathcal{F}_k$, such that each member of~$\mathcal{F}'$ is contained in~${C(X_n \cup V(Y), \epsilon)}$, which exists, by Lemma~\ref{lem_subtribesinheritconcentration}, since both~$X_n$ and~$V(Y)$ are finite. 

Let ${F \in \mathcal{F}'}$ be large enough such that 
we may apply Lemma~\ref{l:link} to find a transitional linkage~$\mathcal{P}$, such that~${\bigcup \mathcal{P} \subseteq C(X_n \cup V(Y), \epsilon)}$, from~${\overline{\mathcal{C}}}$ to~${F^{\downarrow}(\Rcal^{-})}$ after~${X_n \cup V(Y)}$ avoiding some member~${H \in F}$. 
Note that, since~$X_n$ is a bounder and ${\bigcup \mathcal{P} \subseteq C(X_n \cup V(Y), \epsilon)}$, we get that each element of~$\mathcal{P}$ is disjoint from all~$\mathcal{Q}_n$ and~$Y$.

Let
\[
    Q^{n+1}_{k+1} := H(T^{\neg \epsilon}_k).
\]
Note that $Q^{n+1}_{k+1}$ is an inflated copy of~${G(T^{\neg \epsilon}_k)}$. 
Moreover, let ${Q^{n+1}_{i} := Q^{n}_{i}}$ for all~${i \in [k]}$ and~${\mathcal{Q}_{n+1} := ( Q_{i}^{n+1} \colon i \in [k+1] )}$, 
yielding property~\ref{item:induction2}. 

Since~$\mathcal{F}_k$ is well-separated from~$\epsilon$ at~$T_k$, and~${H \in \bigcup \Fcal_k}$, there is a finite set~${X_{n+1} \subseteq V(\Gamma)}$ containing ${X_{n} \cup V(Y)}$, such that~${C(X_{n+1},\epsilon) \cap Q^{n+1}_{k+1} = \emptyset}$. 
This set~$X_{n+1}$ is a bounder for~$\mathcal{Q}_{n+1}$.

Since~$\mathcal{P}$ is transitional, Lemma~{\ref{l:core-preserving}} implies that the linkage is preserving on~$\mathcal{C}$. 
Since all rays in~${F^{\downarrow}(\mathcal{R}^-)}$ are core rays, we have that~$\le_{\epsilon}$ is a linear order on~${F^{\downarrow}(\mathcal{R}^-)}$. 
Moreover, for each~${e \in \partial_\epsilon(T_k)}$, the rays in~${H^{\downarrow}(\mathcal{R}^-_{e})}$ correspond to an interval in this order. 
Thus, deleting these intervals from~${F^{\downarrow}(\mathcal{R}^-)}$ leaves behind at most~${|\partial_\epsilon(T_k)|+1}$ intervals in~${F^{\downarrow}(\mathcal{R}^-)}$ (with respect to~$\le_{\epsilon}$) which do not contain any rays in~${H^{\downarrow}(\mathcal{R}^-)}$.
Since~${|\mathcal{J}_k| \geq (|\partial_\epsilon(T_k)|+1) \cdot |\mathcal{E}_n|}$, by the pigeonhole principle there is one such 
interval on~${F^{\downarrow}(\mathcal{R}^-)}$ that
\begin{itemize}
    \item[--] does not contain rays in~${H^{\downarrow}(\mathcal{R})}$; 
    \item[--] where a subset ${\mathcal{P}' \subseteq \mathcal{P}}$ of size~$|\mathcal{E}_n^-|$ links a corresponding subset~$\mathcal{A}$ of~$\mathcal{C}$ to a set of rays~$\mathcal{B}$ in that interval.
\end{itemize}
By Lemmas~\ref{l:transition-box} , ~{\ref{lem:core-preserving2}} and~\ref{l:core-preserving}, and Remark~\ref{rem:core-preserving}\ref{rem:core-preserving-sub},
there is a linkage~$\mathcal{P}''$ from~${\overline{\mathcal{E}_n^{-}}}$ to~${\mathcal{A}}$ contained in~$Y$ which is preserving on~$\mathcal{E}_n^{-}$. 

For~${e \in \partial_\epsilon(T_k)}$ and~${s \in S(e)}$, define
\[
    E^{n+1}_{e,s,k+1} = H^{\downarrow}(R_{e,s}) \text{ for the corresponding ray } R_{e,s} \in \mathcal{R}_e.
\]
Moreover for each~${i \in [k]}$, we define 
\[
    E_{e,s,i}^{n+1} = (E^{n}_{e,s,i} \circ_{\mathcal{P}''} \mathcal{A}) \circ_{\mathcal{P}'} \mathcal{B},
\]
noting that~$\mathcal{P}''$ is also a linkage from~$\mathcal{E}_n$ to~${\mathcal{A}}$.

By construction, all these rays are, except for their first vertex, disjoint from~$\mathcal{Q}_{n+1}$. 
Moreover, ${\mathcal{E}_{n+1} := ( E^{n+1}_{e,s,i} \colon (e,s,i) \in I(\partial_\epsilon(T_k),k+1) )}$ is an extender for~$\mathcal{Q}_{n+1}$.
Note that each ray in~$\mathcal{E}_{n+1}$ shares a tail with a ray in~${F^{\downarrow}(\mathcal{R}^{-})}$.

We claim that ${(X_{n+1}, \mathcal{E}_{n+1})}$ is an extension scheme for~$\mathcal{Q}_{n+1}$ and hence property~\ref{item:induction3} is satisfied. 
Since every ray in~$\mathcal{E}_{n+1}$ 
has a tail which is also a tail of a ray in~$F^{\downarrow}(\mathcal{R}^{-})$, property~\ref{item:ES-core} is satisfied by Remark~{\ref{rem:core-remarks}\ref{rem:core-tail}}. 
Since~$\mathcal{P}'$ is preserving on~$\mathcal{A'}$ and~$\mathcal{P}''$ is preserving on~$\mathcal{E}_n^-$, 
Remark~{\ref{rem:core-preserving}\ref{rem:core-preserving-concat}} implies that the linkage~${\Pcal'' + \Pcal'}$ is preserving on~$\mathcal{E}_n^-$. 
Hence, property~\ref{item:ES-correct} holds for each~${i \in [k]}$. 
Furthermore, since ${E^{n+1}_{e,s,k+1} = H^{\downarrow}(R_{e,s})}$ for each~${e \in \partial_\epsilon(T_k)}$ and~${s \in S(e)}$ and~$\mathcal{F}_k$ strongly agrees about~$\partial(T_k)$, it is clear that property~\ref{item:ES-correct} holds for~${i = k+1}$. 
Finally, property~\ref{item:ES-interval} holds for~${i=k+1}$ since for each~${e \in \partial_\epsilon(T_k)}$, the rays in~${H^{\downarrow}(\mathcal{R}_{e})}$ are an interval with respect to~$\leq_\epsilon$ on~$F^{\downarrow}(\mathcal{R}^-)$, and it holds for~${i \in [k]}$ by the fact that ${\Pcal'' + \Pcal'}$ is preserving on~$\mathcal{E}_n^-$ together with the fact that ${\Pcal'' + \Pcal'}$ links~$\mathcal{E}_n^-$ to an interval of~${F^{\downarrow}(\mathcal{R}^-)}$ containing no ray in~${H^{\downarrow}(\mathcal{R})}$.

Finally, note that~\ref{item:induction1} is still satisfied by~$\Fcal_k$ and~$T_k$, and~\ref{item:induction4} is vacuously satisfied. 

\begin{sidewaysfigurepage}
    \centering
    \resizebox{.8\textwidth}{!}{
        
\begin{tikzpicture}[scale=1.1]
    \colorlet{junkcolor}{purple}
    \tikzstyle{JunkRay}=[junkcolor,->]
    \colorlet{extendercolor}{green!80!black}
    \tikzstyle{ExtenderRay}=[extendercolor,->]
    \draw[junkcolor] (0, 2) node[above right] {$ \mathcal{J}_k $};
    \begin{scope}[yshift=2cm]
        \foreach \y in {0,...,-2}
            \draw[JunkRay] (0, 0.2*\y) -- +(16,0);
    \end{scope}

    \newcommand\Qin[1]{
        \draw[blue] (0,0) circle [radius=1];
        \draw[blue] (0:1) arc (70:0:-1);
        \draw[blue] (-5:1) arc (-60:-5:-1);
        \draw[blue] (0,0) node {$ Q_{#1}^n $};
        \begin{scope}
            \clip (0,0) circle [radius=1.8];
            \draw[red] (160:1.75) circle [radius=1];
        \end{scope}
        \draw[red] (160:1.75) node {$ Q_{#1}^n(T_k^{\neg\epsilon}) $};
        \draw[ExtenderRay] (.6,.6) -- (16, .6);
        \draw[ExtenderRay] (.8,.3) -- (16, .3);
        \draw[ExtenderRay] (.75,-.5) -- (16,-.5);
    }
    \begin{scope}[yshift=0cm]
        \Qin{1}
    \end{scope}
    \draw[extendercolor] (1, -1) node {$ \mathcal{E}_n $};

    \draw[blue] (0, -2) node {$ \vdots $};
    \draw[blue] (0, -3) node {$ \vdots $};

    \begin{scope}[yshift=-5cm]
        \Qin{k}
    \end{scope}

    \begin{scope}[yshift=-7cm]
        \foreach \y in {0,...,3}
            \draw[JunkRay] (0, 0.2*\y) -- +(16,0);
    \end{scope}

    \draw[red] (2, 3) rectangle ++(1, -11) node[pos=.5] {$ X_n $};

    \colorlet{transitioncolor}{olive}
    \filldraw[draw=transitioncolor,fill=white,opacity=.9] (4, 3) rectangle (7, -8);
    \draw[transitioncolor] (5.5, 2.5) node {transition box $ Y $};
    \draw[transitioncolor] (5.5, -2) node {\huge$ \mathcal{P}'' $};

    \draw[transitioncolor, thick] (4,.6) to[out=0,in=180] (7,2);
    \draw[transitioncolor, thick] (4,.3) to[out=0,in=180] (7,2-.2);
    \draw[transitioncolor, thick] (4,-.5) to[out=0,in=180] (7,2-.4);

    \draw[transitioncolor, thick] (4,.6-5) to[out=0,in=180] (7,.6-5);
    \draw[transitioncolor, thick] (4,.3-5) to[out=0,in=180] (7,-.5-5);
    \draw[transitioncolor, thick] (4,-.5-5) to[out=0,in=180] (7,-7+.6);

    \draw[magenta] (7, -10) node {$ F \in \mathcal{F}_k $};

    \newcommand\newH{
        \begin{scope}[yshift=-10cm]
            \draw[magenta] (0, 0) circle [radius=.5];
            \draw[magenta] (80:.5) arc (15:80:-.5);
            \draw[magenta] (90:.5) arc (-10:-90:.5);
            \begin{scope}
                \clip (0,0) circle [radius=1];
                \draw[red] (280:.9) circle [radius=.6];
            \end{scope}
                \draw[magenta,->] (.25,.3) -- ++(0, 13.25);
                \draw[magenta,->] (-.15,.35) -- ++(0, 13.25);
                \draw[magenta,->] (-.35,.2) -- ++(0, 13.25);
        \end{scope}
    }

    \begin{scope}[xshift=9cm]
        \newH
    \end{scope}
    \draw[magenta] (10, -10) node {$ \cdots $};
    \draw[magenta] (10, 3.25) node {$ \cdots $};
    \begin{scope}[xshift=11cm]
        \newH
    \end{scope}
    \draw[magenta] (12, -10) node {$ \cdots $};
    \draw[magenta] (12, 3.25) node {$ \cdots $};

    \draw[magenta] (14,-10) node {$ \cdots $};
    \draw[magenta] (14,3.25) node {$ \cdots $};
    \begin{scope}[xshift=15cm]
        \newH
    \end{scope}

    \filldraw[draw=orange,fill=white,opacity=.9] (8, 3) rectangle (15.5, -8);

    \colorlet{pprimecolor}{red!50!white}
    \tikzstyle{Pprime}=[draw=pprimecolor, double=black]
    \tikzstyle{Pprimedashedend}=[Pprime,dashed,preaction={draw,solid,Pprime,shorten >=.5cm}]
    \tikzstyle{dashedend}=[dashed,postaction={draw,solid,shorten >=.5cm}]

    \draw[orange] (9, -1) node {linkage};
    \draw (9.2, -2) node {\huge$ {\color{pprimecolor}\mathcal{P}'} {\,\subseteq\,} \mathcal{P} $};

    \draw[Pprime] (8,2) -| (9-.35,3);
    \draw[Pprime] (8,2-.2) -| (9.25,3);
    \draw[Pprime] (8,2-.4) -| (9.5,3);
    \draw (8,.6) -| (9.75,3);
    \draw (8,.3) -| (10,3);
    \draw (8,-.5) -| (10.2,3);

    \draw[Pprimedashedend] (8,.6-5) -| (10.6,-3);
    \draw[dashedend] (8,.3-5) -| (10.7,-3.2);
    \draw[Pprimedashedend] (8,-.5-5) -| (10.8,-4);

    \draw[Pprimedashedend] (8,-7+.6) -| (11.2,-4.5);
    \draw[dashedend] (8,-7+.4) -| (11.8,-5);
    \draw[dashedend] (8,-7+.2) -| (14.7,-5.2);
    \draw[dashedend] (8,-7)    -| (14.9,-5.3);

    \draw (11, 2.5) node {$ \cdots $}
          (11, 1.5) node {$ \cdots $};

    \begin{scope}[xshift=13cm]
        \fill[fill=white, opacity=.95] (-.7, 4) rectangle (.7, -10.55);
        \draw[green, dashed] (-.7,4) |- (.7, -10.55) -- (.7, 4);
        \newH
        \draw[magenta] (0, -10) node {$ H $};
    \end{scope}

\end{tikzpicture}
    }
    \caption{Adding a new copy when~${n=2k}$ is even.} 
    \label{f:adding}
\end{sidewaysfigurepage}

\noindent \textbf{Construction part 2: ${n=2k-1}$ is odd (for~${k \geq 1}$).}

Let~$f$ denote the unique edge of~$T$ between~$T_{k-1}$ and~${T_{k} \setminus T_{k-1}}$. 

\vspace{.5cm}
\noindent \textbf{Case 1:} ${f \notin \partial_{\epsilon} (T_{k-1})}$.

Let~${\Fcal_{k} := \Fcal_{k-1}}$. 
Since~$\mathcal{F}_{k-1}$ is well-separated from~$\epsilon$ at~$T_{k-1}$, it follows that~${e \in \partial_{\neg \epsilon}(T_k)}$ for every~${e \in \partial(T_k) \setminus \partial(T_{k-1})}$. 
Hence~${T^{\neg \epsilon}_{k} = T^{\neg \epsilon}_{k-1}}$ and~${\partial_\epsilon(T_{k-1}) = \partial_\epsilon(T_k)}$, and so~$\mathcal{F}_k$ is well-separated from~$\epsilon$ at~$T_k$ and we can simply take ${\mathcal{Q}_{n+1} := \mathcal{Q}_n}$, ${\mathcal{E}_{n+1} := \mathcal{E}_n}$, ${\mathcal{J}_{k} := \mathcal{J}_{k-1}}$ and ${X_{n+1} := X_n}$ to satisfy~\ref{item:induction1}, \ref{item:induction2}, \ref{item:induction3} and~\ref{item:induction4}.

\vspace{.5cm}
\noindent \textbf{Case 2:} ${f \in \partial_{\epsilon} (T_{k-1})}$. (See Figure~\ref{f:new})

By~\ref{item:induction1} we can apply Lemma~\ref{lem:refinement-lemma} to~$\Fcal_{k-1}$ and~$T_{k-1}$ in order to find a thick $G$-tribe~$\Fcal_{k}$ 
and a thick flat subtribe~$\mathcal{F}^*$ of~$\mathcal{F}_{k-1}$, both concentrated at~$\epsilon$, satisfying properties \ref{itemconcentrated}--\ref{consistentpushingalong} from that lemma. 
It follows that~$\mathcal{F}_{k}$ satisfies~\ref{item:induction1} for the next step.

Let~${F \in \mathcal{F}^*}$ be a layer of~$\mathcal{F}^*$ such that 
\[
    {|F| \geq (\partial_\epsilon(T_k) + 2) \cdot |I(\partial_\epsilon (T_k), k)|}
\]
and consider the rays~${F^{\downarrow}(\mathcal{R}_f)}$. 
Consider the rays in the extender corresponding to the edge~$f$, that is~${\mathcal{E}_f := (E_{f,s,i}^n \colon i \in [k], s \in S(f))}$. 
By Lemma~{\ref{lem:core-preserving2}}, there is, for every subset~$\mathcal{S}$ of~${F^{\downarrow}(\mathcal{R}_f)}$ of size~${|\mathcal{E}_f^-|}$, a transitional linkage~$\mathcal{P}$ 
from~${\mathcal{E}_f^- \subseteq \overline{\mathcal{E}_n^-}}$ to~${\mathcal{S} \subseteq \overline{F^{\downarrow}(\mathcal{R}_f)}}$ after~${X_n \cup \, \init(\mathcal{E}_n)}$, which is preserving on~$\mathcal{E}^-_f$. 

Let us choose~${H_1,H_2,\ldots,H_k \in F}$ and let~${\mathcal{S} = \left(H^{\downarrow}_i(R_{f,s}) \colon i \in [k], s\in S(f)\right)}$. 
Let~$\mathcal{P}$ be the linkage given by the previous paragraph, which we recall is preserving on~$\mathcal{E}^-_f$. 
Since for every~${i \leq k}$, the family ${\left(E^{n-}_{f,s,i} \colon s \in S(f)\right)}$ forms an interval in~$\mathcal{E}^-_n$ and the set~${H_i^{\downarrow}(\mathcal{R}_{f})}$ forms an interval in~${F^{\downarrow}(\mathcal{R}_f)}$, and furthermore the order~$\leq_\epsilon$ agrees with~$\leq_{\mathcal{F}_k,f}$ on~${S(f)}$, it follows that, after perhaps relabelling the~$H_i$, for every~${i \in [k]}$ and~${s \in S(f)}$, $\Pcal$ links~$E^{n-}_{f,s,i}$ to~${H^{\downarrow}_i(R_{f,s})}$.

Let~${Z \subseteq V(\Gamma)}$ be a finite set such that~${\top(\omega, R)}$ and~${\bot(\omega, R)}$ are separated by~$Z$ in~${\Gamma - V(R)}$ 
for all~${R \in F^{\downarrow}(\mathcal{R}_f)}$ (cf.~Lemma~\ref{lem:core-exchange}).

Since~${|F|}$ is finite and ${(T,\mathcal{V})}$ is an extensive tree-decomposition, there exists an~${m \in \Nbb}$ such that if~${e \in T_{f^+}}$ with~${\dist(f^-,e^-) = m}$, then~${H(B(e)) \cap \left({X_n \cup Z \cup V(\bigcup \mathcal{P})}\right) = \emptyset}$ for every~${H \in F}$. 
Let~${F' \in \mathcal{F}_{k}}$ be as in Lemma~\ref{lem:refinement-lemma}\ref{consistentpushingalong} for~$F$ with such an~$m$. 

Hence, by definition, for each~${H_i \in F}$ there is some~${H'_i \in F'}$ which is a push-out of~$H_i$ to depth~$m$ along~$f$, and so there is some edge~${e \in T_{f^+}}$ with~${\dist(f^-,e^-) = m}$ and some subgraph ${W_i \subseteq H(B(e))}$ which is an~${I\overline{G[B(f)]}}$ such that for each~${s \in S(f)}$, we have that~${W_i(s)}$ contains the first vertex of~$W_i$ on~${H_i^{\downarrow}(R_{f,s})}$.

For each~${i \in [k]}$ we construct~$Q_{i}^{n+1}$ from~$Q_{i}^{n}$ as follows. 
Consider the part of~$G$ that we want to add~${G(T_{k-1}^{\neg \epsilon})}$ to obtain~${G(T_k^{\neg \epsilon})}$, namely 
\[
    {D := \overline{G[B(f)]} \left[ V_{f^+} \cup 
    \bigcup \big\{ B(e) \colon {e \in \partial_{\neg \epsilon}(T_k) \setminus \partial_{\neg \epsilon}(T_{k-1})} \big\} \right]}. 
\]
Let~${K_i := W_i(D)}$. 
Note that this is an inflated copy of~$D$, and for each~${s \in S(f)}$ and each~${i \in [k]}$ the branch set~${K_i(s)}$ contains the first vertex of~$K_i$ on~${H_i^{\downarrow}(R_{f,s})}$. 

Note further that, by the choice of~$m$, all the~$K_i$ are disjoint to~$\mathcal{Q}_n$. 
Let~$x_{f,s,i}$ denote the first vertex on the ray~${H^{\downarrow}_i(R_{f,s})}$ in~$K_i$, and let 
\[
    O_{s,i} := (E^n_{f,s,i} \circ_{\mathcal{P}} F^{\downarrow}(\mathcal{R}_f)) x_{f,s,i},
\]
where as before we note that~$\mathcal{P}$ is also a linkage from~${\mathcal{E}_n}$ to~${F^{\downarrow}(\mathcal{R}_f)}$.

Then, if we let ${\mathcal{O}_i := (O_{s,i} \colon s \in S(f))}$ and ${\mathcal{O} = (O_{s,i} \colon s \in S(f), i\in [k])}$, we see that 
\[
    Q_{i}^{n+1} := Q_{i}^{n} \oplus_{\mathcal{O}_i} K_i
\]
(see Definition~\ref{d:amalgamation}) is an inflated copy of~${G(T_k^{\neg \epsilon})}$ extending~$Q_{i}^{n}$. 
Hence, 
\[
	{\mathcal{Q}^{n+1} := ( Q^{n+1}_i \colon i \in [k])}
\]
is a family satisfying~\ref{item:induction2}.

Since~$\mathcal{F}_k$ is well-separated from~$\epsilon$ at~$T_k$, and each~$K_i$ is a subgraph of the restriction of ${W_i \subseteq H'_i}$ to~$D$, for each~$K_i$, there is a finite set~$\hat{X}_i$ separating~$K_i$ from~$\epsilon$, and hence the set
\[
   X_{n+1} := X_n \cup \bigcup_{i \in [k]} \hat{X}_i \cup V \left( \bigcup \mathcal{O} \right)
\]
is a bounder for~$\mathcal{Q}^{n+1}$. 

For ${e \in \partial_\epsilon(T_{k-1}) \setminus \{f\}}$, ${s \in S(e)}$, and~${i \in [k]}$, we set
\[
    {E}^{n+1}_{e,s,i} = E^n_{e,s,i} \circ_\mathcal{P} F^{\downarrow}(\mathcal{R}_f),
\]
and set 
\[
    \mathcal{E}' := \left(E^{n+1}_{e,s,i} \colon (e,s,i) \in I\left(\partial_\epsilon(T_{k-1}) \setminus \{f\},k\right)\right)
\]

Moreover, for~${e \in \partial_\epsilon(T_k) \setminus \partial_\epsilon(T_{k-1})}$, ${s \in S(e)}$, and~${i \in [k]}$, we set 
\[
    {E}^{n+1}_{e,s,i} = H'^{\downarrow}_i(R_{e,s}), 
\]
and set 
\[
    \mathcal{E}'' := \left(E^{n+1}_{e,s,i} \colon (e,s,i) \in I\left(\partial_\epsilon(T_k)\setminus\partial_\epsilon(T_{k-1}),k\right)\right).
\]
Note that, by construction, any such ray~$E^{n+1}_{e, s, i}$ has its initial vertex in the branch set~$Q^{n+1}_i(s)$ and is otherwise disjoint to~${\bigcup \mathcal{Q}_{n+1}}$. 
We set~${\mathcal{E}_{n+1} := \mathcal{E}' \cup \mathcal{E}''}$.
It is easy to check that this is an extender for~$\mathcal{Q}_{n+1}$.

We claim that ${(X_{n+1}, \mathcal{E}_{n+1})}$ is an extension scheme. 
Property~\ref{item:ES-defs} is apparent. 
Since~$\mathcal{F}_k$ strongly agrees about~$\partial(T_k)$, every $\epsilon$-ray in an any member of~$\mathcal{F}_k$ is core. 
Then, since~$\mathcal{F}^*$ is a flat subtribe of~$\mathcal{F}_k$ and every ray in~$\mathcal{E}_{n+1}$ shares a tail with a ray in a member of~$\mathcal{F}_k$ or~$\mathcal{F}^*$, 
it follows by Remark~{\ref{rem:core-remarks}}~\ref{rem:core-tail} that all rays in~$\mathcal{E}_{n+1}$ are core rays, and so~\ref{item:ES-core} holds. 

For any~${e \in \partial_\epsilon(T_{k-1}) \setminus \{f\}}$ and~${i \in [k]}$, the rays~$\mathcal{E}_{n+1,e,i}$ are a subfamily of~$\mathcal{E}'$, obtained by transitioning from the family~$\mathcal{E}_{n,e,i}$ to~${F^{\downarrow}(\mathcal{R}_f)}$ along the linkage~$\mathcal{P}$. 
By the induction hypothesis, $\leq_\epsilon$ agreed with the order induced by~$\leq_{\mathcal{F}_{k-1},e}$ on~$\mathcal{E}_{n,e,i}$, 
and, since ${\mathcal{F}_k \cup \mathcal{F}_{k-1}}$ strongly agrees about~${\partial_\epsilon(T_{k-1}) \setminus \{f\}}$, this is also the order induced by~$\leq_{\mathcal{F}_{k},e}$. 
Hence, since~$\mathcal{P}$ is preserving, by Lemma~\ref{l:core-preserving}, it follows that the order induced by~$\leq_{\mathcal{F}_{k},e}$ on~$\mathcal{E}_{n+1,e,i}$ agrees with~$\leq_\epsilon$.

For for ${e \in \partial_\epsilon(T_k) \setminus \partial_\epsilon(T_{k-1})}$ and~${i \in [k]}$, the rays~$\mathcal{E}_{n+1,e,i}$ are ${( H'^{\downarrow}_i(R_{e,s}) \colon s \in S(e))}$. 
Since ${H'_i \in F' \in \mathcal{F}_k}$ and~$\mathcal{F}_k$ strongly agrees about~$\partial(T_k)$, it follows that the order induced by~$\leq_{\mathcal{F}_{k},e}$ on~$\mathcal{E}_{n+1,e,i}$ agrees with~$\leq_\epsilon$. 
Hence Property~\ref{item:ES-correct} holds.

Finally, by Lemma~\ref{l:rayinducedsubgraph} it is clear that for any~${e \in \partial_\epsilon(T_{k-1}) \setminus \{f\}}$ and~${i \in [k]}$, the rays~$\mathcal{E}^{-}_{n+1,e,i}$ form an interval with respect to~$\leq_\epsilon$ on~$\mathcal{E}^-_{n+1}$, since they are each contained in a connected subgraph~$H'_i$ to which the tails of the rest of~$\mathcal{E}^-_{n+1}$ are disjoint. 
Furthermore, by choice of~$Z$ and Lemma~\ref{lem:core-exchange}, it it clear that, since~$\mathcal{P}$ is preserving on~$\mathcal{E}^{-}_n$, for each~${e \in \partial_\epsilon(T_k) \setminus \partial_\epsilon(T_{k-1})}$ and~${i \in [k]}$, the rays~$\mathcal{E}^-_{n+1,e,i}$ also form an interval with respect to~$\leq_\epsilon$ on~$\mathcal{E}^-_{n+1}$. 
Hence, property~\ref{item:ES-interval} holds and therefore~\ref{item:induction3} is satisfied for the next step.

For property~\ref{item:induction4}, we note that every ray in~$\mathcal{E}_{n+1}$ has a tail in some~${H \in F \in \mathcal{F}^*}$ or some pushout~$H'$ of~$H$ in~$\mathcal{F}_k$. 
Note that~${V(H') \subseteq V(H)}$. 
Since there is at least one core $\epsilon$-ray in each~${H \in F \in \mathcal{F}^*}$, and the~$H$ in~$F$ are pairwise disjoint, we can find a family of at least~${|F| - |\mathcal{E}_{n+1}|}$ such rays disjoint from~$\mathcal{E}_{n+1}$. 
However, since
\[
    |F| \geq (\partial_\epsilon(T_k) + 2) \cdot |\mathcal{E}_{n+1}|,
\]
it follows that we can find a suitable family~${|\mathcal{J}_k|}$.

This concludes the induction step.
\hfill 
\qed

\begin{sidewaysfigurepage}
    \centering
    \resizebox{.8\textwidth}{!}{
        
\begin{tikzpicture}[scale=1.1]
    \colorlet{extendercolor}{green!80!black}
    \tikzstyle{ExtenderRay}=[extendercolor,->]

    \newcommand\Qin[1]{
        \draw[blue] (0,0) circle [radius=1];
        \draw[blue] (0:1) arc (70:0:-1);
        \draw[blue] (-5:1) arc (-60:-5:-1);
        \draw[blue] (0,0) node {$ Q_{#1}^n $};
        \begin{scope}
            \clip (0,0) circle [radius=1.8];
            \draw[red] (160:1.75) circle [radius=1];
        \end{scope}
        \draw[red] (160:1.75) node {$ Q_{#1}^n(T_k^{\neg\epsilon}) $};
        \draw[ExtenderRay] (.6,.6) -- (16, .6);
        \draw[ExtenderRay] (.8,.3) -- (16, .3);
        \draw[ExtenderRay,blue] (.75,-.5) -- (16,-.5);
    }
    \begin{scope}
        \Qin{1}
        \draw[blue,->] (1,1) node[above] {\small$ S(f) $} -- ++(-.2,-.2);
    \end{scope}
    \draw[blue] (0, -2) node {$ \vdots $};
    \draw[blue] (0, -3) node {$ \vdots $};
    \begin{scope}[yshift=-5cm]
        \Qin{1}
    \end{scope}

    \draw[red] (1.5, 1.5) rectangle (2.5, -6.5) node[pos=.5] {$ X_n $};
    \draw[magenta] (7, -8.5) node {$ F \in \mathcal{F}^\ast $};

    \newcommand\newHBot{
        \draw[magenta] (0,0) circle [radius=.5];
        \draw[magenta] (0, .5) arc (-10:-90:.5);
        \begin{scope}
            \clip (0,0) circle [radius=1];
            \draw[red] (280:.9) circle [radius=.6];
        \end{scope}
    }

    \newcommand\newHBlob[1]{
        \draw[magenta] (0, 0) node {$ H_{#1} $};
        \newHBot
        \draw[magenta] (-.15, .35) -- ++(0, 11-.15);
        \draw[magenta] (-.35, .20) -- ++(0, 11);

        \begin{scope}[xshift=-.2cm, yshift=11.5cm]
            \draw[red] (0, 0) circle [radius=.5];
            \draw[red] (-1, 0.6) .. controls +(.8,-.8) and +(.8,-.8) .. (-.6,1);
            \draw[red] (0, 0) node {\tiny$ H_{#1}(f^+) $};
            \draw[red] (-.6, .5) node {\small$ \neg\epsilon $};
            \begin{scope}
                \clip (0,0) circle [radius=.5];
                \draw[blue] (110:.5) arc (30:145:-.2);
                \draw[blue] (60:.5) arc (-10:95:-.2);
            \end{scope}
            \draw[blue,->] (40:.45) -- +(up:1.1);
            \draw[blue,->] (100:.45) -- +(up:1);
            \draw[blue,->] (80:.45) -- +(up:1);
        \end{scope}
    }

    \newcommand\newHstraight[1]{
        \draw[magenta] (0, 0) node {$ H_{#1} $};
        \newHBot
        \draw[magenta,->] (-.15, .35) -- ++(0, 12.5);
        \draw[magenta,->] (-.35, .2) -- ++(0, 12.65);
    }

    \begin{scope}[xshift=9cm, yshift=-8.5cm]
        \draw[magenta,<-] (150:.55) -- ++(150:.3) node[left] {\tiny$ S(f) $};
        \newHBlob{1}
    \end{scope}
    \begin{scope}[xshift=11cm, yshift=-8.5cm]
        \newHstraight{2}
    \end{scope}
    \begin{scope}[xshift=13cm, yshift=-8.5cm]
        \newHBlob{3}
    \end{scope}

    \draw[magenta] (14, -8.5) node {$ \dots $};
    \draw[magenta] (14, 3.25) node {$ \dots $};

    \begin{scope}[xshift=15cm, yshift=-8.5cm]
        \newHstraight{m}
    \end{scope}

    \filldraw[draw=orange,fill=white,opacity=.9] (4, 1.5) rectangle (15, -6.5);
    \draw[orange] (5,1) node {linkage};

    \tikzstyle{pprime}=[very thick, black!40!white]
    \tikzstyle{dottedend}=[dotted,postaction={draw,solid,shorten >=.5cm}]
    \draw[pprime] (4,.6) -| (8.65, 1.5);
    \draw[pprime] (4,.3) -| (8.85, 1.5);
    \draw[pprime] (4,-.5) -| (10.85, 1.5);
    \draw[blue, ->] (10.85, 1.5) -- ++(0,2.85);

    \draw[blue, ->] (14.85, 1.5) -- ++(0,2.85);

    \draw[pprime,dottedend] (4,.6-5) -- ++(6, 0);
    \draw[pprime,dottedend] (4,.3-5) -- ++(6.5, 0);
    \draw[pprime,dottedend] (4,-5.5) -- ++(8, 0);
    \draw[pprime,dottedend] (12.65, 1.5) -- ++(0,-1);
    \draw[pprime,dottedend] (12.85, 1.5) -- ++(0,-1.5);
    \draw[pprime,dottedend] (14.85, 1.5) -- ++(0,-2.5);

    \draw[pprime] (8, -2.5) node {\huge$ \mathcal{P}' $};
    \draw[pprime] (5, -2.5) node {\huge$ \vdots $};
    \draw[pprime] (13.75, -4.5) node {\Huge$ \ddots $};
    \draw[pprime] (13.75, 1) node {\Huge$ \cdots $};

    \draw[blue] (14, 4) node {$ \mathcal{E}_{n+1} $};
\end{tikzpicture}
    }
    \caption{Extending the copies when~${n=2k-1}$ is odd.}
    \label{f:new}
\end{sidewaysfigurepage}

\subsection{The grid-like case}
\label{s:gridlike} 

In this section we will give a brief sketch of how the argument differs in the case where the end~$\epsilon$, towards which we may assume our $G$-tribe~$\mathcal{F}$ is concentrated, is grid-like.

In the case where~$\epsilon$ is half-grid-like we showed that the end~$\epsilon$ had a roughly linear structure, in the sense that there is a global partial order~$\leq_{\epsilon}$ which is defined on almost all of the $\epsilon$-rays, namely the core ones, such that every pair of disjoint core rays are comparable, and that this order determines the relative structure of any finite family of disjoint core rays, since it determines the ray graph.

Since, by Corollary~\ref{c:linearend}, $\RG_G(\mathcal{R}_e)$ is a path whenever~${e \in \partial_\epsilon(T_k)}$, there are only two ways that~$\leq_{\epsilon}$ can order~${H^{\downarrow}(\mathcal{R}_e)}$, and, since~${\partial_\epsilon(T_k)}$ is finite, by various pigeon-hole type arguments we can assume that it does so consistently for each~${H \in \bigcup \mathcal{F}_k}$ and each~$\Ecal_{e,i}$. 

We use this fact crucially in part~2 of the construction, where we wish to extend the graphs ${(Q^n_i \colon i \in [k])}$ from inflated copies of~${G(T^{\neg \epsilon}_{k-1})}$ to inflated copies of~${G(T^{\neg \epsilon}_{k})}$ along an edge~${e \in \partial(T_{k-1})}$. 
We wish to do so by constructing a linkage from the extender~$\Ecal_n$ to some layer~${F \in \mathcal{F}_k}$, using the self-similarity of~$G$ to find an inflated copy of~${G[B(e)]}$ which is `rooted' on the rays~${H^{\downarrow}(\mathcal{R}_e)}$ and extending each~$Q^n_i$ by such a subgraph.

However, for this step to work it is necessary that the linkage from~$\Ecal_n$ to~${F^{\downarrow}(\mathcal{R}_e)}$ is such that for each~${i \in [k]}$, there is some~${H \in F}$ such that ray~$E_{e,s,i}$ is linked to~${H^{\downarrow}(R_{e,s})}$ for each~${s \in S(e)}$. 
However, since any transitional linkage we construct between~$\Ecal_n$ and a layer~${F \in \mathcal{F}_n}$ will respect~$\leq_{\epsilon}$, we can use a transition box to `re-route' our linkage such that the above property holds.

In the case where~$\epsilon$ is grid-like we would like to say that the end has a roughly cyclic structure, in the sense that there is a global `partial cyclic order'~$C_\epsilon$, defined again on almost all of the $\epsilon$-rays, which will again determine the relative structure of any finite family of disjoint `core' rays.

As before, since~${\RG_{G}(\mathcal{R}_e)}$ is a path whenever~${e \in \partial_\epsilon(T_n)}$, there are only two ways that~$C_\epsilon$ can order~${H^{\downarrow}(\mathcal{R}_e)}$ (`clockwise' or `anti-clockwise') and so we can use similar arguments to assume that it does so consistently for each~${H \in \bigcup \mathcal{F}_k}$ and each~$\Ecal_{e,i}$, which allows us as before to control the linkages we build. 

To this end, suppose~$\epsilon$ is a grid-like end, and that~$N$ is as in Lemma~\ref{l:gridstructure}, so that the ray graph of any family of at least~${N+2}$ disjoint rays is a cycle. 
We say that an $\epsilon$-ray~$R$ is a \emph{core ray (of~$\epsilon$)} if there is some finite family~${(R_i \colon i \in[n])}$ of~${n \geq N+3}$ disjoint $\epsilon$-rays such that~${R = R_i}$ for some~${i \in [n]}$\footnote{We note that it is possible to show that, if~$\epsilon$ is grid-like, then in fact~${N=3}$.}. 

Every large enough ray graph is a cycle, which has a correct orientation by Lemma~\ref{l:gridstructure}, and we would like to say that this orientation is induced by a global `partial cyclic order' defined on the core rays of~$\epsilon$.

By a similar argument as in Section~\ref{s:core}, one can show the following:

\begin{lemma}
    For every core ray~$R$ of a grid-like end~$\epsilon$ there is a unique sub-end of~$\epsilon$ in~${G - V(R)}$, which is linear (cf.~Definition~\ref{d:linear}). 
\end{lemma}

It follows that if~$R$ and~$R'$ are disjoint core rays then~$\epsilon$ splits into at most two ends in~${G - (V(R) \cup V(R'))}$.

\begin{definition}
    Let~$R$ and~$R'$ be disjoint core rays of~$\epsilon$. 
    We denote by~${\top(\epsilon, R,R')}$ the end of~${G - (V(R) \cup V(R'))}$ containing rays which appear between~$R$ and~$R'$ according to the correct orientation of some ray graph of a family of at least $N+3$ $\epsilon$-rays and by~${\bot(\epsilon, R,R')}$ the end of~${G - (V(R) \cup V(R'))}$ containing rays which appear between~$R'$ and~$R$ in the correct orientation of some ray graph of a family of at least $N+3$ $\epsilon$-rays.
\end{definition}

We will model our global `partial cyclic order' as a ternary relation on the set of core rays of~$\epsilon$. 
That is, a \emph{partial cyclic order} on a set~$X$ is a relation~${C \subset X^3}$ written~${[a,b,c]}$ satisfying the following axioms:
\begin{itemize}
    \item If~${[a,b,c]}$ then~${[b,c,a]}$.
    \item If~${[a,b,c]}$ then not~${[c,b,a]}$.
    \item If~${[a,b,c]}$ and~${[a,c,d]}$ then~${[a,b,d]}$.
\end{itemize}

\begin{lemmadef}
    Let~${\core(\epsilon)}$ denote the set of core rays of~$\epsilon$. 
    We define a partial cyclic order~$C_\epsilon$ on~${\core(\epsilon)}$ as follows:
    \[
        [R,S,T] \text{ if and only if } R,S,T \text{ have disjoint tails } xR,yS,zT \text{ and  } yS \in \top(\epsilon, xR,zT).
    \]
    Then, for any family~${(R_i \colon i \in [n])}$ of~${n \geq N+3}$ disjoint $\epsilon$-rays, the cyclic order induced on~${(R_i \colon i \in [n])}$ by~$C_\epsilon$ agrees with the correct orientation. 
\end{lemmadef}

Again, by a similar argument as in Section~\ref{s:core}, one can show that this relation is in fact a partial cyclic order and that it always agrees with the correction orientation of large enough ray graphs. 
Furthermore, by Lemma~\ref{l:gridstructure}, given two families~$\Rcal$ and~$\Scal$ of at least~${N+3}$ disjoint $\epsilon$-rays, every transitional linkage between~$\Rcal$ and~$\Scal$ \emph{preserves}~$C_\epsilon$, for the obvious definition of preserving.

Given a family of disjoint $\epsilon$-rays ${\mathcal{R} = (R_i \colon i \in [n])}$ with a linear order~$\leq$ on~$\mathcal{R}$, we say that~$\leq$ \emph{agrees} with~$C_\epsilon$ if~${[R_i,R_j,R_k]}$ whenever~${R_i < R_j < R_k}$.

Given a family ${F = ( f_i \colon i \in I )}$ and a linear order~$\leq$ on~$I$, we denote by~${F(\leq)}$ the linear order on~$F$ induced by~$\leq$, i.e.~the order defined by~${f_i F(\leq) f_j}$ if and only if~${i \leq j}$.

As in Section~\ref{s:tribes} we say a thick $G$-tribe~$\mathcal{F}$ \emph{strongly agrees about~${\partial(T_n)}$} if 
\begin{itemize}
    \item it weakly agrees about~$\partial(T_n)$; 
    \item for each~${H \in \bigcup \mathcal{F}}$ every $\epsilon$-ray ${R \subseteq H}$ is in $\core(\epsilon)$; 
    \item for every~${e \in \partial_\epsilon(T_n)}$ there is a linear order~$\leq_{\mathcal{F},e}$ on~${S(e)}$ such that~${H^{\downarrow}(\mathcal{R}_e)(\leq_{\mathcal{F},e})}$ agrees with~$C_\epsilon$ on~${H^{\downarrow}(\mathcal{R}_e)}$ for all~${H \in \bigcup F}$. 
\end{itemize}

Using this definition, the $G$-tribe refinement lemma (Lemma~\ref{lem:refinement-lemma}) can also be shown to hold in the case where~$\epsilon$ is a grid-like-end.

Furthermore, we modify the definition of an extension scheme for a family of disjoint inflated copies of~${G(T^{\neg \epsilon}_n})$. 

\begin{definition}[Extension scheme]
    Let ${\mathcal{Q} = (Q_i \colon i \in [k])}$ be a family of disjoint inflated copies of~${G(S^{\neg \epsilon})}$ and~$\Fcal$ be a $G$-tribe which strongly agrees about~${\partial(S)}$. 
    We call a tuple~${(X,\Ecal)}$ an \emph{extension scheme} for~$\mathcal{Q}$ if the following holds:
    \begin{enumerate}[label=(ES\arabic*)]
        \item $X$ is a bounder for~$\mathcal{Q}$ and~$\Ecal$ is an extender for~$\mathcal{Q}$;
        \item $\Ecal$ is a family of core rays;
        \item the order~$C_\epsilon$ agrees with~${\Ecal_{e,i}^-(\leq_{\mathcal{F},e})}$ for every~${e \in \partial_\epsilon(S)}$; 
        \item the sets~$\mathcal{E}_{s,i}^-$ are intervals of~$C_\epsilon$ on~${\mathcal{E}^-}$ for all~${e \in \partial_\epsilon(S)}$ and~${i \in [k]}$.
    \end{enumerate}
\end{definition}

We can then proceed by induction as before, with the same induction hypotheses. 
For the most part the proof will follow verbatim, apart from one slight technical issue.

Recall that, in the case where~$n$ is even, we use the existence of the family of rays~${\overline{\mathcal{C}}}$ to find a linkage from~$\mathcal{C}$ to~${F^{\downarrow}(\mathcal{R}^-)}$ which is preserving on~$\mathcal{C}$ and similarly, in the case where~$n$ is odd, we do the same for~$\overline{\mathcal{E}_n^-}$. 
In the grid-like case we do not have to be so careful, since every transitional linkage from~$\mathcal{C}$ to~${F^{\downarrow}(\mathcal{R}^-)}$ will preserve~$C_\epsilon$, as long as~${|\mathcal{C}|}$ is large enough.

However, in order to ensure that~${|\mathcal{C}|}$ and~${|\mathcal{E}_n^-|}$ are large enough in each step, we should start by building~${N+3}$ inflated copies of~${G(T^{\neg \epsilon}_0)}$ in the first step, which can be done relatively straightforwardly. 
Indeed, in the case~${n=0}$ most of the argument in the construction is unnecessary, since a large part of the construction is constructing a new copy whilst re-routing the rays~$\mathcal{E}_n$ to avoid this new copy, but~$\mathcal{E}_0$ is empty. 
Therefore, it is enough to choose a layer~${F \in \mathcal{F}_0}$ with~${|F| \geq N+3}$, with say~${H_1,\ldots,H_{N+3} \in F}$ and to take
\[
    Q^1_i := H_i(T^{\neg \epsilon}_k)
\]
for each~${i \in [N+3]}$, and to take~${E^1_{e,s,i} = H^{\downarrow}_i(R_{e,s})}$ for each~${e \in \partial_\epsilon(T_0)}$, ${s \in S(e)}$, and ${i \in [N+3]}$. 
One can then proceed as before, extending the copies in odd steps and adding a new copy in even steps.

\section{Outlook: connections with well-quasi-ordering and better-quasi-ordering}
\label{s:WQO}

Our aim in this section is to sketch what we believe to be the limitations of the techniques of this paper. 
We will often omit or ignore technical details in order to give a simpler account of the relationship of the ideas involved.

Our strategy for proving ubiquity is heavily reliant on well-quasi-ordering results. 
The reason is that they are the only known tool for finding extensive tree-decompositions for broad classes of graphs. 

To more fully understand this, let us recall how well-quasi-ordering was used in the proofs of Lemmas~\ref{lem:finendsisext} and~\ref{lem:boundwidthisext}. 
Lemma~\ref{lem:finendsisext} states that any locally finite connected graph with only finitely many ends, all of them thin, has an extensive tree-decomposition. 
The key idea of the proof was as follows: 
for each end, there is a sequence of separators converging towards that end. 
The graphs between these separators are finite, and so are well-quasi-ordered by the Graph Minor Theorem. 
This well-quasi-ordering guarantees the necessary self-similarity.

Lemma~\ref{lem:boundwidthisext}, where infinitely many ends are allowed but the graph must have finite tree-width, is similar: 
once more, for each end there is a sequence of separators converging towards that end. 
The graphs between these separators are not necessarily finite, but they have bounded tree-width and so they are again well-quasi-ordered. 

Note that the Graph Minor Theorem is not needed for this latter result. 
Instead, the reason it works can be expressed in the following slogan, which will motivate the considerations in the rest of this section:

\begin{quote}
    Trees of wombats are well-quasi-ordered precisely when wombats themselves are better-quasi-ordered.
\end{quote}

Here better-quasi-ordering is a strengthening of well-quasi-ordering, introduced by Nash-Williams in~\cite{N65} essentially in order to make this slogan be true. 
Since graphs of bounded tree-width can be encoded as trees of graphs of bounded size, what is used here is that graphs of bounded size are better-quasi-ordered.

What if we wanted to go a little further, for example by allowing infinite tree-width but requiring that all ends should be thin? 
In that case, all we would know about the graphs between the separators would be that all their ends are thin. Such graphs are essentially trees of finite graphs. So, by the slogan above, to show that such trees are well-quasi-ordered we would need the statement that finite graphs are better-quasi-ordered. 

Indeed, this problem arises even if we restrict our attention to the following natural common strengthening of Theorems~\ref{t:And1} and~\ref{t:And2}:

\begin{conjecture}
    Any locally finite connected graph in which all blocks are finite is $\preceq$-ubiquitous.
\end{conjecture}

In order to attack this conjecture with our current techniques we would need better-quasi-ordering of finite graphs.

Thomas has conjectured~\cite{T89} that countable graphs are well-quasi-ordered with respect to the minor relation. 
If this were true, it could allow us to resolve problems like those discussed above for countable graphs at least, since all the graphs appearing between the separators are countable. 
But this approach does not allow us to avoid the issue of better-quasi-ordering of finite graphs. 
Indeed, since countable trees of finite graphs can be coded as countable graphs, well-quasi-ordering of countable graphs would imply better-quasi-ordering of finite graphs.

Thus until better-quasi-ordering of finite graphs has been established, the best that we can hope for -- using our current techniques -- is to drop the condition of local finiteness from the main results of this paper. 
For countable graphs we hope to show this in the next paper in the series, 
however for graphs of larger cardinalities further issues arise.

\bibliographystyle{abbrv}
\bibliography{main}

\end{document}